\documentclass[a4paper]{amsart}
\pdfoutput=1
\usepackage{mathmode}
\usepackage{bm}
\usepackage{yhmath}
\usepackage{tikz,tikz-cd}
\usepackage[margin=1.3in]{geometry}
\usepackage{graphicx,array,subfigure}
\usepackage{stackengine,scalerel}
\usepackage{amsthm}
\usepackage{enumitem}
 \usepackage{multirow}
\usepackage{caption}
\usepackage{float}
\floatstyle{plain} %
\newfloat{diagram}{tbp}{lod} %
\floatname{diagram}{Diagram} %
\makeatletter
\renewcommand{\thediagram}{\Alph{diagram}} %
\renewcommand{\fnum@diagram}{Diagram~\thediagram.}
\makeatother
\def\parital{\partial}
\def\H{\mathbf{H}}
\def\DELTA{\mathbf{\Delta}}
\def\coef{\mathbf{coef}}
\def\deg{\mathbf{deg}}
\def\sgn{\mathrm{sgn}}
\def\cups{\cup\cdots\cup}
\def\Z{\mathbb{Z}}
\def\abs#1{\left |#1 \right |}

\def\szkh#1{[\![#1]\!]}

\def\R{\mathbb{R}}

\def\O{\mathcal{O}}
\def\slope{\text{slp}}
\def\Q{\mathbb{Q}}

\def\tto{\longrightarrow}

 \DeclareRobustCommand{\longleftmapsto}{\text{\reflectbox{$\longmapsto$}}}
\def\quotient#1#2{%
    \raise1ex\hbox{$#1$}\Big/\lower1ex\hbox{$#2$}%
}
\def\quo#1#2{%
    \raise0.6ex\hbox{$#1$}\big/\lower0.6ex\hbox{$#2$}%
}

\def\F{f^\dagger}

\def\limi{\varprojlim}

\def\Zx{\widehat{\Z}^{\times}}
\def\ab#1{{#1}^{\mathrm{ab}}}
\def\GL{\mathrm{GL}}
\def\SL{\mathrm{SL}}
\def\C{\mathbb{C}}
\def\Ab#1{{#1}^{\mathrm{Ab}}}
\def\slope{\text{slp}}
\def\slp{\text{slp}^\ast}
\def\Zxpm{\Zx/\{\pm1\}}
\def\tensor{\widehat{\mathbb{Z}}{\otimes}_{\mathbb{Z}}}

\tikzset{%
  symbol/.style={
    draw=none,
    every to/.append style={
      edge node={node [sloped, allow upside down, auto=false]{$#1$}}
    },
  },
}

\def\fdf#1{f^{\mathrm{Df}(\mathbf{#1})}}

\stackMath
\newcommand\reallywidehat[1]{%
\savestack{\tmpbox}{\stretchto{%
  \scaleto{%
    \scalerel*[\widthof{\ensuremath{#1}}]{\kern.1pt\mathchar"0362\kern.1pt}%
    {\rule{0ex}{\textheight}}
  }{\textheight}%
}{2.4ex}}%
\stackon[-6.9pt]{#1}{\tmpbox}%
}

\date{\today}

\author{Xiaoyu Xu}
\address{Beijing International Center for Mathematical Research\\Peking University\\ Beijing 100871, China P.R. }
\email{xuxiaoyu@stu.pku.edu.cn}

\title[Regularity of profinite isomorphisms and the A-polynomial]{On regularity of profinite isomorphisms between cusped hyperbolic 3-manifolds and the A-polynomial}

\begin{document}
\begin{sloppypar}
\maketitle
\begin{abstract}
We prove that any isomorphism between the profinite completions of the fundamental groups of two cusped finite-volume hyperbolic 3-manifolds is regular and peripheral regular. As an application, we show that the $A$-polynomial of prime knots in $S^3$ is a profinite invariant, up to possible mirror image. 
\end{abstract}
\setcounter{tocdepth}{1}
\tableofcontents 
\section{Introduction}
Profinite rigidity questions to which extent a finitely generated, residually finite group is determined by its finite quotient groups. These data are equivalently encoded in an algebraic  construction named profinite completion according to \cite{DFPR82}.
\begin{definition}\label{DEF: profinite completion}
Let $G$ be a group, the {\em profinite completion} of $G$ is defined as 
\begin{equation*}
\widehat{G}=\limi\,\quo{G}{N}
\end{equation*}
where $N$ ranges over all finite-index normal subgroups of $G$.
\end{definition}

The question becomes more accessible when we restrict ourselves to 3-manifold groups. Since 3-manifolds are largely determined by their fundamental groups, it is natural to ask which topological or geometrical properties of a compact 3-manifold $M$ is determined by the isomorphism type of $\widehat{\pi_1M}$.

\begin{convention*}
In this article, all 3-manifolds are orientable unless otherwise stated. 
\end{convention*}

\setcounter{footnote}{1}

Proven in a series of works by Wilton--Zalesskii \cite{WZ17, WZ17b, WZ19}, as well as the works of Wilkes \cite{WilkesJSJ,Wil19} for the bounded case, the profinite completion of the fundamental group determines whether a compact 3-manifold $M$ is geometric (in the sense of Thurston); it determines the geometry type in the geometric case, and determines the prime decomposition and the JSJ-decomposition in the non-geometric case. Many examples of 3-manifolds are proven to be profinitely rigid, that is, uniquely distinguished among compact 3-manifolds by the profinite completion of its fundamental group, see \cite{BMRS20, Cw24, Wil17, Wil18, Xu24b}. However, the general question for profinite rigidity in 3-manifolds is still widely open. In \cite{Xu24a}, the author proved that the the profinite completion of the fundamental group determines the homeomorphism type of a compact 3-manifold $M$ with empty or toral boundary up to finitely many possibilities. To determine whether such finite ambiguity really exists, the remaining challenging part lies in the case of hyperbolic 3-manifolds, including closed and cusped manifolds. 

\subsection{Regularity and peripheral regularity}
The first part of this paper characterizes a special feature of isomorphisms between the profinite completion of the fundamental groups of hyperbolic 3-manifolds. 
The notion of regularity for a profinite isomorphism was first introduced by Boileau--Friedl \cite{BF19}.
\begin{definition}\label{introdef: Zx regular}
Let $G$ and $H$ be two (finitely generated) groups. An isomorphism $f:\widehat{G}\to \widehat{H}$ is called {\em regular} if the induced isomorphism $\widehat{\ab{G}}\to \widehat{\ab{H}}$ is the profinite completion of an isomorphism $\ab{G}\to \ab{H}$, where `$\mathrm{ab}$' denotes the abelianization.  
\end{definition}

In the study of profinite rigidity among 3-manifold groups, peripheral subgroups in 3-manifold groups plays an important role. A related notion named peripheral regularity was introduced by the author in \cite{Xu24a,Xu24b}. 
\begin{definition}\label{introdef: peripheral Zx regular}
Let $M$ and $N$ be compact 3-manifolds with incompressible toral boundaries, and let $f:\widehat{\pi_1M}\to \widehat{\pi_1N}$ be an isomorphism. 
\begin{enumerate}[leftmargin=*]
\item\label{(1)} We say that $f$ {\em respects the peripheral structure} if there is a one-to-one correspondence between the boundary components of $M$ and $N$, denoted by $\partial_iM\longleftrightarrow \partial_iN$, such that  $f(\overline{\pi_1\partial_iM})=g_i^{-1}\cdot \overline{\pi_1\partial_iN} \cdot g_i$  for some $g_i\in \widehat{\pi_1N}$, where $\pi_1\partial_iM\le \pi_1M$ and $\pi_1\partial_iN\le \pi_1N$ are conjugacy representatives of the peripheral subgroups.
\item We say that $f$ is {\em peripheral regular} if $f$ respects the peripheral structure (so we follow the notation in (\ref{(1)})), and for each $i$, the isomorphism $C_{g_i}\circ f|_{\overline{\pi_1\partial_iM}}:\widehat{\pi_1\partial_iM}\cong \overline{\pi_1\partial_iM}\tto \overline{\pi_1\partial_iN}\cong\widehat{\pi_1\partial_iN}$ is regular, where $C_{g_i}(x)=g_ixg_i^{-1}$ denotes the conjugation.  
\end{enumerate}
%
\end{definition}

The main result of this article is the following theorem. 

\begin{theorem}\label{mainthm: regular}
Suppose $M$ and $N$ are cusped finite-volume hyperbolic 3-manifolds and $f:\widehat{\pi_1M}\to \widehat{\pi_1N}$ is an isomorphism. Then $f$ is regular and peripheral regular.
\end{theorem}

We remark that in general, a profinite isomorphism between 3-manifold groups may not be regular. This is demonstrated by the example of Hempel pairs \cite{Hem14}, as well as easily constructable examples for the 3-torus $T^3$, or more generally, the product of a compact surface with $S^1$. Thus, \autoref{mainthm: regular} is particularly special for hyperbolic 3-manifolds. 

Weaker versions of \autoref{mainthm: regular} have been proven. Liu \cite{Liu23} proved that any isomorphism between the profinite completion of the fundamental groups of finite-volume hyperbolic 3-manifolds (including closed and cusped) is $\Zx$-regular, that is, the induced isomorphism $\widehat{H_1(M;\Z)}\to \widehat{H_1(N;\Z)}$ is the profinite completion of an isomorphism $H_1(M;\Z)\to H_1(N;\Z)$ composed with a scalar multiplication by some element $\lambda\in \Zx$, see \autoref{THM: Zx regular}. The author  proved in \cite{Xu24a} that any isomorphism between the profinite completion of the fundamental groups of cusped hyperbolic 3-manifolds is peripheral $\Zx$-regular, that is, replacing `regular' in \autoref{introdef: peripheral Zx regular} with `$\Zx$-regular', see \autoref{THM: peripheral Zx regular}. 

The proof of \autoref{mainthm: regular} is an improvement of these two weaker results. The key technique being utilized in this paper is the relative cohomology theory for profinite groups introduced by Wilkes \cite{Wil19}. In fact, for an isomorphism $f$ between the profinite completion of the fundamental groups of cusped hyperbolic 3-manifolds, we can define a profinite mapping degree $\deg(f)\in \Zx$ witnessed in the third relative homology groups. The coefficient $\lambda\in\Zx$ that appears in the $\Zx$-regularity statement is denoted by $\coef(f)$. We establish two relations between $\coef(f)$ and $\deg(f)$. On one hand, through passing onto the peripheral subgroups, we prove that $\coef(f)^2=\pm\deg(f)$. On the other hand, using the  Dehn filling technique introduced in \cite{Xu24b}, we can also pass these invariants to a profinite isomorphism between closed hyperbolic 3-manifolds and prove that $\coef(f)^3=\pm \deg(f)$ based on a formula of Liu \cite[Lemma 5.3]{Liu25}. These two methods demonstrate a dimension shift, which in turn implies that $\coef(f)=\pm1$. 

\begin{remark}
When $M$ and $N$ are closed hyperbolic 3-manifolds and $f:\widehat{\pi_1M}\to \widehat{\pi_1N}$ is an isomorphism, Liu \cite[Lemma 7.1]{Liu25} showed that $\coef(f)^2=1$. This is currently the closest result towards regularity in the closed case. Our proof for the cusped case relies on the boundary structure, and it cannot be extended directly to the closed case. 
\end{remark}

\subsection{The $A$-polynomial}

Knot complements in $S^3$ provide a delightful testing ground for profinite rigidity or profinite properties in 3-manifold groups. For instance, Ueki \cite{Ueki} proved that knots in $S^3$ with profinitely isomorphic knot groups have identical Alexander polynomials. 
Torus knot complements and prime graph knot complements are profinitely rigid among all knot complements in $S^3$, according to Boileau--Friedl \cite{BF19} and Wilkes \cite{WilkesKnot} respectively.  
In \cite[Theorem C]{Xu24b}, the author proved that the complement of hyperbolic knots or hyperbolic-type satellite knots in $S^3$ are distinguished from compact orientable 3-manifolds by the profinite completion of their fundamental groups. In addition, in \cite{Xu24b}, the author also showed the profinite rigidity of many hyperbolic knot complements, such as the complement of full twist knots, including a new proof for the figure-eight knot firstly proven by Bridson--Reid \cite{BR20}. 

The $A$-polynomial of a knot $K$ is a two-variable polynomial $A_K(M,L)$  defined by Cooper, Culler, Gillet, Long, and Shalen in \cite{CCGLS94}. This is a knot invariant related to the $\SL(2,\C)$-representations of 3-manifold groups. An enhanced version of the $A$-polynomial is also defined and used in many literatures  containing slightly more information than the $A$-polynomial, which we denote as $\widetilde{A}_K(M,L)$, see \autoref{subsec: A-poly} for complete definition. 
Briefly speaking, when defining the enhanced version, the component in the character variety of the knot complement given by reducible representations is excluded.
The relation between these two versions is simple. In fact,
$
A_K(M,L)=\mathrm{Red}((L-1)\widetilde{A}_K(M,L))
$, 
where `$\mathrm{Red}$' means reducing repeating factors. In particular, the enhanced $A$-polynomial completely determines the $A$-polynomial.

For a knot $K$ in $S^3$, we denote $X(K)=S^3-\text{int}(N(K))$ as the exterior of $K$. An application of \autoref{mainthm: regular} is the following theorem.

\begin{theorem}\label{mainthm: A-polynomial general}
Suppose $J$ and $K$ are knots in $S^3$, and $f: \widehat{\pi_1(X(J))}\to \widehat{\pi_1(X(K))}$ is an isomorphism. Suppose either $f$ respects the peripheral structure\footnote{assuming $J$ and $K$ are non-trivial.}, or one of $J$ and $K$ is a prime knot. Then, up to possibly replacing $K$ with its mirror image, the enhanced $A$-polynomials of $J$ and $K$ are identical: $\widetilde{A}_J(M,L)\doteq \widetilde{A}_K(M,L)$, and consequently their $A$ polynomials are also identical.
\end{theorem}

The assumption that either the isomorphism respects the peripheral structure, or one of the knots is a prime knot is analogous to the condition in distinguishing knots in $S^3$ through the fundmental group of their exteriors. In fact, given knots $J$ and $K$ in $S^3$, suppose $f: \pi_1(X(J))\to \pi_1(X(K))$ is an isomorphism; then, $J$ and $K$ are equivalent (including mirror image) if either $f$ respects the peripheral structure \cite{Wal68} or one of $J$ and $K$ is prime \cite{GL89}. Moreover, this assumption is necessary in \autoref{mainthm: A-polynomial general}. For non-prime knots $J=K_1\#K_2$ and $K=K_1\#\overline{K_2}$, where $\overline{K_2}$ denotes the mirror image of $K_2$, we always have $\pi_1(X(J))\cong \pi_1(X(K))$. However, in this case $A_J(M,L)$ and $A_K(M,L)$ may not coincide, see \cite{CL96} for the example of reef knot and granny knot.

The core behind \autoref{mainthm: A-polynomial general} can be stated in the fashion of character varieties. For a compact connected manifold $Y$, let $\mathscr X_n(Y)$ denote the $\SL(n,\C)$-character variety of $\pi_1Y$, where $n\in \mathbb{N}$. For a compact disconnected manifold $Z$ consisting of components $Z_1,\cdots, Z_k$, we denote $\mathscr X_n(Z)=\mathscr{X}_n(Z_1)\times \cdots \times \mathscr{X}_n (Z_k)$. 
\begin{theorem}\label{mainthm: character variety}
Suppose $M$ and $N$ are cusped finite-volume hyperbolic 3-manifolds, and $\widehat{\pi_1M}\cong \widehat{\pi_1N}$. Then, there exists a homeomorphism $\Psi:\partial M\to \partial N$ such that the images of the restriction maps $\mathscr{X}_n(M)\to \mathscr{X}_n(\partial M)$ and $\mathscr{X}_n(N)\to \mathscr{X}_n(\partial N)$ are identical through $\Psi$.  In addition, $\Psi$ is either orientation-preserving on all components or orientation-reversing on all components.
\end{theorem}

\subsection*{Recent related works}
Upon submitting this paper, the author found the very recent article \cite{CWYao} by Cheetham-West and Yao. Assuming  the validity of our \autoref{mainthm: regular}, they  have independently proven  \autoref{mainthm: character variety} for one-cusped hyperbolic manifolds in the case $n=2$. Our proof for \autoref{mainthm: character variety} is based on a different method.

\subsection{Structure of the paper}
\autoref{sec:2} is a brief collection of preliminary materials that might be familiar to experts. In \autoref{sec:3}, we review the existing results on $\Zx$-regularity \cite{Liu23} and peripheral $\Zx$-regularity \cite{Xu24a}; and in \autoref{sec:4}, we define the homology coefficient $\coef(f)\in \Zx/\{\pm1\}$ for a profinite isomorphism $f$ between hyperbolic 3-manifolds.

In \autoref{sec:5}, we expand on the relative cohomology theory for profinite groups established by Wilkes \cite{Wil19}; and in \autoref{sec:6}, we define the profinite mapping degree $\deg(f)\in \Zx$ for a profinite isomorphism $f$ between oriented hyperbolic 3-manifolds.

\autoref{sec:7} and \autoref{SEC: Dehn filling} introduce two relations between the homology coefficient and the profinite mapping degree; namely, we prove $\coef(f)^2=\pm \deg(f)$ in \autoref{sec:7} and $\coef(f)^3=\pm \deg(f)$ in \autoref{SEC: Dehn filling}. These yield the proof for \autoref{mainthm: regular} as restated in \autoref{COR: restate main}.

As a prepartion for \autoref{mainthm: A-polynomial general}, we extend the result of peripheral regularity to profinite isomorphisms within a  certain class of mixed manifolds in \autoref{sec:9}. Finally, we approach \autoref{mainthm: A-polynomial general} in \autoref{sec:10}, with its two cases proved separately as  \autoref{THM: A-polynomial peripheral} and \autoref{THM: A prime}. \autoref{mainthm: character variety} is proved as \autoref{COR: rep variety}.


\setcounter{footnote}{2}
\section{Preliminaries on profinite groups}\label{sec:2}
\subsection{Profinite groups}
A {\em profinite group} is an inverse limit of finite groups indexed over a directed partially ordered set, equipped with the subspace topology inherited from the product topology. A direct characterisation for profinite groups is shown by the following proposition.
\begin{proposition}[{\cite[Theorem 2.1.3]{RZ10}}]\label{PROP: Top}
A topological group $\Pi$ is a profinite group if and only if $\Pi$ is compact, Hausdorff, and totally disconnected.
\end{proposition}

The profinite completion of an abstract group $G$, defined as $\widehat{G}=\limi_{N\lhd_{f.i.}G} G/N$ in \autoref{DEF: profinite completion}, is thus a profinite group. 

\subsection{Residual finiteness}
There is a canonical homomorphism.
$$\begin{aligned}
\iota_G: \; G&\longrightarrow \widehat{G}\\
g&\longmapsto (gN)_{N\lhd_{f.i.} G}
\end{aligned}
$$
It is easy to see that $\iota_G$ is injective if and only if $G$ is residually finite. Thus, in the following context, when $G$ is residually finite, we always identify $G$ as its image in $\widehat{G}$, and elements in $G$ are also viewed as elements in $\widehat{G}$. 

Hempel's theorem for residual finiteness \cite{Hem87}, together with the virtual Haken theorem \cite{Ago13}, yields the following result.

\begin{proposition}\label{PROP: RF}
The fundamental group of any compact 3-manifold is residually finite.
\end{proposition}

\subsection{The completion fuctor}
The profinite completion is functorial. Any homomorphism of abstract group $f: G_1\to G_2$ induces a continuous homomorphism  $\widehat{f}:\widehat{G_1}\to \widehat{G_2}$ so that the following diagram commutes.
\begin{equation*}
\begin{tikzcd}
G_1 \arrow[r, "f"] \arrow[d,"\iota_{G_1}"'] & G_2 \arrow[d,"\iota_{G_2}"] \\
\widehat{G_1} \arrow[r, "\widehat{f}"]   & \widehat{G_2}           
\end{tikzcd}
\end{equation*}

For any two finitely generated groups $G_1$ and $G_2$, the theorem of Nikolov-Segal \cite{NS03} implies that any isomorphism between $\widehat{G_1}$ and $\widehat{G_2}$ as abstract groups is indeed an isomorphism as profinte groups, that is, additionally a homeomorphism. Thus, we shall not emphasize on the continuity of isomorphisms between profinite completions of finitely generated groups in the following context.

The completion functor is right exact.
\begin{proposition}[{\cite[Proposition 3.2.5]{RZ10}}]\label{PROP: Completion Exact}
Given a short exact sequence of abstract groups
\begin{equation*}
\centering
\begin{tikzcd}
1 \arrow[r] & K \arrow[r, "\varphi"] & G \arrow[r, "\psi"] & L \arrow[r] & 1,
\end{tikzcd}
\end{equation*}
there is an exact sequence of profinite groups
\begin{equation*}
\centering
\begin{tikzcd}
\widehat K \arrow[r, "\widehat \varphi"] & \widehat G \arrow[r, "\widehat \psi"] & \widehat L \arrow[r] & 1.
\end{tikzcd}
\end{equation*}
\end{proposition}

However, the profinite completion may not preserve injectiveness.

\begin{proposition}[{\cite[Lemma 3.2.6]{RZ10}}]\label{PROP: Left exact}
Suppose $f:H\to G$ is an injective homomorphism of abstract groups. Then $\widehat{f}:\widehat{H}\to \widehat{G}$ is injective if and only if \textit{the profinite topology on $G$ induces the full profinite topology on $H$}, that is, for any finite-index subgroup $M\le H$, there exists a finite-index subgroup $N\le G$ such that $f^{-1}(N)\subseteq M$.
\end{proposition}


For any subgroup $H\le G$, we denote $\overline{H}$ as the closure of $\iota_G(H)$ in $\widehat{G}$. Note that $H$ does not always inject into $\overline{H}$, for instance, when $G$ is not residually finite. In addition, the surjective homomorphism $\widehat{H}\to \overline{H}$ is an isomorphism if and only if the condition in \autoref{PROP: Left exact} holds.

\subsection{Lattice of finite-index subgroups}
For finitely generated groups $G_1$ and $G_2$, it is shown by \cite{DFPR82} that $\widehat{G_1}\cong \widehat{G_2}$ if and only if $G_1$ and $G_2$ have the same collection of finite quotient groups (up to isomorphism). Indeed, $G_1$ and $G_2$ actually have isomorphic lattices of finite-index subgroups as shown by the following proposition.

\begin{proposition}[{\cite[Proposition 3.2.2]{RZ10}}]\label{THM: correspondence of subgroup}
For an abstract group $G$, there is an isomorphism between the lattice of finite-index subgroups of $G$ and the lattice of open subgroups in $\widehat{G}$ given as follows. 
\begin{equation*}
\begin{tikzcd}[row sep=0.15em]
\left\{\text{Finite-index subgroups of }G\right\} \arrow[leftrightarrow,"1:1"]{r} & \left\{\text{Open subgroups of }\widehat{G}\right\} \\
H  \arrow[r, maps to] &\overline{H}                            \cong \widehat{H}    \\
\iota^{-1}_G(U) & U \arrow[l, mapsto]
\end{tikzcd}
\end{equation*}
%
%
%
In addition, this correspondence sends normal subgroups to normal subgroups.
\end{proposition}
\begin{definition}\label{DEF: corresponding}
Given an isomorphism $f:\widehat{G_1}\to \widehat{G_2}$, we say that $H_1\le_{f.i.} G_1$ and $H_2\le_{f.i.} G_2$ is an {\em $f$--corresponding pair of finite-index subgroups} if $f(\overline{H_1})=\overline{H_2}$. 
\end{definition}

By \autoref{THM: correspondence of subgroup}, the $f$--correspondence yields an isomorphism between the lattices of finite-index subgroups of $G_1$ and $G_2$. 


Finite-index subgroups in a 3-manifold group $\pi_1M$ bijectively correspond to finite coverings of $M$ (marked with basepoints). We also use \autoref{DEF: corresponding} for the case of 3-manifold groups. 
\begin{definition}\label{DEF: corresponding cover}
Let $M$ and $N$ be compact 3-manifolds, and let $f:\widehat{\pi_1M}\to \widehat{\pi_1N}$ be an isomorphism. We say that a pair of finite covers $M'$ and $N'$ of $M$ and $N$ is an {\em $f$--corresponding pair of finite cover} if $f(\widehat{\pi_1M'})=\widehat{\pi_1N'}$. In this case, we usually denote $f':\widehat{\pi_1M'}\to \widehat{\pi_1N'}$ as the restriction of $f$ on $\widehat{\pi_1M'}$. 
\end{definition}

\subsection{The profinite integers}
The profinite completion of $\Z$, denoted as $\widehat{\Z}$, is called the ring of {\em profinite integers}. By definition,
$$
\widehat{\Z}=\limi \Z/n\Z
$$ 
where $n\in\mathbb{N}$ is partially ordered by divisiblity. 
Two more equivalent definitions for $\widehat{\Z}$ are also useful. First, we can take a cofinal system $n!\in \mathbb{N}$, so that we can express $\widehat{\Z}$ as an inverse limit over a sequence
$$
\widehat{\Z}=\limi_{n\to \infty} \Z/n!.
$$
Second, we can decompose $\Z/n\Z$ into its prime components, and find $$
\widehat{\Z}=\prod_{p\text{: prime}} \Z_p.$$ 
In particular, this shows that $\widehat{\Z}$ is torsion-free.

The group of multiplicatively invertible elements in $\Z$ is denoted by $\Zx$. Indeed, 
$$
\Zx=\limi (\Z/n\Z)^\times
$$
is a profinite group. 

\begin{lemma}\label{LEM: Zx+-1}
Let $\lambda,\mu\in \Zx$ and $m,n\in \Z$, where $n\neq 0$. Suppose $m\lambda=n\mu$. Then $\lambda=\pm \mu$ and $m=\pm n$.
\end{lemma}
\begin{proof}
$m= n(\mu \lambda^{-1})$, so $m\equiv 0\pmod{n}$. Since $m\in \Z$, there exists $a\in \Z$ such that $m=na$. Then,  $na=n(\mu \lambda^{-1})$; in other words, $n(a-\mu\lambda^{-1})=0$. Since $\widehat{\Z}$ is torsion-free, it follows that $a-\mu\lambda^{-1}=0$. In particular, $a=\mu\lambda^{-1}\in \Zx$. Thus, $a=\pm 1$, for otherwise, $a\notin (\Z/a)^\times$. Therefore, $\mu=\pm\lambda$ and $m=\pm n$. 
\end{proof}

\section{$\Zx$-regularity and peripheral $\Zx$-regularity}\label{sec:3}

\subsection{$\Zx$-regularity}

For a discrete group $G$, let $\ab{G}=G/[G,G]$ denote its abelianization; and for a profinite group $\Pi$, let $\Ab{\Pi}=\Pi/\overline{[\Pi,\Pi]}$ denote its profinite abelianization. When $G$ is finitely generated, the profinite abelianization of $\widehat{G}$ is related to the abelianization of $G$ by the following proposition.

\begin{proposition}[{\cite[Proposition 2.11]{Xu24a}}]\label{PROP: abelianization}
Let $G$ be a finitely generated group, then there are natural isomorphisms $\tensor \ab{G} \cong \widehat{\ab{G}}\cong \Ab{\widehat{G}}$ such that the following diagram commutes.
\begin{equation*}
\begin{tikzcd}
G \arrow[r, two heads] \arrow[d, "\iota_G"'] & \ab{G} \arrow[r, hook] & \tensor \ab{G} \arrow[d, "\cong"]    \\
\widehat{G} \arrow[r, two heads]             & \Ab{\widehat{G}}       & \widehat{\ab{G}} \arrow[l, "\cong"']
\end{tikzcd}
\end{equation*}
\end{proposition}

\begin{corollary}\label{COR: abelianization}
Let $G$ and $H$ be finitely generated groups, and let $f:\widehat{G}\to \widehat{H}$ be a continuous homomorphism. Then $f$ induces a homomorphism of $\widehat{\Z}$-modules
$$
f_\ast: \;\tensor \ab{G}\cong \Ab{\widehat{G}}\tto \Ab{\widehat{H}}\cong \tensor \ab{H}.
$$
\end{corollary}

\begin{corollary}[{\cite[Remark 3.2]{Reid:2018}}]\label{COR: b1}
Suppose $G$ and $H$ are finitely generated groups and $\widehat{G}\cong \widehat{H}$. Then $\ab{G}\cong \ab{H}$, and in particular, $b_1(G)=b_1(H)$. 
\end{corollary}

Liu \cite{Liu23} first introduced and proved the notion of $\Zx$-regularity, an intermediate towards regularity, in profinite isomorphisms between finite-volume hyperbolic 3-manifolds, including both closed and cusped cases.

\begin{definition}\label{DEF: Zx-regular}
Let $G$ and $H$ be finitely generated groups. Suppose $f:\widehat{G}\to\widehat{H}$ is a continuous homomorphism. We say that $f$ is {\em $\Zx$-regular} if the induced map
$$
f_\ast: \;\tensor \ab{G} \cong \Ab{\widehat{G}}\tto \Ab{\widehat{H}}\cong \tensor \ab{H}
$$
decomposes as $\lambda\otimes \phi$, where $\lambda\in \Zx$ and $\phi: \ab{G}\to \ab{H}$ is a homomorphism of abelian groups.  To specify the coefficient in $\Zx$, we say that in this case $f$ is {\em $\lambda$-regular}. 
\end{definition}

Comparing the definitions, $f$ is regular (in the sense of \autoref{introdef: Zx regular}) is equivalent to say that $f$ is $1$-regular (in the sense of \autoref{DEF: Zx-regular}).

For an abstract abelian group $A$, let $A_{\mathrm{tor}}$ denote its torsion subgroup, and let $A_{\mathrm{free}}=A/A_{\mathrm{tor}}$. According to \cite[Lemma 2.9]{Xu24a}, for a finitely generated abelian group $A$, $(\widehat{A})_{\mathrm{tor}}\cong \tensor A_{\mathrm{tor}}\cong A_{\mathrm{tor}}$ within the canonical isomorphism in  \autoref{PROP: abelianization}. Thus, the homomorphism $f_\ast: \tensor \ab{G}\to \tensor \ab{H}$ in \autoref{COR: abelianization} descends to the free part
$$
{(f_\ast)}_{\mathrm{free}}:\; \tensor \ab{G}_{\mathrm{free}}\tto \tensor \ab{H}_{\mathrm{free}}.
$$

We point out that the $\Zx$-regularity of $f$ only depends on the free part of the abelianizations.
\begin{lemma}\label{LEM: free part}
For finitely generated groups $G$ and $H$, a continuous homomorphism $f:\widehat{G}\to\widehat{H}$ is $\lambda$-regular if and only if the induced map
$$
{(f_\ast)}_{\mathrm{free}}:\; \tensor \ab{G}_{\mathrm{free}}\tto \tensor \ab{H}_{\mathrm{free}}
$$
splits as $\lambda\otimes \phi_{\mathrm{free}}$ for some homomorphism $\phi_{\mathrm{free}}: \ab{G}_{\mathrm{free}}\to \ab{H}_{\mathrm{free}}$.
\end{lemma}
\begin{proof}
The ``only if'' part is clear. In fact, if $f_\ast: \tensor \ab{G} \to \tensor \ab{H}$ splits as $\lambda \otimes \phi$, then $\phi:\ab{G}\to \ab{H}$ naturally descends to $\phi_{\mathrm{free}}: \ab{G}_{\mathrm{free}}\to \ab{H}_{\mathrm{free}}$, and $(f_\ast)_{\mathrm{free}}= \lambda\otimes \phi_{\mathrm{free}}$.

We now prove the ``if'' part. Suppose $(f_\ast)_{\mathrm{free}}=\lambda\otimes \phi_{\mathrm{free}}$. Then there is the following commutative diagram 
\begin{equation*}
\begin{tikzcd}
                                                                            & \tensor \ab{G} \arrow[dd, "p", two heads] \arrow[rrr, "\lambda^{-1}\circ f_\ast"]                                        &  &                                         & \tensor \ab{H} \arrow[dd, "q", two heads] \\
\ab{G} \arrow[dd, two heads] \arrow[ru, hook]                               &                                                                                                                          &  &                                         &                                           \\
                                                                            & \tensor \ab{G}_{\mathrm{free}} \arrow[rrr, "\lambda^{-1}\circ (f_{\ast})_{\mathrm{free}}=1\otimes \phi_{\mathrm{free}}"] &  &                                         & \tensor \ab{H}_{\mathrm{free}}            \\
\ab{G}_{\mathrm{free}} \arrow[ru, hook] \arrow[rrr, "\phi_{\mathrm{free}}"] &                                                                                                                          &  & \ab{H}_{\mathrm{free}} \arrow[ru, hook] &                                          
\end{tikzcd}
\end{equation*}
where $\lambda^{-1}$ denotes the scalar multiplication by $\lambda^{-1}$ in $\tensor \ab{H}$ or $\tensor \ab{H}_{\mathrm{free}}$.

Consequently, $(\lambda^{-1}\circ f_\ast)(\ab{G})\subseteq q^{-1}(\ab{H}_{\mathrm{free}})$. $\ab{H}$ injects into $\tensor \ab{H}$ as an abelian subgroup. We claim that $q^{-1}(\ab{H}_{\mathrm{free}})= \ab{H}$. On one hand, it is clear that $q(\ab{H})= \ab{H}_{\mathrm{free}}$. On the other hand, since $\widehat{\Z}$ is torsion-free, $\widehat{\Z}$ is a flat $\Z$-module, and  $\ker(q)=\tensor \ab{H}_{\mathrm{tor}}=\ab{H}_{\mathrm{tor}}$. Thus, $\ker(q)\subseteq \ab{H}$. Therefore, $q^{-1}(\ab{H}_{\mathrm{free}})= \ab{H}$, so $(\lambda^{-1}\circ f_\ast)(\ab{G})\subseteq \ab{H}$. We rename this map as $\phi=\lambda^{-1}\circ f_\ast : \ab{G}\to \ab{H}$, which is a homomorphism of abelian groups. Then, $f_\ast=\lambda\otimes\phi$ and $f$ is $\lambda$-regular. 
\end{proof}

The next proposition discusses the uniqueness of the coefficient $\lambda$ in  \autoref{DEF: Zx-regular}.
\begin{proposition}\label{RMK: betti number 0}
Suppose $G$ and $H$ are finitely generated groups and $f:\widehat{G}\to \widehat{H}$ is a $\Zx$-regular isomorphism.
\begin{enumerate}[leftmargin=*]
\item If $b_1(G)=b_1(H)>0$, then the coefficient $\lambda\in \Zx$ such that $f$ is $\lambda$-regular is unique up to a $\pm$-sign.
\item If $b_1(G)=b_1(H)=0$, then the coefficient $\lambda$ can be any element in $\Zx$.
\end{enumerate}
\end{proposition}
\begin{proof}
(1) Suppose $b_1(G)=b_1(H)=n$. Let $(u_1,\cdots,u_n)$ be a free basis of $\ab{G}_{\mathrm{free}}$ over $\Z$, and let $(v_1,\cdots,v_n)$ be a free basis of $\ab{H}_{\mathrm{free}}$ over $\Z$. Then $(f_\ast)_{\mathrm{free}}$ is represented by a matrix $P\in \mathrm{GL}(n,\widehat{\Z})$, such that $$f\begin{pmatrix} u_1  &  \cdots & u_n \end{pmatrix}= \begin{pmatrix} v_1  &  \cdots & v_n \end{pmatrix} P.$$ Suppose $f$ is $\lambda$-regular for $\lambda\in\Zx$. Then, by \autoref{LEM: free part}, $P=\lambda\cdot A$ for some $A\in \mathrm{GL}(n,\Z)$. If $f$ is also $\mu$-regular for $\mu\in \Zx$, then $P=\mu \cdot B$ for some $B\in \mathrm{GL}(n,\Z)$. There is always a non-zero entry in $B$, so by \autoref{LEM: Zx+-1}, $\lambda\cdot A=\mu\cdot B$ implies $\lambda=\pm \mu$.

(2) When $b_1(G)=b_1(H)=0$, both $\ab{G}_{\mathrm{free}}$ and $\ab{H}_{\mathrm{free}}$ are trivial groups. Let $\phi_{\mathrm{free}}: \ab{G}_{\mathrm{free}}\to \ab{H}_{\mathrm{free}}$ be the trivial homomorphism. Then, ${(f_\ast)}_{\mathrm{free}}=\lambda\otimes \phi_{\mathrm{free}}$ for any $\lambda\in\Zx$. Thus, by   \autoref{LEM: free part}, $f$ is $\lambda$-regular for any $\lambda\in\Zx$.
\end{proof}

For profinite isomorphisms between hyperbolic 3-manifolds, Liu \cite{Liu23} proves the following result.

\begin{theorem}[{\cite[Theorem 1.2]{Liu23}}]\label{THM: Zx regular}
Suppose $M$ and $N$ are finite-volume hyperbolic 3-manifolds (including both closed and cusped manifolds), and $f:\widehat{\pi_1M}\to \widehat{\pi_1N}$ is an isomorphism. Then $f$ is $\Zx$-regular.
\end{theorem}

\subsection{Peripheral $\Zx$-regularity}

Let $M$ be a finite-volume hyperbolic 3-manifold with $n$ cusps. By a slight abuse of notation, we view $M$ as a compact 3-manifold with boundary consisting of $n$ tori, denoted as $\partial_1M,\cdots, \partial_nM$. The {\em peripheral subgroups} of $\pi_1M$, usually denoted by $P_1,\cdots, P_n$ in this paper, are conjugacy representatives of the images of $\pi_1\partial_iM\xrightarrow{\text{incl}_\ast} \pi_1M$, which are indeed isomorphic to $\Z^2$. 

The closure of each peripheral subgroup in $\widehat{\pi_1M}$, denoted as $\overline{P_i}$, is canonically isomorphic to $\widehat{P_i}$. The reason is that any abelian subgroup in $\pi_1M$ is separable \cite{Ham01}, so it induces the full profinite topology on $P_i$ (see also \autoref{PROP: surface full top}), and then the profinite completion of the inclusion map $\widehat{P_i}\to \widehat{\pi_1M}$ is injective by \autoref{PROP: Left exact}. To be more specific, $\widehat{\Z}^2\cong \tensor P_i\cong \widehat{P_i}\cong \overline{P_i}$ by \autoref{PROP: abelianization}.

A series of works by Wilton--Zalesskii \cite{WZ17,WZ19} imply that any profinite isomorphism between finite-volume hyperbolic 3-manifolds respects the peripheral structure, see \autoref{introdef: peripheral Zx regular}.

\begin{proposition}[{\cite[Lemma 6.3]{Xu24a}}]\label{PROP: respect peripheral structure}
Let $M$ and $N$ be finite-volume hyperbolic 3-manifolds, and suppose $f:\widehat{\pi_1M}\to \widehat{\pi_1N}$ is an isomorphism. Then $f$ respects the peripheral structure. In particular, $M$ and $N$ have the same number of cusps, and $M$ is closed if and only if $N$ is closed. 
\end{proposition}

Peripheral $\Zx$-regularity is a concept introduced by the author in \cite{Xu24a,Xu24b}  that plays an important role in describing the gluing between 3-manifolds along tori. 

\begin{theorem}[{\cite[Theorem 7.2]{Xu24a}}]\label{THM: peripheral Zx regular}
Suppose $M$ and $N$ are cusped finite-volume hyperbolic 3-manifolds and   $f:\widehat{\pi_1M}\to \widehat{\pi_1N}$ is an isomorphism. Then $f$ is peripheral $\Zx$-regular. 

To be precise, let $P_1,\cdots, P_n$ be the peripheral subgroups of ${\pi_1M}$, and let $Q_1,\cdots, Q_n$ be the peripheral subgroups of ${\pi_1N}$. Up to a possible reordering, there exist  $g_1,\cdots, g_n\in \widehat{\pi_1N}$ such that $f(\overline{P_i})=\overline{Q_i}^{g_i}$ for each $i$, where $\overline{Q_i}^{g_i}$ is the abbreviation for the conjugate $g_i^{-1}\overline{Q_i}g_i$. And then,  there exists a unified $\lambda\in \Zx$ such that for each $1\le i \le n$,  the map 
$$C_{g_i}\circ f|_{\overline{P_i}}: \widehat{P_i}\cong \overline{P_i} \tto \overline{Q_i}\cong \widehat{Q_i}$$
is $\lambda$-regular.
\end{theorem}


It is proven in \cite{Xu24a} that the peripheral $\Zx$-regularity does not depend on the choice of the conjugacy representatives of the peripheral subgroups, nor does it depend on the elements $g_i$  that induces the conjugation. 

\begin{remark}\label{samecoef}
 $\lambda\in \Zx$ that appears in \autoref{THM: peripheral Zx regular} is exactly the coefficient such that $f$ is $\lambda$-regular. This is because the maps $ P_i\cong\Z^2\cong H_1( \partial _i M;\Z)\to H_1(M;\Z)_{\mathrm{free}}$ have non-trivial image by the ``half lives, half dies'' theorem (\cite[Lemma 8.15]{Lic97}), so that $f$ can only be $\pm\lambda$-regular according to \autoref{LEM: Zx+-1}.
\end{remark}

For brevity of notation, we abbreviate the results of \autoref{THM: Zx regular}, \autoref{PROP: respect peripheral structure}, \autoref{THM: peripheral Zx regular} and \autoref{samecoef} into the following setting, and conventionally fix the use of the letters therein from \autoref{sec:4} through \autoref{SEC: Dehn filling}.
\begin{convention}\label{conv+}
A {\em cusped hyperbolic setting of profinite isomorphism} $(M,N,f,\lambda,\Psi)$ consists of the following ingredients.
\begin{enumerate}[leftmargin=*]
\item $M$ and $N$ are finite-volume hyperbolic 3-manifolds each with $n>0$ cusps denoted as $\partial_1M,\cdots,\partial_nM$ and $\partial_1N,\cdots,\partial_nM$ respectively.
\item $f:\widehat{\pi_1M}\to \widehat{\pi_1N}$ is an isomorphism which is $\lambda$-regular.
\item $P_1,\cdots, P_n$ are peripheral subgroups of $\pi_1M$ and $Q_1,\cdots, Q_n$ are peripheral subgroups of $\pi_1N$, such that $f(\overline{P_i})=\overline{Q_i}^{g_i}$ for certain fixed elements $g_1,\cdots, g_n\in \widehat{\pi_1N}$.
\item For $1\le i \le n$, denote $\partial f_i= C_{g_i}\circ f|_{\overline{P_i}}: \widehat{P_i}\to \widehat{Q_i}$, which uniquely decomposes as $\partial f_i=\lambda \otimes \psi_i$ for some isomorphism $\psi_i: P_i\to Q_i$.
\item Let $\Psi:\partial M\to \partial N$ be the unqiue homeomorphism up to isotopy that induces $\psi_1,\cdots,\psi_n$. In other words, $\Psi(\partial_iM)=\partial_iN$ and $(\Psi|_{\partial_iM})_\ast=\psi_i:P_i\cong \pi_1\partial_iM\to \pi_1\partial_iN\cong Q_i$.
\end{enumerate}
\end{convention}

Note that altering $\lambda$ by a $\pm$-sign changes $\Psi$ by an orientation-preserving involution, which does not affect our reasonings in the remaining context.

\section{The homology coefficient}\label{sec:4}

In this section, we elaborate on the coefficient in $\Zx$ appearing in the $\Zx$-regular profinite isomorphism $f:\widehat{\pi_1M}\to \widehat{\pi_1N}$ between finite-volume hyperbolic 3-manifolds. The first question to be settled is the ambiguity stated in \autoref{RMK: betti number 0}, that is, when $M$ and $N$ are closed and $b_1(M)=b_1(N)=0$, the coefficient in $\Zx$ cannot be determined directly from the abelianization level. Thus, we need a more careful definition.

\begin{definition}\label{DEF: homology coefficient}
Suppose $M$ and $N$ are finite-volume hyperbolic 3-manifolds (including both closed and cusped manifolds), and $f:\widehat{\pi_1M}\to \widehat{\pi_1N}$ is an isomorphism. Let $M'$ and $N'$ be any pair of $f$--corresponding finite covers of $M$ and $N$ such that $b_1(M')=b_1(N')>0$, and let $f':\widehat{\pi_1M'}\to \widehat{\pi_1N'}$ be the restriction of $f$ on $\widehat{\pi_1M'}$ (see \autoref{DEF: corresponding cover}). The {\em homology coefficient} of $f$ is the unique element $\coef(f)=\pm\lambda\in \Zxpm$\footnote{In this article, an element in $\Zxpm$ is conventionally denoted as $\pm\lambda$ where $\lambda\in \Zx$.} such that $f'$ is $\lambda$-regular.
\end{definition}

The finite covers with positive first betti numbers exist by the result of Agol and Wise \cite{Ago13, Wise21}, and with the finite covers $M'$ and $N'$ fixed, the uniqueness of $\pm \lambda \in \Zx/\{\pm1\}$ follows from \autoref{RMK: betti number 0}. To justify that the homology coefficient  is well-defined, the following lemma is involved to show that $\coef(f)$ does not depend on the choice of the finite cover $M'$.

\begin{lemma}\label{LEM: homology coefficient well defined}
The homology coefficient in \autoref{DEF: homology coefficient} is well-defined. To be precise, under the setting of \autoref{DEF: homology coefficient}, suppose $M''$ and $N''$ is another pair of $f$--corresponding finite covers of $M$ and $N$ with $b_1(M'')=b_1(N'')>0$, and likewise, denote $f'': \widehat{\pi_1M''}\to \widehat{\pi_1N''}$ as the restriction of $f$ on $\widehat{\pi_1M''}$. 
Then, $f''$ is also $\lambda$-regular.
\end{lemma}
\begin{proof}
The proof relies on the following sub-lemma.
\def\free{\mathrm{free}}
\begin{sublemma}\label{LEM: Descend Zx regular}
Let $G$, $H$, $K$, $L$ be finitely generated groups, and let $p: G\rightarrow K$ and $q: H\rightarrow L$ be group homomorphisms such that $p(G)$ has finite index in $K$ and $q(H)$ has finite index in $L$. Suppose $f: \widehat{G}\to \widehat{H}$ and $g: \widehat{K}\to \widehat{L}$ are continuous homomorphisms such that the following diagram commutes.
\begin{equation*}
\begin{tikzcd}
\widehat{G} \arrow[r, "f"] \arrow[d, "\widehat{p}"'] & \widehat{H} \arrow[d, "\widehat{q}"] \\
\widehat{K} \arrow[r, "g"]                           & \widehat{L}                         
\end{tikzcd}
\end{equation*}
If $f$ is $\mu$-regular for some $\mu\in\Zx$, then $g$ is also $\mu$-regular.
\end{sublemma}
\begin{proof}
Denote $p_\ast: \ab{G}_{\free}\to \ab{K}_{\free}$ and $q_\ast: \ab{H}_{\free}\to \ab{L}_{\free}$. Then, through the isomorphisms in \autoref{PROP: abelianization}, $\widehat{p}_{\ast}:\tensor \ab{G}_{\free} \to \tensor \ab{K}_{\free}$ is decomposed as $1\otimes p_\ast$, and $\widehat{q}_\ast:\tensor \ab{H}_{\free}\to \tensor \ab{L}_{\free}$ is decomposed as $1\otimes q_\ast$. According to \autoref{LEM: free part}, we assume that $f_\ast:\tensor \ab{G}_{\free} \to \tensor\ab{H}_{\free}$ is decomposed as $\mu\otimes \phi$. 
Then the naturalness of the isomorphisms in \autoref{PROP: abelianization} implies a commutative diagram of $\widehat{\Z}$-modules.

\begin{equation*}
\begin{tikzcd}
\tensor \ab{G}_{\free} \arrow[dd, "1\otimes p_\ast=\widehat{p}_\ast"'] \arrow[rr, "f_\ast=\mu\otimes \phi"] &  & \tensor \ab{H}_{\free} \arrow[dd, "\widehat{q}_\ast=1\otimes q_\ast"] \\
                                                                                                            &  &                                                                       \\
\tensor \ab{K}_{\free} \arrow[rr, "g_\ast"]                                                                 &  & \tensor \ab{L}_{\free}                                               
\end{tikzcd}
\end{equation*}

In particular, the commutative diagram implies that $\phi(\ker(p_\ast))\subseteq \ker(q_\ast)$. Thus, $\phi$ descends to a homomorphism $\psi: p_\ast (\ab{G}_{\free}) \to q_\ast (\ab{H}_{\free})$, and the following diagram commutes.

\begin{equation*}
\begin{tikzcd}
\tensor p_\ast(\ab{G}_{\free}) \arrow[rr, "\mu\otimes \psi"] \arrow[dd, hook,"\text{incl.}"'] &  & \tensor q_\ast( \ab{H}_{\free}) \arrow[dd, hook,"\text{incl.}"] \\
                                                                              &  &                                                  \\
\tensor \ab{K}_{\free} \arrow[rr, "g_\ast"]                                   &  & \tensor \ab{L}_{\free}                          
\end{tikzcd}
\end{equation*}

Since $p_\ast(\ab{G}_{\free})$ is the image of $p(G)$ in $\ab{K}_{\free}$ through abelianization, and $p(G)$ has finite index in $K$, it follows that $p_\ast(\ab{G}_{\free})$ has finite index in $\ab{K}_{\free}$. Similarly, $q_\ast(\ab{H}_{\free})$ has finite index in $\ab{L}_{\free}$. It then follows from \cite[Lemma 7.10 (3)]{Xu24a} that $g_\ast$ can be decomposed as $\mu\otimes \widetilde{\psi}$ such that $\widetilde{\psi}|_{p_\ast(\ab{G}_{\free})}=\psi$. Thus, according to \autoref{LEM: free part}, $g $ is $\mu$-regular. 
\end{proof}

In the situation of \autoref{LEM: homology coefficient well defined}, let $M^\star$ be a common finite cover of $M'$ and $M''$, and let $N^\star$ be the $f$--corresponding finite cover of $N$. Denote  $f^\star=f|_{\widehat{\pi_1M^\star}}: \widehat{\pi_1M^\star}\to \widehat{\pi_1N  ^\star}$. 
Then according to \autoref{THM: Zx regular}, $f^\star$ is $\Zx$-regular. We suppose that $f^\star$ is $\mu$-regular, where $\mu \in\Zx$. 

Let $p': \pi_1M^\star\to \pi_1M'$ and $q':\pi_1N^\star\to \pi_1N'$  be the inclusion maps induced by the coverings. Then, we have the following commutative diagram.
\begin{equation*}
\begin{tikzcd}
\widehat{\pi_1M^\star} \arrow[r,"f^\star"] \arrow[d,"\widehat{p'}"'] & \widehat{\pi_1N^\star} \arrow[d, "\widehat{q'}"]\\
\widehat{\pi_1M'}\arrow[r,"f'"] & \widehat{\pi_1N'}
\end{tikzcd}
\end{equation*}

Since these maps are induced by finite coverings, $p'(\pi_1M^\star)$ has finite index in $\pi_1M'$ and $q'(\pi_1N^\star)$ has finite index in $\pi_1N'$. Consequently, according to \autoref{LEM: Descend Zx regular}, $f'$ is $\mu$-regular. However, $f'$ is already $\lambda$-regular. Since $b_1(M')=b_1(N')>0$, \autoref{RMK: betti number 0} implies that $\mu=\pm \lambda$; in other words, $f^\star$ is $\lambda$-regular.

Then, consider the commutative diagram
\begin{equation*}
\begin{tikzcd}
\widehat{\pi_1M^\star} \arrow[r,"f^\star"] \arrow[d,"\widehat{p''}"'] & \widehat{\pi_1N^\star} \arrow[d, "\widehat{q''}"]\\
\widehat{\pi_1M''}\arrow[r,"f''"] & \widehat{\pi_1N''}
\end{tikzcd}
\end{equation*}
where $p'': \pi_1M^\star\to \pi_1M''$ and $q'':\pi_1N^\star\to \pi_1N''$ are the inclusion maps induced by the coverings, with images having finite index. Again by \autoref{LEM: Descend Zx regular}, we derive from $f^\star $ being $\lambda$-regular that $f''$ is $\lambda$-regular.
\end{proof}

As a corollary, the homology coefficient is invariant  under taking finite covers. 

\begin{corollary}\label{COR: homology coef finite cover}
Suppose $M$ and $N$ are finite-volume hyperbolic 3-manifolds, and $f:\widehat{\pi_1M}\to \widehat{\pi_1N}$ is an isomorphism. Suppose $M^\star$ and $N^\star$ is a pair of $f$--corresponding finite covers of $M$and $N$ and $f^\star=f|_{\widehat{\pi_1M^\star}}:\widehat{\pi_1M^\star}\to \widehat{\pi_1N^\star}$. Then $\coef(f^\star)=\coef(f)$. 
\end{corollary}
\begin{proof}
In the definition of the homology coefficient, we can choose an $f$--corresponding pair of finite covers $M'$ and $N'$ over $M^\star$ and $N^\star$ with $b_1(M')=b_1(N')>0$. Then by the well-definedness \autoref{LEM: homology coefficient well defined}, both $\coef(f)$ and $\coef(f^\star)$ are the unique element $\pm \lambda\in \Zx/\{\pm 1\}$ such that $f'=f|_{\widehat{\pi_1M'}}:\widehat{\pi_1M'}\to \widehat{\pi_1N'}$ is $\lambda$-regular. 
\end{proof}

We remind the readers that in a cusped hyperbolic setting of profinite isomorphism $(M,N,f,\lambda,\Psi)$, $\coef(f)=\pm \lambda$, as explained in \autoref{samecoef}. 

\subsection{The group-level explanation}
This is  a  supplementary subsection. Though it will not be used in the subsequent sections, it might provide some help in understanding the reasonableness of the definition of homology coefficient. Indeed,  we show that 
there is another way to determine the  homology coefficient, focusing only on the group level instead of the abelianization (homology) level. 
 Let us start with some preparations.

\begin{lemma}\label{LEM: conjugate coef}
Suppose $M$ and $N$ are finite-volume hyperbolic 3-manifolds, and  $f:\widehat{\pi_1M}\to \widehat{\pi_1N}$ is an isomorphism. Let $\gamma\in \widehat{\pi_1N}$, and let $C_{\gamma}:\widehat{\pi_1N} \to \widehat{\pi_1N}$ be the conjugation by $\gamma$ on the left. Then $\coef(f)=\coef(C_{\gamma}\circ f)$. 
\end{lemma}
\begin{proof}For simplicity of notation, we assume $\coef(f)=\pm \lambda$. 

\textbf{Case 1:} $b_1(M)=b_1(N)>0$.

In fact, the conjugation has no effect on the (profinite) abelianization level, that is to say $f_\ast : \tensor H_1(M;\Z)\to \tensor H_1(N;\Z)$ and $(C_\gamma\circ f)_{\ast}: \tensor H_1(M;\Z)\to \tensor H_1(N;\Z)$ are the same map. Thus, the assumption that $f$ is $\lambda$-regular implies that $C_\gamma\circ f$ is also $\lambda$-regular. In this case, $b_1(M)=b_1(N)>0$, so there is no need to pass on to finite covers. Thus, $\coef(C_\gamma\circ f)= \pm \lambda$. 

\textbf{Case 2:} $b_1(M)=b_1(N)=0$.

Let $M'$ and $N'$ be a pair of $f$--corresponding finite regular covers of $M$ and $N$ with $b_1(M')=b_1(N')>0$, that is, $\pi_1M'$ and $\pi_1N'$ are normal subgroups of $\pi_1M$ and $\pi_1N$ respectively. Then $(C_{\gamma}\circ f) (\widehat{\pi_1M'})= C_{\gamma}(\widehat{\pi_1N'})= \widehat{\pi_1N'}$.  Denote $f'=f|_{\widehat{\pi_1M'}}$ and $(C_{\gamma}\circ f)'= (C_{\gamma}\circ f)|_{\widehat{\pi_1M'}}$. Then $f'$ is $\lambda$-regular.

We choose $\eta\in \widehat{\pi_1N'} \lhd \widehat{\pi_1N}$ and $h\in \pi_1N$ such that $\gamma= h\eta$. Since $\pi_1N'\lhd \pi_1N$, the conjugation map $C_{h}: \pi_1N\to \pi_1N$ restricts to an isomorphism $\varphi: \pi_1N'\to \pi_1N'$. Then it is easy to verify that the following diagram commutes. 
\begin{equation*}
\begin{tikzcd}[column sep=large, row sep=large]
\widehat{\pi_1M'} \arrow[r,"C_{\eta}\circ f'"] \arrow[d,"id=\widehat{id}"'] & \widehat{\pi_1N'}\arrow[d,"\widehat{\varphi}","\cong"'] \\
\widehat{\pi_1M'}\arrow[r, " (C_{\gamma}\circ f)'"] & \widehat{\pi_1N'}
\end{tikzcd}
\end{equation*}

By \textbf{Case 1}, since $\eta\in \widehat{\pi_1N'}$, $C_\eta \circ f'$ is $\lambda$-regular. Then, applying \autoref{LEM: Descend Zx regular} to this commutative diagram implies that $(C_\gamma\circ f)'$ is $\lambda$-regular. Recall that $b_1(M')=b_1(N')>0$, so by \autoref{DEF: homology coefficient}, $\coef(C_{\gamma}\circ f)=\pm \lambda$. 
\end{proof}

  For a profinite group $\Pi$ and an element $\alpha\in \Pi$, there exists a unique continuous homomorphism $\varphi_{\alpha}: \widehat{\Z}\to \Pi$ such that $\varphi_\alpha(1)=\alpha$, see \cite[Section 4.1]{RZ10}. For $\mu \in \widehat{\Z}$, we define $\alpha^\mu=\varphi_\alpha(\mu)\in \Pi$.  

\begin{proposition}\label{lem: conjugate homology coefficient}
Let $M$ and $N$ be finite-volume hyperbolic 3-manifolds. Suppose $f:\widehat{\pi_1M}\to \widehat{\pi_1N}$ is an isomorphism. If   $\alpha\in \pi_1M$ and $\beta\in \pi_1N$ are non-trivial elements such that $f(\alpha)$ is conjugate to $\beta^\lambda$ in $\widehat{\pi_1N}$ for some $\lambda\in \Zx$. Then, $\coef(f)=\pm \lambda$.
\end{proposition}
\def\N{\mathbb{N}}
\begin{proof}
By \autoref{LEM: conjugate coef}, we can compose $f$ with a conjugation and assume that $f(\alpha)=\beta^\lambda$. 

Since $\pi_1M$ and $\pi_1N$ are torsion-free, the subgroups $\langle \alpha \rangle < \pi_1M$ and $\langle \beta \rangle <\pi_1N$ are isomorphic to $\Z$. 
According to \cite{Min19}, any abelian subgroup of a virtually compact special group is a virtual retract. Combined with the virtual specialization theorem \cite{Ago13,Wise21}, there exists a finite-index subgroup of $\pi_1M$ corresponding to a finite cover $M'$ such that $\langle \alpha\rangle \subseteq  \pi_1M'$ and $\langle \alpha\rangle$ is a retract of $\pi_1M'$. By \cite[Lemma 7.6]{Xu24a}, the map $$\Z\cong \langle \alpha\rangle\cong \ab{\langle \alpha\rangle} \to \ab{(\pi_1M)}\cong H_1(M;\Z)\to \ab{(\pi_1M)}_{\mathrm{free}}\cong H_1(M;\Z)_{\mathrm{free}}$$ is injective. In other words, $[\alpha]$ represents a  non-trivial class in $H_1(M';\Z)_{\mathrm{free}}$. In particular, $b_1(M')>0$. 

Let $N'$ be the $f$--corresponding finite cover of $N$ that corresponds to $M'$. Then $\overline{\langle \beta \rangle} =f(\overline{\langle \alpha \rangle}) \subseteq f(\overline{\pi_1M'})= \widehat{\pi_1N'}$. As a consequence, $\langle \beta \rangle \subseteq \pi_1N'$, and $[\beta]$ represents a class in $H_1(N';\Z)_{\mathrm{free}}$. 

Let $f'=f|_{\widehat{\pi_1M'}}: \widehat{\pi_1M'} \to \widehat{\pi_1N'}$. Then, $f'_{\ast} : \tensor H_1(M';\Z)_{\mathrm{free}} \to \tensor H_1(N';\Z)_{\mathrm{free}}$ is an isomorphism of $\widehat{\Z}$-modules. In addition, since $f'(\alpha)=\beta^\lambda$, $f'_{\ast}$ sends the non-trivial class $1\otimes[\alpha]$ to $\lambda \otimes [\beta]$. In particular, since $f'_\ast$ is an isomorphism, $[\beta]$ is a non-trivial class in $H_1(N';\Z)_{\mathrm{free}}$. Recall that $f'$ is $\Zx$-regular by \autoref{THM: Zx regular}, so $f'$ can only be $\pm\lambda$-regular according to \autoref{LEM: Zx+-1}. Thus, $\coef(f)=\pm \lambda$ by \autoref{DEF: homology coefficient}.
\end{proof}

Moreover, the candidates $\alpha$ and $\beta$ in \autoref{lem: conjugate homology coefficient} always exist  as shown by the following proposition.

\begin{proposition}
Suppose $M$ and $N$ are finite-volume hyperbolic 3-manifolds, and $f:\widehat{\pi_1M}\to \widehat{\pi_1N}$ is an isomorphism. Then there exist  non-trivial elements  $\alpha\in \pi_1M$ and $\beta\in \pi_1N$ such that $f(\alpha)$ is conjugate to $\beta^\lambda$ in $\widehat{\pi_1N}$ for some $\lambda\in \Zx$.
\end{proposition}
\begin{proof}
When $M$ and $N$ are cusped, one can choose $\alpha$ in a peripheral subgroup of $\pi_1M$. Then, by \autoref{THM: peripheral Zx regular}, $f(\alpha)$ is conjugate to $\beta^\lambda$ for some $\beta$ in the corresponding peripheral subgroup of $\pi_1N$.

When $M$ and $N$ are closed, by the virtual fibering theorem \cite{Ago13}, there exists an $f$--corresponding pair of finite covers $M'$ and $N'$ which are fibered. Denote by $f':\widehat{\pi_1M'}\to \widehat{\pi_1N'}$ the restriction of $f$ on $\widehat{\pi_1M'}$. Liu \cite[Theorem 1.3]{Liu23} proved that $f'$ bijectively matches up the fiber structures of $M'$ and $N'$. Furthermore, fixing any $f'$--corresponding fiber structures on $M'$ and $N'$, Liu in \cite[Lemma 5.2]{Liu25} showed that $f'$ sends any element $\alpha\in \pi_1M'$ that represents a periodic trajectory in the psuedo-Anosov suspension flow of $M$ to a conjugate of the $\lambda$-power of some element   $\beta \in \pi_1N'$ that also  represents a periodic trajectory. Thus, $f(\alpha)$ is also a conjugate of $\beta^\lambda$.
\end{proof}


\section{(Co)homology theory and relative (co)homology theory for profinite groups}\label{sec:5}
The (co)homology theory for profinite groups was first defined by Serre \cite{Ser01}, and a relative version of this theory was defined and studied in depth by Wilkes \cite{Wil19}, which is essential to the proof in this paper. 
In this section, we briefly review this theory and sketch some important properties that will be used in the following sections.



\subsection{Definitions}
We start by reviewing the definition for relative (co)homology theory of discrete groups. We refer the readers to \cite{BE78,Tak59} for general references. 

\begin{definition}
Let $G$ be a discrete group, and let $M$ be a $G$-module.
\begin{enumerate}[leftmargin=*]
\item The {\em homology and cohomology of $G$ with coefficients in $M$} are defined as 
$$ H_\ast(G;M)=\mathrm{Tor}^{G}_\ast (\Z,M),\; \text{and}\; H^\ast(G;M)=\mathrm{Ext}^{\ast}_{G} (\Z,M).$$
\item Let $\mathcal{H}=\{H_i\}_{i\in I}$ be a non-empty  collection of subgroups in $G$, and let $\Delta_{G,\mathcal{H}}$ be the kernel of the augmentation homomorphism $\mathop{\oplus}_{i\in I} \Z[G/H_i] \to \Z$. The {\em homology and cohomology of $G$ relative to $\mathcal{H}$  with coefficients in $M$} are defined as  
$$ H_{\ast+1} (G,\mathcal{H};M)= \mathrm{Tor}^{G}_\ast(\Delta_{G,\mathcal{H}}^\perp,M),\;\text{and}\; H^{\ast+1}(G,\mathcal{H};M)=\mathrm{Ext}_{G}^\ast (\Delta_{G,\mathcal{H}},M),$$
where $\Delta_{G,\mathcal{H}}^\perp$ denotes $\Delta_{G,\mathcal{H}}$ with the right $G$-action $\delta\cdot g=g^{-1}\delta$. 
\end{enumerate}
\end{definition}


With the above functors replaced by their continuous versions, a similar construction is adapted to profinite groups that applies for specific coefficient modules.

\begin{notation}
Let $\Pi$ be a profinite group. $\mathfrak{C}(\Pi)$ denotes the category of profinite $\Pi$-modules, $\mathfrak{D}(\Pi)$ denotes the category of discrete $\Pi$-modules, and $\mathfrak{F}(\Pi)=\mathfrak{C}(\Pi)\cap \mathfrak{D}(\Pi)$ denotes the category  of finite $\Pi$-modules.
\end{notation}

For a profinite group $\Pi$, $\widehat{\Z}\szkh{\Pi}$ denotes the complete group algebra of $\Pi$, which is a profinite ring, see \cite[Section 5.3]{RZ10}. 
When applying the continuous version of Tor, Ext, and (co)homology functors, it is sometimes more convenient to restrict ourselves to the category of $\widehat{\Z}\szkh{\Pi}$-modules rather than the category of $\Pi$-modules. The next proposition points out their relations. 

\begin{proposition}[{\cite[Proposition 5.3.6]{RZ10}}]
Let $\Pi$ be a profinite group.
\begin{enumerate}[leftmargin=*]
\item The category of profinite $\widehat{\Z}\szkh{\Pi}$-modules $\mathfrak{C}(\widehat{\Z}\szkh{\Pi})$ is identical with $\mathfrak{C}(\Pi)$.
\item The category of discrete $\widehat{\Z}\szkh{\Pi}$-modules $\mathfrak{D}(\widehat{\Z}\szkh{\Pi})$ is identical with the subcategory of torsion modules in $\mathfrak{D}(\Pi)$.
\item $\mathfrak{F}(\Pi)=\mathfrak{C}(\widehat{\Z}\szkh{\Pi})\cap \mathfrak{D}(\widehat{\Z}\szkh{\Pi})$.
\end{enumerate}
\end{proposition}

\def\Tor{\mathbf{Tor}}
\def\Ext{\mathbf{Ext}}

Though the cohomology of a profinite group $\Pi$ is actually defined for coefficients in $\mathfrak{D}(\Pi)$, see \cite[Section 6.4]{RZ10}, only those in $\mathfrak{D}(\widehat{\Z}\szkh{\Pi})$ will be used in this article. Thus, we introduce the definitions following \cite[Section 6.2--6.3]{RZ10}.

\begin{definition}\label{DEF: profinite homology}
Let $\Pi$ be a profinite group.
\begin{enumerate}[leftmargin=*]
\item For $M\in \mathfrak{C}(\Pi)=\mathfrak{C}(\widehat{\Z}\szkh{\Pi})$, the {\em profinite homology of $\Pi$ with coefficients in $M$} is defined as 
$$ 
\H_\ast(\Pi;M)=\Tor_\ast^\Pi(\widehat{\Z},M) \in \mathfrak{C}(\widehat{\Z}) .
$$
\item For $A\in \mathfrak{D}(\widehat{\Z}\szkh{\Pi})$, the {\em profinite cohomology of $\Pi$ with coefficients in $A$} is defined as 
$$ 
\H^\ast(\Pi;A)=\Ext^\ast_\Pi(\widehat{\Z},A) \in \mathfrak{D}(\widehat{\Z}).
$$
\end{enumerate}
Here $\Tor$ and $\Ext$ are the continuous versions of $\mathrm{Tor}$ and $\mathrm{Ext}$ defined in \cite[Section 6.1]{RZ10}.
\end{definition}

Most generally, relative profinite (co)homology is defined for a profinite group relative to a collection of closed subgroups continuously indexed over a profinite space, see \cite{Wil19}. In this article, we shall only  consider the profinite (co)homology of a profinite group relative to a finite collection of closed subgroups. We abbreviate this setting into the following terminology.

\begin{definition}
\begin{enumerate}[leftmargin=*]
\item We say that $(G,\mathcal{H})$ is a {\em discrete group pair} if $G$ is a discrete group, and $\mathcal{H}=\{H_i\}_{i=1}^n$ is an ordered, non-empty, finite collection of   subgroups in $G$.
\item We say that $(\Pi,\mathcal{S})$ is a {\em profinite group pair} if $\Pi$ is a profinite  group, and $\mathcal{S}=\{S_i\}_{i=1}^n$ is an ordered, non-empty, finite collection of    closed subgroups in $\Pi$.
\end{enumerate}
\end{definition}

We now define the relative profinite (co)homology following \cite[Definition 2.1]{Wil19}. 

\begin{definition}\label{DEF: relative profinite homology}
Given a profinite group pair $(\Pi,\mathcal{S})$, let $\DELTA_{\Pi,\mathcal{S}}$ be the kernel of the augmentation homomorphism
$$ \mathop{\oplus}_{i=1}^{n}\widehat{\Z}\szkh{\Pi/S_i} \longrightarrow \widehat{\Z}.$$
Then $\DELTA_{\Pi,\mathcal{S}}\in \mathfrak{C}(\Pi)=\mathfrak{C}(\widehat{\Z}\szkh{\Pi})$.
\begin{enumerate}[leftmargin=*]
\item For $M\in \mathfrak{C}(\Pi)=\mathfrak{C}(\widehat{\Z}\szkh{\Pi})$, the {\em profinite homology of $\Pi$ relative to $\mathcal{S}$ with coefficients in $M$} is defined as 
$$ 
\H_{\ast+1}(\Pi,\mathcal{S};M)=\Tor_{\ast}^\Pi(\DELTA_{\Pi,\mathcal{S}}^\perp,M)\in \mathfrak{C}(\widehat{\Z}).
$$
\item For $A\in \mathfrak{D}(\widehat{\Z}\szkh{\Pi})$, the {\em profinite cohomology of $\Pi$ relative to $\mathcal{S}$ with coefficients in $A$} is defined as 
$$ 
\H^{\ast+1}(\Pi,\mathcal{S};A)=\Ext^\ast_\Pi(\DELTA_{\Pi,\mathcal{S}},A)\in \mathfrak{D}(\widehat{\Z}).
$$
\end{enumerate}
\end{definition}

\subsection{The long exact sequence}
The following proposition demonstrates the long exact sequence for relative profinite homology, analogous to the one for relative homology of discrete groups.

\begin{proposition}[{\cite[Proposition 2.4]{Wil19}}]\label{PROP: LES}
Let $(\Pi,\mathcal{S}=\{S_i\}_{i=1}^{n})$ be a profinite group pair. There is a natural long exact sequence for $M\in \mathfrak{C}(\Pi)$: 
$$
\cdots \to \H_{k+1}(\Pi,\mathcal{S};M)\xrightarrow{\partial} \mathop{\oplus}_{i=1}^{n}\H_k(  S_i;M) \to \H_k(\Pi;M) \to \H_{k}(\Pi,\mathcal{S};M) \xrightarrow{\partial} \cdots .
$$
\end{proposition}
\begin{proof}
The short exact sequence $$0\to \DELTA_{\Pi,\mathcal{S}}\to \mathop{\oplus}_{i=1}^{n} \widehat{\Z}\szkh{\Pi/S_i} \to \widehat{\Z}\to 0$$ in $\mathfrak{C}(\Pi)$ induces a long exact sequence
\newsavebox{\torequ}
\begin{lrbox}{\torequ}
$\cdots \to \Tor_{k}^{\Pi}(\DELTA_{\Pi,\mathcal{S}}^\perp, M)\to \mathop{\oplus}\limits_{i=1}^{n} \Tor_{k}^{\Pi}(\widehat{\Z}\szkh{\Pi/S_i}^\perp,M)\to \Tor_{k}^{\Pi}(\widehat{\Z},M)\to \Tor_{k-1}^{\Pi}(\DELTA_{\Pi,\mathcal{S}}^\perp, M)\to\cdots$
\end{lrbox}
\begin{equation*}
\usebox{\torequ}
\end{equation*}
by the homological functors $\Tor_{\ast}^{\Pi}$, see \cite[Proposition 6.1.9]{RZ10}. 

Note that $\Tor_{k}^{\Pi}(\DELTA_{\Pi,\mathcal{S}}^\perp, M)= \H_{k+1}(\Pi,\mathcal{S};M)$ and $ \Tor_{k-1}^{\Pi}(\DELTA_{\Pi,\mathcal{S}}^\perp, M)= \H_{k}(\Pi,\mathcal{S};M) $ by \autoref{DEF: relative profinite homology}, and $\Tor_{k}^{\Pi}(\widehat{\Z},M)= \H_k(\Pi;M)$ by \autoref{DEF: profinite homology}. In addition, 
$$\Tor_{k}^{\Pi}(\widehat{\Z}\szkh{\Pi/S_i}^\perp,M)\cong \Tor_{k}^{\Pi}((\widehat{\Z}\szkh{\Pi}\widehat{\otimes}_{\widehat{\Z}\szkh{S_i}}\widehat{\Z})^\perp,M)\cong \Tor^{S_i}_{k}(\widehat{\Z},M)=\H_k(S_i;M)$$ 
by the profinite version of Shapiro's lemma \cite[Theorem 6.10.9]{RZ10}. Thus, this implies the long exact sequence stated in the proposition. 
\end{proof}

Several properties related to this long exact sequence are proven in \cite{Wil19}.  
While most of them are only proved for cohomology,  the proofs for the homology version are completely analogous. 

{\em A  map of profinite group pairs} $f:(\Gamma,\mathcal{T}=\{T_i\}_{i=1}^{n})\to (\Pi,\mathcal{S}=\{S_i\}_{i=1}^{n})$ is a continuous homomorphism $f:\Gamma\to \Pi$ such that $f(T_i)\subseteq S_i$ for each $1\le i\le n$. The next proposition shows the functorial property of  relative profinite homology. 

\begin{proposition}[{\cite[Proposition 2.6]{Wil19}}]\label{PROP: functorial LES}
The relative profinite homology functors and the long exact squence (\autoref{PROP: LES})  are natural with respct to the maps of profinite group pairs $f:(\Gamma,\mathcal{T})\to (\Pi,\mathcal{S})$. In other words, for any $M\in \mathfrak{C}(\Pi)$, we can also view $M$ as belonging to $\mathfrak{C}(\Gamma)$ through $f$, and there is a commutative diagram of long exact sequences.
\begin{equation*}
\begin{tikzcd}[column sep=0.37cm]
\cdots \arrow[r] & {\H_{k+1}(\Gamma,\mathcal{T};M)} \arrow[d, "f_\ast"] \arrow[r, "\partial"] & \mathop{\oplus}\limits_{i=1}^{n}\H_k(T_i;M) \arrow[d, "(f|_{T_i})_\ast"] \arrow[r] & \H_k(\Gamma;M) \arrow[d, "f_\ast"] \arrow[r] & {\H_k(\Gamma,\mathcal{T};M)} \arrow[d, "f_\ast"] \arrow[r, "\partial"] & \cdots \\
\cdots \arrow[r] & {\H_{k+1}(\Pi,\mathcal{S};M)} \arrow[r, "\partial"]                           & \mathop{\oplus}\limits_{i=1}^{n}\H_k(S_i;M) \arrow[r]                                 & \H_k(\Pi;M) \arrow[r]                           & {\H_{k}(\Pi,\mathcal{S};M)} \arrow[r, "\partial"]                         & \cdots
\end{tikzcd}
\end{equation*}
\end{proposition}
In fact, \autoref{PROP: functorial LES}  
follows from the commutative diagram of short exact sequences in $\mathfrak{C}(\Gamma)$.
\begin{equation*}
\begin{tikzcd}
0 \arrow[r] & {\DELTA_{\Gamma,\mathcal{T}}} \arrow[d,"f_\ast"] \arrow[r] & \mathop{\oplus}\limits_{i=1}^{n}\widehat{\Z}\szkh{\Gamma/T_i} \arrow[d,"f_\ast"] \arrow[r] & \widehat{\Z} \arrow[d,"id"] \arrow[r] & 0 \\
0 \arrow[r] & {\DELTA_{\Pi,\mathcal{S}}} \arrow[r]              & \mathop{\oplus}\limits_{i=1}^{n}\widehat{\Z}\szkh{\Pi/S_i} \arrow[r]              & \widehat{\Z} \arrow[r]           & 0
\end{tikzcd}
\end{equation*}

The relative profinite (co)homology is not sensitive to taking conjugations of the closed subgroups, as shown by the following proposition.

\begin{proposition}[{\cite[Proposition 2.9]{Wil19}}]\label{PROP: conjugation invariant}
Let $(\Pi,\mathcal{S}=\{S_i\}_{i=1}^{n})$ be a profinite group pair. For $\gamma_1,\cdots, \gamma_n\in \Pi$, denote $\mathcal{S}^{\mathbf{\gamma}}=\{S_i^{\gamma_i}\}_{i=1}^{n}$. Then for any $M\in \mathfrak{C}(\Pi)$, there is an isomorphism $\H_\ast( \Pi,\mathcal{S}^{\mathbf{\gamma}};M)\cong \H_\ast( \Pi,\mathcal{S};M)$  such that the following diagram commutes.
\begin{equation*}
\begin{tikzcd}[column sep=0.31cm]
\cdots \arrow[r] & {\H_{k+1}(\Pi,\mathcal{S}^{\mathbf{\gamma}};M)} \arrow[d, "\cong"'] \arrow[r, "\partial"] & \mathop{\oplus}\limits_{i=1}^{n}\H_k(S_i^{\gamma_i};M) \arrow[d, "(C_{\gamma_i})_\ast"',"\cong"] \arrow[r] & \H_k(\Pi;M) \arrow[d, "id"'] \arrow[r] & {\H_{k}(\Pi,\mathcal{S}^{\mathbf{\gamma}};M)} \arrow[d, "\cong"'] \arrow[r, "\partial"] & \cdots \\
\cdots \arrow[r] & {\H_{k+1}(\Pi,\mathcal{S};M)} \arrow[r, "\partial"]                                       & \mathop{\oplus}\limits_{i=1}^{n}\H_k(S_i;M) \arrow[r]                                              & \H_k(\Pi;M) \arrow[r]                  & {\H_{k}(\Pi,\mathcal{S};M)} \arrow[r, "\partial"]                     & \cdots
\end{tikzcd}
\end{equation*}
\end{proposition}
In fact, \autoref{PROP: conjugation invariant} follows from the commutative diagram of short exact sequences in $\mathfrak{C}(\Pi)$.
\begin{equation*}
\begin{tikzcd}
0 \arrow[r] & {\DELTA_{\Pi,\mathcal{S}^{\mathbf{\gamma}}}} \arrow[d] \arrow[r] & \mathop{\oplus}\limits_{i=1}^{n}\widehat{\Z}\szkh{\Pi/S_i^{\gamma_i}} \arrow[d, "(C_{\gamma_i})_{i=1}^{n}"',"\cong"] \arrow[r] & \widehat{\Z} \arrow[d,"id"] \arrow[r] & 0 \\
0 \arrow[r] & {\DELTA_{\Pi,\mathcal{S}}} \arrow[r]                             & \mathop{\oplus}\limits_{i=1}^{n}\widehat{\Z}\szkh{\Pi/S_i} \arrow[r]                                      & \widehat{\Z} \arrow[r]           & 0
\end{tikzcd}
\end{equation*}

\begin{proposition}\label{PROP: conj natural}
Suppose $f: (\Gamma,\mathcal{T}=\{T_i\}_{i=1}^{n}) \to (\Pi, \mathcal{S}=\{S_i\}_{i=1}^{n})$ is a map of profinite group pairs. For elements $\gamma_1,\cdots, \gamma_n\in \Gamma$, let $\eta_i= f(\gamma_i)\in \Pi$. Then $f: (\Gamma, \mathcal{T}^{\gamma}=\{T_i^{\gamma_i}\}_{i=1}^{n})\to (\Pi, \mathcal{S}^\eta = \{S_i^{\eta_i}\}_{i=1}^{n})$ is also a map of profinite group pairs. In addition, for any $M\in \mathfrak{C}(\Pi)$, the following diagram commutes 
\begin{equation*}
\begin{tikzcd}
{\H_\ast(\Gamma,\mathcal{T}^{\gamma};M)} \arrow[d, "\cong"'] \arrow[r, "f_\ast"] & {\H_\ast(\Pi,\mathcal{S}^\eta; M)} \arrow[d, "\cong"] \\
{\H_\ast(\Gamma,\mathcal{T};M)} \arrow[r, "f_\ast"]                              & {\H_\ast(\Pi,\mathcal{S};M)}                         
\end{tikzcd}
\end{equation*}
where the vertical isomorphisms are given by \autoref{PROP: conjugation invariant}. 
\end{proposition}
\newsavebox{\conjconjequ}
\begin{lrbox}{\conjconjequ}
$
\begin{tikzcd}[column sep=0.3cm]
             & 0 \arrow[rr] &                                                                                             & {\DELTA_{\Pi,\mathcal{S}^\eta}} \arrow[dd, "\cong"'{yshift=-2ex}] \arrow[rr] &                                                                                                                                               & \mathop{\oplus}\limits_{i=1}^{n}\widehat{\Z}\szkh{\Pi/S_i^{\eta_i}} \arrow[rr] \arrow[dd, "((C_{\eta_i})_\ast)"'{yshift=-2ex}] &                                                            & \widehat{\Z} \arrow[rr] \arrow[dd, "id"{yshift=-2ex}] &   & 0 \\
0 \arrow[rr] &              & {\DELTA_{\Gamma,\mathcal{T}^{\gamma}}} \arrow[rr] \arrow[ru, "f_\ast"] \arrow[dd, "\cong"'{yshift=-2ex}] &                                                                 & \mathop{\oplus}\limits_{i=1}^{n}\widehat{\Z}\szkh{\Gamma/T_i^{\gamma_i}} \arrow[rr] \arrow[dd, "((C_{\gamma_i})_\ast)"'{yshift=-2ex}] \arrow[ru, "f_\ast"] &                                                                                                                   & \widehat{\Z} \arrow[rr] \arrow[dd, "id"'{yshift=-2ex}] \arrow[ru, "id"] &                                          & 0 &   \\
             & 0 \arrow[rr] &                                                                                             & {\DELTA_{\Pi,\mathcal{S}}} \arrow[rr]                           &                                                                                                                                               & \mathop{\oplus}\limits_{i=1}^{n}\widehat{\Z}\szkh{\Pi/S_i} \arrow[rr]                                             &                                                            & \widehat{\Z} \arrow[rr]                  &   & 0 \\
0 \arrow[rr] &              & {\DELTA_{\Gamma,\mathcal{T}}} \arrow[rr] \arrow[ru, "f_\ast"]                               &                                                                 & \mathop{\oplus}\limits_{i=1}^{n}\widehat{\Z}\szkh{\Gamma/T_i} \arrow[rr] \arrow[ru, "f_\ast"]                                                 &                                                                                                                   & \widehat{\Z} \arrow[rr] \arrow[ru, "id"]                   &                                          & 0 &  
\end{tikzcd}
$
\end{lrbox}
\begin{proof}
This is deduced from the following commutative diagram. 
\begin{equation*}
\scalebox{0.9}{\usebox{\conjconjequ}}
\end{equation*}
\end{proof}

\subsection{The profinite completion}
The relative (co)homology of a discrete group pair and the relative profinite (co)homology of its profinite completion can be  related by the following proposition.
\begin{proposition}\label{PROP: completion induce LES}
Let $(G,\mathcal{H}=\{H_i\}_{i=1}^{n})$ be a discrete group pair. Let $\overline{H_i}$ be the closure of the image of $H_i$ in $\widehat{G}$ through the canonical homomorphism $\iota:G\to \widehat{G}$, and let $\overline{\mathcal{H}}=\{\overline{H_i}\}_{i=1}^{n}$. Then $(\widehat{G},\overline{\mathcal{H}})$ is a profinite group pair. 
\begin{enumerate}[leftmargin=*]
\item\label{ciLES1} For $M\in\mathfrak{C}(\widehat{G})=\mathfrak{C}(\widehat{\Z}\szkh{\widehat{G}})$, $M$ can be viewed as a $\widehat{\Z}[G]$-module, and there is a natural homomorphism $H_\ast(G,\mathcal{H};M)\to \H_\ast(\widehat{G},\overline{\mathcal{H}};M)$ of $\widehat{\Z}$-modules induced by $\iota$ such that the following diagram commutes. 
\begin{equation*}
\begin{tikzcd}[column sep=0.3cm]
\cdots \arrow[r] & {H_{k+1}(G,\mathcal{H};M)} \arrow[d, "\iota_\ast"] \arrow[r, "\partial"] & \mathop{\oplus}\limits_{i=1}^{n} H_k(H_i;M) \arrow[r] \arrow[d, "(\iota|_{H_i})_\ast"] & H_k(G;M) \arrow[r] \arrow[d, "\iota_\ast"] & {H_{k}(G,\mathcal{H};M)} \arrow[r, "\partial"] \arrow[d, "\iota_\ast"] & \cdots \\
\cdots \arrow[r] & {\H_{k+1}(\widehat{G},\overline{\mathcal{H}};M)} \arrow[r, "\partial"]   & \mathop{\oplus}\limits_{i=1}^{n} \H_k(\overline{H_i};M) \arrow[r]                      & \H_k(\widehat{G};M) \arrow[r]              & {\H_{k}(\widehat{G},\overline{\mathcal{H}};M)} \arrow[r, "\partial"]   & \cdots
\end{tikzcd}
\end{equation*}
\item For   $A\in \mathfrak{D}(\widehat{\Z}\szkh{\widehat{G}})$, $A$ can be viewed as a $\widehat{\Z}[G]$-module, and there is a natural homomorphism $\H^\ast(\widehat{G},\overline{\mathcal{H}};A)\to H^\ast(G,\mathcal{H};A)$ of $\widehat{\Z}$-modules induced by $\iota$ such that the following diagram commutes. 
\begin{equation*}
\begin{tikzcd}[column sep=0.33cm]
\cdots \arrow[r] & {\H^{k}(\widehat{G},\overline{\mathcal{H}};A)} \arrow[r] \arrow[d, "\iota^\ast"] & \H^k(\widehat{G};A) \arrow[r] \arrow[d, "\iota^\ast"] & \mathop{\oplus}\limits_{i=1}^{n} \H^k(\overline{H_i};A) \arrow[r] \arrow[d, "(\iota|_{H_i})^\ast"] & {\H^{k+1}(\widehat{G},\overline{\mathcal{H}};A)} \arrow[d, "\iota^\ast"] \arrow[r] & \cdots \\
\cdots \arrow[r] & {H^{k}(G,\mathcal{H};A)} \arrow[r]                                               & H^k(G;A) \arrow[r]                                    & \mathop{\oplus}\limits_{i=1}^{n} H^k(H_i;A) \arrow[r]                                              & {H^{k+1}(G,\mathcal{H};A)} \arrow[r]                                               & \cdots
\end{tikzcd}
\end{equation*}
\end{enumerate}
\end{proposition}
\begin{proof}
These are induced by the following commutative diagram of short exact sequences of $G$-modules, see \cite[Section 6.1]{Wil19}.
\begin{equation*}
\begin{tikzcd}
0 \arrow[r] & {\Delta_{G,\mathcal{H}}} \arrow[d, "\iota_\ast"] \arrow[r] & {\mathop{\oplus}\limits_{i=1}^{n}\Z[{G/H_i}]} \arrow[d, "\iota_\ast"] \arrow[r]      & {\Z} \arrow[d,hook] \arrow[r] & 0 \\
0 \arrow[r] & {\DELTA_{\widehat{G},\overline{\mathcal{H}}}} \arrow[r]    & \mathop{\oplus}\limits_{i=1}^{n}\widehat{\Z}\szkh{\widehat{G}/\overline{H_i}} \arrow[r] & \widehat{\Z} \arrow[r]                 & 0       
\end{tikzcd}
\end{equation*}
\end{proof}

The homomorphism $\iota_\ast$ in \autoref{PROP: completion induce LES}~(\ref{ciLES1}) is also natural with respect to the group pairs, as shown by the following proposition.
\begin{proposition}\label{PROP: completion induce natural}
Let  $(G,\mathcal{H}=\{H_i\}_{i=1}^{n})$ and $(G',\mathcal{H}'=\{H_i'\}_{i=1}^{n})$ be  discrete group pairs. Suppose $\varphi: (G,\mathcal{H}) \to (G',\mathcal{H}')$ is a map of discrete group pairs, that is, $\varphi:G\to G'$ is a group homomorphism such that $\varphi(H_i)\subseteq H_i'$. Then $\widehat{\varphi} : (\widehat{G},\overline{\mathcal{H}})\to (\widehat{G'},\overline{\mathcal{H}'})$ is a map of profinite group pairs, and for any $M\in \mathfrak{C}(\widehat{G'})$, the following diagram commutes.
\begin{equation*}
\begin{tikzcd}
{H_\ast(G,\mathcal{H};M)} \arrow[r, "\varphi_\ast"] \arrow[d, "\iota_\ast"]         & {H_\ast(G',\mathcal{H}';M)} \arrow[d, "\iota'_\ast"] \\
{\H_\ast(\widehat{G},\overline{\mathcal{H}};M)} \arrow[r, "\widehat{\varphi}_\ast"] & {\H_\ast(\widehat{G'},\overline{\mathcal{H}'};M)}   
\end{tikzcd}
\end{equation*}
\end{proposition}
\begin{proof}
This follows from the commutative diagram 
\begin{equation*}
\begin{tikzcd}[column sep=small, row sep=small]
{\Delta_{G,\mathcal{H}}} \arrow[dd, "\iota_\ast"] \arrow[rr, hook] \arrow[rd, "\varphi_\ast"]       &                                                                       & {\mathop{\oplus}\limits_{i=1}^{n} \Z[G/H_i]} \arrow[dd, "\iota_\ast"{yshift=-1.5ex}] \arrow[rd, "\varphi_\ast"]                 &                                                                               \\
                                                                                                    & {\Delta_{G',\mathcal{H}'}} \arrow[dd, "\iota'_\ast"{yshift=-1.5ex}] \arrow[rr, hook] &                                                                                                                  & {\mathop{\oplus}\limits_{i=1}^{n} \Z[G'/H_i']} \arrow[dd, "\iota'_\ast"]      \\
{\DELTA_{\widehat{G},\overline{\mathcal{H}}}} \arrow[rr, hook] \arrow[rd, "\widehat{\varphi}_\ast"'] &                                                                       & \mathop{\oplus}\limits_{i=1}^{n} \widehat{\Z}\szkh{\widehat{G}/\overline{H_i}} \arrow[rd, "\widehat{\varphi}_\ast"'] &                                                                               \\
                                                                                                    & {\DELTA_{\widehat{G'},\overline{\mathcal{H}'}}} \arrow[rr, hook]      &                                                                                                                  & \mathop{\oplus}\limits_{i=1}^{n} \widehat{\Z}\szkh{\widehat{G'}/\overline{H_i'}}
\end{tikzcd}
\end{equation*}

In fact, $\varphi_\ast: H_\ast(G,\mathcal{H};M)\to H_\ast(G',\mathcal{H'};M)$ is induced by $\varphi_\ast: \Delta_{G,\mathcal{H}}\to \Delta_{G',\mathcal{H}'}$, and $\widehat{\varphi}_\ast$ is induced by $\widehat{\varphi}_\ast: \DELTA_{\widehat{G},\overline{\mathcal{H}}}\to \DELTA_{\widehat{G'},\overline{\mathcal{H}'}}$ according to \autoref{PROP: LES}.
\end{proof}

\subsection{Cohomological goodness}

Cohomological goodness was first defined by Serre in \cite{Ser01}, which can also be generalized to the relative version. 

For a discrete group $G$, let $\mathfrak{F}(G)$ denote the category of finite $G$-modules. It is easy to see that $\mathfrak{F}(G)=\mathfrak{F}(\widehat{G})$. 
\begin{definition}
\begin{enumerate}[leftmargin=*]
\item A discrete group $G$ is {\em cohomologically good}  (in the sense of Serre) if for any $A\in \mathfrak{F}(G)=\mathfrak{F}(\widehat{G})$, the map $\iota^\ast: \H^\ast( \widehat{G};A)\to H^\ast(G;A)$ is an isomorphism. 
\item A discrete group pair $(G,\mathcal{H})$ is {\em cohomologically good} if for any $A\in \mathfrak{F}(G)=\mathfrak{F}(\widehat{G})$, the map  $\iota^\ast: \H^\ast( \widehat{G},\overline{\mathcal{H}};A)\to H^\ast(G,\mathcal{H};A)$ is an isomorphism.
\end{enumerate}
\end{definition}

\begin{example}\label{EX: good}
The following groups are cohomologically good.
\begin{enumerate}[leftmargin=*]
\item Finite groups.
\item Finitely generated abelian groups, \cite[p. 16, Exercise 2]{Ser01}.
\item Finitely generated free groups and surface groups, \cite[Proposition 3.6]{GJZ08}.
\item Finitely generated 3-manifold groups, \cite[Theorem 3.5.1 and Lemma 3.7.1]{Cav12}.
\end{enumerate}
\end{example}

A combination of \autoref{PROP: Left exact} and \cite[Proposition 6.3]{Wil19} yields the following proposition, which passes the cohomological goodness from groups to a group pair.

\begin{proposition}\label{PROP: pair goodness}
Let $(G,\mathcal{H}=\{H_i\}_{i=1}^{n})$ be a discrete group pair. Suppose that $G$ is cohomologically good, each $H_i$ is cohomologically good, and the profinite topology on $G$ induces the   full profinite topology on each $H_i$. Then the discrete group pair $(G,\mathcal{H})$ is  cohomologically good.
\end{proposition}

The cohomological goodness can be transformed into the homological version by Pontryagin duality.

\begin{definition}
A discrete group $G$ is of {\em type $\text{FP}_{\infty}$} if the $G$-module $\Z$ is of type $\text{FP}_{\infty}$. A discrete group pair $(G,\mathcal{H})$ is of {\em type $\text{FP}_{\infty}$} if the $G$-module $\Delta_{G,\mathcal{H}}$ is of type $\text{FP}_{\infty}$.
\end{definition}

\begin{proposition}[{\cite[Proposition 6.2]{Wil19}}]\label{PROP: homology good}
\begin{enumerate}[leftmargin=*]
\item Suppose $G$ is a cohomologically good discrete group   of type $\text{FP}_{\infty}$. Then for any $M\in \mathfrak{F}(G)=\mathfrak{F}(\widehat{G})$, the map  $\iota_\ast: H_\ast(G ;M)\to \H_\ast (\widehat{G};M)$ is an isomorphism.
\item Suppose $(G,\mathcal{H})$ is a cohomologically good discrete group pair of type $\text{FP}_{\infty}$. Then for any $M\in \mathfrak{F}(G)=\mathfrak{F}(\widehat{G})$, the map  $\iota_\ast: H_\ast (G,\mathcal{H};M)\to \H_\ast(\widehat{G},\overline{\mathcal{H}};M)$ is an isomorphism.
\end{enumerate}
\end{proposition}
 
\begin{convention}
In the remaining part of this article, $\Z$ or $\widehat{\Z}$ appearing as coefficient modules of a certain (relative) (profinite) homology is assumed to have trivial action by the group unless otherwise stated. 
\end{convention}

\begin{proposition}\label{PROP: good tensor hatZ}
\begin{enumerate}[leftmargin=*]
\item\label{good tensor 1} Let $G$ be a cohomologically good discrete group of type $\text{FP}_{\infty}$. Then, there are isomorphisms of $\widehat{\Z}$-modules $$\tensor H_\ast(G;\Z) \xrightarrow{\;\;\cong\;\;} H_\ast (G;\widehat{\Z}) \xrightarrow[\iota_\ast]{\;\;\cong\;\;} \H_\ast(\widehat{G};\widehat{\Z}).$$
\item\label{good tensor 2} Let $(G,\mathcal{H})$ be a cohomologically good discrete group pair of type $\text{FP}_{\infty}$. Then, there are isomorphisms of $\widehat{\Z}$-modules $$\tensor H_\ast(G,\mathcal{H};\Z) \xrightarrow{\;\;\cong\;\;} H_\ast(G,\mathcal{H};\widehat{\Z})\xrightarrow[\iota_\ast]{\;\;\cong\;\;}  \H_\ast(\widehat{G},\overline{\mathcal{H}};\widehat{\Z}).$$
\end{enumerate}
\end{proposition}
\begin{proof}
We only prove part (\ref{good tensor 2}), and the proof for part (\ref{good tensor 1}) is similar. The first isomorphism follows from the universal coefficient theorem. Since $\widehat{\Z}$ has trivial $G$-action, there is a short exact sequence $$0\to \tensor H_\ast(G,\mathcal{H};\Z)  \to H_\ast(G,\mathcal{H};\widehat{\Z}) \to \mathrm{Tor}^{\Z}_1 ( \widehat{\Z}, H_{\ast-1}(G,\mathcal{H};\widehat{\Z}))\to 0,$$ see \cite[p. 60, Exercise 3]{Bro82}. $\widehat{\Z}$ is torsion-free, and is hence a flat $\Z$-module. Thus, $\mathrm{Tor}^{\Z}_1 ( \widehat{\Z}, H_{\ast-1}(G,\mathcal{H};\widehat{\Z}))=0$, which implies the first isomorphism. 

Moving on to the second isomorphism, we first deal with the profinite side. For $n\in \mathbb{N}$, let $\Z/n! $ be the finite $\widehat{G}$-module with trivial $\widehat{G}$-action, and the quotient maps $\Z/(n+1)! \to \Z/n! $ are also homomorphisms of $\widehat{G}$-modules. Then in $\mathfrak{C}(\widehat{G})$, we still have $\widehat{\Z}= \limi_n \Z/n! $. The functor $\Tor_\ast^{\widehat{G}}$ commutes with inverse limits \cite[Corollary 6.1.10]{RZ10}, so there is an isomorphism
\begin{equation*}
\H_\ast(\widehat{G},\overline{\mathcal{H}};\widehat{\Z})=\Tor_{\ast-1}^{\widehat{G}}(\DELTA_{\widehat{G},\overline{\mathcal{H}}}^{\perp},\widehat{\Z})\xrightarrow{\;\cong\;}  \limi\limits_{n \to \infty} \Tor_{\ast-1}^{\widehat{G}}(\DELTA_{\widehat{G},\overline{\mathcal{H}}}^{\perp},\Z/n!) =\limi\limits_{n \to \infty} \H_\ast(\widehat{G},\overline{\mathcal{H}};\Z/n!).
\end{equation*}

As for the discrete side, the relative homology of a group pair may not, in general, commute with the inverse limit  taken in the coefficient modules. However, in this specific case, we show that 
\begin{equation}\label{tensor hatZ equ1}
H_\ast(G,\mathcal{H};\widehat{\Z})\tto \limi\limits_{n \to \infty} H_\ast(G,\mathcal{H};\Z/n!)
\end{equation}
is an isomorphism.

Since $(G,\mathcal{H})$ is of type $\text{FP}_\infty$, by \cite[Chapter VIII, Proposition 4.5]{Bro82}, we can take a free resolution $F_\bullet\to \Delta_{G,\mathcal{H}}^\perp \to 0$ of finite type in the category of right $\Z[G]$-modules. Let $C^n_\bullet=F_\bullet \otimes_{\Z[G]} \Z/n!$. Then,  $C^0_\bullet\leftarrow C^1_\bullet \leftarrow C^2_\bullet \leftarrow \cdots$ is a tower of chain complexes of abelian groups. 
Note that $\Z/(n+1)!\to \Z/n!$ is surjective and the tensor product $F_k\otimes_{\Z[G]} -$ is right exact, so $C^{n+1}_k \to C^n_k$ is surjective for any $n,k\in \mathbb{N}$. Consequently, the tower  $C^0_\bullet\leftarrow C^1_\bullet  \leftarrow \cdots$ satisfies degree-wise the Mittag-Leffler condition. Then, according to \cite[Theorem 3.5.8]{Wei94}, there is a short exact sequence 
$$
0\to {\limi\limits_n}^1 H_{\ast+1} (C^n_{\bullet})\to H_\ast (\limi\limits_n C^n_\bullet) \to \limi\limits_n H_\ast (C^n_\bullet)\to 0.
$$

By construction $H_\ast (C^n_\bullet)= \mathrm{Tor}_{\ast}^G(\Delta^\perp_{G,\mathcal{H}}, \Z/n!)= H_{\ast+1}(G,\mathcal{H};\Z/n!)$. Note that $C^n_k\cong (\Z/n!)^{\oplus r_k}$, where $F_k= \Z[G]^{\oplus r_k}$. Thus, $\limi_n C^n_k\cong \widehat{\Z}^{\oplus r_k}\cong F_k\otimes_{\Z[G]} \widehat{\Z}$, and $H_\ast(\limi_n C^n_\bullet)\cong H_\ast(F_\bullet \otimes_{\Z[G]} \widehat{\Z})=\mathrm{Tor}_\ast^G(\Delta^\perp_{G,\mathcal{H}}, \widehat{\Z})= H_{\ast+1}(G,\mathcal{H};\widehat{\Z})$. Thus, we actually have a short exact sequence
$$
0\to {\limi\limits_n}^1 H_{\ast} (C^n_{\bullet})\to H_{\ast}(G,\mathcal{H};\widehat{\Z}) \xrightarrow{\;\phi\;} \limi\limits_n H_{\ast}(G,\mathcal{H};\Z/n!)\to 0.
$$
The map $\phi$ is exactly the one in (\ref{tensor hatZ equ1}), since it is induced by the maps between the coefficient modules $\widehat{\Z}\to \Z/n!$.

Note that for each $n,k\in \mathbb{N}$, $C^n_k\cong (\Z/n!)^{\oplus r_k}$ is finite, so $H_k(C^n_\bullet)$ is also finite. As a consequence, the tower of abelian groups $H_k(C^n_\bullet)$ satisfies the Mittag-Leffler condition, and $\limi^1_n H_k (C^n_\bullet)=0$ according to \cite[Proposition 3.5.7]{Wei94}. Therefore, the map $\phi$ is indeed an isomorphism. 

We can now relate the two sides. 
Since $(G,\mathcal{H})$ is cohomologically good and of type $\text{FP}_\infty$, \autoref{PROP: homology good} implies that for $n\in \mathbb{N}$, the map $\iota_\ast: H_\ast(G,\mathcal{H};\Z/n!) \to \H_\ast (\widehat{G},\overline{\mathcal{H}};\Z/n!)$ is an  isomorphism. In addition, these isomorphisms are natural with respect to the coefficient modules $\Z/n!$. Thus, we have a commutative diagram of abelian groups
\begin{equation*}
\begin{tikzcd}
{H_\ast(G,\mathcal{H};\widehat{\Z})} \arrow[d, "\cong"',"\phi"] \arrow[r, "\iota_\ast"] & {\H_\ast(\widehat{G},\overline{\mathcal{H}};\widehat{\Z})} \arrow[d, "\cong"] \\
{\limi\limits_n H_\ast(G,\mathcal{H};\Z/n!)} \arrow[r, "\cong"',"\iota_\ast"]                  & {\limi\limits_n \H_\ast(\widehat{G},\overline{\mathcal{H}};\Z/n!)}           
\end{tikzcd}
\end{equation*}
where the two vertical maps and the bottom map are all proven to be isomorphisms. Thus, $\iota_\ast: {H_\ast(G,\mathcal{H};\widehat{\Z})} \to  {\H_\ast(\widehat{G},\overline{\mathcal{H}};\widehat{\Z})}$  is also an isomorphism (of abelian groups). Note that $\iota_\ast$ is obviously a homomorphism of $\widehat{\Z}$-modules, so the proof completes.
\end{proof}

A simple application of cohomological goodness is illustrated by finitely generated abelian groups. The following proposition is also stated in \cite[Lemma 5.3]{Liu25}, for which we present a different proof.

\begin{proposition}\label{PROP: abelian}
Let $A=\Z u_1\oplus \cdots \oplus \Z u_n$ be a finitely generated free abelian group. Then $\widehat{A}=\widehat{\Z}u_1\oplus \cdots\oplus \widehat{\Z}u_n$, and 
\begin{equation}\label{equisoA}
\begin{tikzcd}[column sep=small]
{\bigwedge} _{\widehat{\Z}}[u_1,\cdots,u_n] \cong \widehat{\Z}\otimes_\Z H_\ast(A;\Z)\arrow[r] & H_\ast(A;\widehat{\Z}) \arrow[r,"\iota_\ast"]& \H_\ast(\widehat{A};\widehat{\Z}), 
\end{tikzcd}
\end{equation}
is an isomorphism of $\widehat{\Z}$-modules, where $\bigwedge _{\widehat{\Z}}[u_1,\cdots,u_n]$ is the exterior algebra over $\widehat{\Z}$ freely generated by $u_1,\cdots,u_n$ of degree $1$. 

In addition, suppose $B=\Z v_1\oplus \cdots \oplus \Z v_m$ and $f:\widehat{A}\to \widehat{B}$ is a continuous homomorphism. Then $f_\ast : \H_\ast(\widehat{A};\widehat{\Z})\to \H_\ast(\widehat{B};\widehat{\Z})$ is exactly the induced homomorphism of exterior $\widehat{\Z}$-algebras under this isomorphism. In other words, $f_\ast(u_{k_1}\wedge \cdots\wedge u_{k_j})= f(u_{k_1})\wedge \cdots \wedge f(u_{k_j})$. 
\end{proposition}
\begin{proof}
\def\a{\alpha}
\def\b{\beta}
$A\cong \Z^n$ is cohomologically good (\autoref{EX: good}) and of type $\text{FP}_\infty$ (in fact, $\text{FP}$) since it is the fundamental group of an aspherical finite CW-complex $T^n$. Thus, the isomorphsims $\widehat{\Z}\otimes_\Z H_\ast(A;\Z)\to H_\ast(A;\widehat{\Z}) \to \H_\ast(\widehat{A};\widehat{\Z})$ follow from \autoref{PROP: good tensor hatZ}. In order to determine the map $f_\ast: \H_\ast(\widehat{A};\widehat{\Z})\to \H_\ast(\widehat{B};\widehat{\Z})$, we present here a concrete calculation which meanwhile reproves  the isomorphism  (\ref{equisoA}).

There is a standard free resolution of $\Z$ as (left) ${\Z[A]}$-modules
\begin{equation*}
\begin{tikzcd}
{{\bigwedge}_{{\Z[A]}}[\a_1,\cdots,\a_n]=F_\ast(A)} \arrow[r, "\epsilon"] & \Z \arrow[r] & {0,}
\end{tikzcd}
\end{equation*}
where ${\bigwedge}_{{\Z[A]}}[\a_1,\cdots,\a_n]$ is a graded ${\Z[A]}$-module isomorphic to the exterior algebra generated by $\a_1,\cdots,\a_n$ of degree $1$.  The differentials are given by
\newsavebox{\equdiff}
\begin{lrbox}{\equdiff}
$d(\a_{k_1}\wedge\cdots\wedge \a_{k_r})=\mathop{\sum}\limits_{j=1}^{r} (-1)^{j-1}(u_{k_j}-1)\, \a_{k_1}\wedge\cdots\wedge \a_{k_{j-1}}\wedge \a_{k_{j+1}}\wedge \cdots\wedge \a_{k_r}$
\end{lrbox}
\begin{equation}\label{equdiff}
{\usebox{\equdiff}}
\end{equation}
satisfying the Leibniz rule, 
and $\epsilon: F_0(A)\cong {\Z[A]}\to \Z$ is the augmentation homomorphism.

It is easy to check that $\Z\otimes_{\Z[A]} F_\ast(A) \cong {\bigwedge}_{\Z}[\a_1,\cdots,\a_n]$ and the differentials in $\Z\otimes_{\Z[A]} F_\ast(A)$ are all trivial. Thus, $$H_\ast(A;\Z)=H_\ast(\Z\otimes_{\Z[A]} F (A))\cong \Z\otimes_{\Z[A]} F_\ast(A) \cong {\bigwedge}_\Z[\a_1,\cdots,\a_n].$$

Now we move on to the profinite side. Since ${\widehat{\Z}\szkh{\widehat{A}}}$ is an ${\Z[A]}$-algebra, let 
$$\widehat{F_\ast(A)}={\widehat{\Z}\szkh{\widehat{A}}} \otimes_{\Z[A]} F_\ast(A)={\bigwedge}_{{\widehat{\Z}\szkh{\widehat{A}}}}[\a_1,\cdots,\a_n]$$ be the graded ${\widehat{\Z}\szkh{\widehat{A}}}$-module equipped with the same differentials   defined by (\ref{equdiff}). Each $F_k(A)$ is a finitely generated free ${\Z[A]}$-module, and $A$ is cohomologically good. Thus, according to \cite[Proposition 2.9]{Wil17}, 
\begin{equation*}
\begin{tikzcd}
{{\bigwedge}_{{\widehat{\Z}\szkh{\widehat{A}}}}[\a_1,\cdots,\a_n]=\widehat{F_\ast(A)}} \arrow[r, "\epsilon"] & \widehat{\Z} \arrow[r] & {0,}
\end{tikzcd}
\end{equation*}
is also a free resolution of $\widehat{\Z}$ as (left) ${\widehat{\Z}\szkh{\widehat{A}}}$-modules. 
In addition,
\newsavebox{\equtensor}
\begin{lrbox}{\equtensor}
$
\widehat{\Z}\widehat{\otimes}_{\widehat{\Z}\szkh{\widehat{A}}} \widehat{F_\ast(A)} = \widehat{\Z}\widehat{\otimes}_{\widehat{\Z}\szkh{\widehat{A}}}(\widehat{\Z}\szkh{\widehat{A}} \otimes_{\Z[A]} F_\ast (A)) \cong \widehat{\Z}\otimes_{\Z} (\Z\otimes_{\Z[A]} F_\ast(A)) \cong {\bigwedge}_{\widehat{\Z}}[\a_1,\cdots,\a_n],$
\end{lrbox}
\begin{equation*}
{\usebox{\equtensor}}
\end{equation*}
and the differentials in $\widehat{\Z}\widehat{\otimes}_{\widehat{\Z}\szkh{\widehat{A}}} \widehat{F_\ast(A)}$ are again trivial by construction. 
Thus, 
\newsavebox{\tensorequii}
\begin{lrbox}{\tensorequii}
$\H_\ast(\widehat{A};\widehat{\Z})=H_\ast(\widehat{\Z}\widehat{\otimes}_{\widehat{\Z}\szkh{\widehat{A}}} \widehat{F_\ast(A)})\cong \widehat{\Z}\widehat{\otimes}_{\widehat{\Z}\szkh{\widehat{A}}} \widehat{F_\ast(A)} \cong 
 \widehat{\Z} {\otimes}_{\Z} H_\ast(A;\Z) \cong {\bigwedge}_{\widehat{\Z}}[\a_1,\cdots,\a_n]. 
$
\end{lrbox}
\begin{equation*}
{\usebox{\tensorequii}}
\end{equation*}

We apply the same construction to $B$, so that we obtain free resolutions
\begin{equation*}
\begin{tikzcd}
{{\bigwedge}_{{\Z[B]}}[\b_1,\cdots,\b_m]=F_\ast(B)} \arrow[r, "\epsilon"] & \Z \arrow[r] & {0}
\end{tikzcd}
\end{equation*}
and
\begin{equation*}
\begin{tikzcd}
{{\bigwedge}_{{\widehat{\Z}\szkh{\widehat{B}}}}[\b_1,\cdots,\b_m]=\widehat{F_\ast(B)}} \arrow[r, "\epsilon"] & \widehat{\Z} \arrow[r] & {0,}
\end{tikzcd}
\end{equation*}
whose differentials are similarly defined as in (\ref{equdiff}). 

The continuous homomorphism $f:\widehat{A}\to \widehat{B}$ induces a continuous homomorphism of complete group algebras $f_\sharp: \widehat{\Z}\szkh{\widehat{A}}\to \widehat{\Z}\szkh{\widehat{B}}$, so that $\widehat{F_\ast(B)}  \to \widehat{\Z} \to{0}$ also becomes a resolution of $\widehat{\Z}$ in $\mathfrak{C}(\widehat{\Z}\szkh{\widehat{A}})$. 
In order to determine $f_{\ast}:\H_{\ast}(\widehat{A},\widehat{\Z})\to \H_{\ast}(\widehat{B},\widehat{\Z})$, it remains to construct a homomorphism between the resolutions $\widehat{F_{\ast}(A)}\to \widehat{\Z}\to 0$ and $\widehat{F_\ast(B)}  \to \widehat{\Z} \to{0}$ in the category $\mathfrak{C}(\widehat{\Z}\szkh{\widehat{A}})$. 

Firstly, $\widehat{F_0(A)}=\widehat{\Z}\szkh{\widehat{A}}$ and $\widehat{F_0(B)}=\widehat{\Z}\szkh{\widehat{B}}$. Let $\Phi_0=f_\sharp: \widehat{F_0(A)}=\widehat{\Z}\szkh{\widehat{A}} \to \widehat{\Z}\szkh{\widehat{B}} = \widehat{F_0(B)}$. 
For the construction of $\Phi_1: \widehat{F_1(A)}\to \widehat{F_1(B)}$, we need the following sublemma whose proof is postponed towards the end of  this proposition. 
\begin{sublemma}\label{sublemma}
Let $\Pi$ be a profinite group, and let $\epsilon: \widehat{\Z}\szkh{\Pi}\to \widehat{\Z}$ be the augmentation homomorphism.  Suppose $x\in \Pi$ and $\gamma\in \widehat{\Z}$. Then, there exists $\omega_{x,\gamma}\in  \widehat{\Z}\szkh{\Pi}$ such that $\omega_{x,\gamma}(x-1)=x^\gamma-1$ and $\epsilon(\omega_{x,\gamma})=\gamma$. 
\end{sublemma}

Suppose that $f$ is defined by $f(u_i)=\sum_{j=1}^{m}c_i^jv_j$, where $c_i^j\in \widehat{\Z}$, written in the additive convention. Let $\Phi_1: \widehat{F_1(A)}= \widehat{\Z}\szkh{\widehat{A}}\alpha_1\oplus \cdots \oplus \widehat{\Z}\szkh{\widehat{A}}\alpha_n\to \widehat{F_1(B)}$ be the   homomorphism of $\widehat{\Z}\szkh{\widehat{A}}$-modules defined on the basis by
$$
\Phi_1(\alpha_i)=\sum_{j=1}^{m}(\prod_{k=1}^{j-1}v_{k}^{c_{i}^k})\omega_{v_j,c_i^j}\beta_j,
$$
where $\omega_{v_j,c_i^j}\in \widehat{\Z}\szkh{\widehat{B}}$ is given by \autoref{sublemma}. 

We claim that the following diagram commutes.
\begin{equation}\label{diagcommmm}
\begin{tikzcd}
\widehat{F_1(A)} \arrow[r,"d"] \arrow[d, "\Phi_1"] & \widehat{F_0(A)} \arrow[r, "\epsilon"] \arrow[d, "\Phi_0=f_{\sharp}"] & \widehat{\Z} \arrow[d, "id"] \\
\widehat{F_1(B)} \arrow[r,"d"]                     & \widehat{F_0(B)} \arrow[r, "\epsilon"]                                & \widehat{\Z}                
\end{tikzcd}
\end{equation}

\newsavebox{\aaay}
\begin{lrbox}{\aaay}
$
\begin{aligned}
d(\Phi_1(\alpha_i))=&\mathop{\textstyle\sum}\limits_{j=1}^{m} (\mathop{\textstyle\prod}\limits_{k=1}^{j-1} v_k^{c_i^k}) \omega_{v_jc_i^j}d(\beta_j) = \mathop{\textstyle\sum}\limits_{j=1}^{m} (\mathop{\textstyle\prod}\limits_{k=1}^{j-1} v_k^{c_i^k}) \omega_{v_jc_i^j}(v_j-1)\\
=& \mathop{\textstyle\sum}\limits_{j=1}^{m} (\mathop{\textstyle\prod}\limits_{k=1}^{j-1} v_k^{c_i^k}) (v_j^{c_{i}^{j}}-1) 
=   (\mathop{\textstyle\prod}\limits_{k=1}^{m}v_k^{c_i^k})-1 
=  f_\sharp (u_i-1) = \Phi_0(d(\alpha_i)),
\end{aligned}
$
\end{lrbox}

The commutativity of the right block is obvious. For the left block, since the diagram consists of homomorphisms of $\widehat{\Z}\szkh{\widehat{A}}$-modules, it suffices to verify the basis:
\begin{equation*}
\usebox{\aaay}
\end{equation*}
where the elements in $\widehat{B}$ are written in multiplicative convention. 

Then, $\Phi_0$ and $\Phi_1$ extend uniquely to a homomorphism of  $\widehat{\Z}\szkh{\widehat{A}}$-exterior algebras $\Phi_\ast:\widehat{F_{\ast}(A)}\to \widehat{F_{\ast}(B)}$. It follows   that $\Phi_\ast$ commutes with the differentials, since $\Phi_1$ and $\Phi_0$ do, and the differentials satisfy the Leibniz rule. 
%
In other words, we obtain a homomorphism of resolutions in $\mathfrak{C}(\widehat{\Z}{\szkh{\widehat{A}}})$.

\begin{equation*}
\begin{tikzcd}
\widehat{F_\ast(A)} \arrow[d, "\Phi_\ast"] \arrow[r,"\epsilon"] & \widehat{\Z} \arrow[d,"id"] \arrow[r] & 0\\
\widehat{F_\ast(B)} \arrow[r,"\epsilon"] &\widehat{\Z} \arrow[r] & 0 
\end{tikzcd}
\end{equation*}

Thus, $f_\ast: \H_\ast(\widehat{A};\widehat{\Z}) =\Tor_\ast^{\widehat{\Z}{\szkh{\widehat{A}}}}(\widehat{\Z},\widehat{\Z}) \to \Tor_\ast^{\widehat{\Z}{\szkh{\widehat{B}}}}(\widehat{\Z},\widehat{\Z})= \H_\ast(\widehat{B};\widehat{\Z})$ is induced by the homomorphism of chain complexes 
\begin{equation*}
\begin{tikzcd}[column sep=large]
\widehat{\Z}\widehat{\otimes}_{\widehat{\Z}{\szkh{\widehat{A}}}}\widehat{F_\ast(A)} \arrow[rr,"id\,\widehat{\otimes}_{f_\sharp} \Phi_\ast"] &  & \widehat{\Z}\widehat{\otimes}_{\widehat{\Z}{\szkh{\widehat{B}}}}\widehat{F_\ast(B)} .
\end{tikzcd}
\end{equation*}
Both $\widehat{\Z}\widehat{\otimes}_{\widehat{\Z}{\szkh{\widehat{A}}}}\widehat{F_\ast(A)}$ and $\widehat{\Z}\widehat{\otimes}_{\widehat{\Z}{\szkh{\widehat{B}}}}\widehat{F_\ast(B)}$ have trivial differentials, so we finally obtain the following commutative diagram. 

\begin{equation*}
\begin{tikzcd}[column sep=small,row sep=large]
\H_\ast(\widehat{A};\widehat{\Z}) \arrow[d, "f_\ast"'] &  & H_\ast(\widehat{\Z}\widehat{\otimes}_{\widehat{\Z}{\szkh{\widehat{A}}}}\widehat{F_\ast(A)}) \arrow[d, "(id\,\widehat{\otimes}_{f_\sharp} \Phi_\ast)_\ast"'] \arrow[ll, "\cong"'] &  & \widehat{\Z}\widehat{\otimes}_{\widehat{\Z}{\szkh{\widehat{A}}}}\widehat{F_\ast(A)} \arrow[ll, "\cong"'] \arrow[d, "id\,\widehat{\otimes}_{f_\#}\Phi_\ast"] &  & {{\bigwedge}_{\widehat{\Z}}[\a_1,\cdots, \a_n]} \arrow[ll, "\cong"'] \arrow[d, "\Psi_\ast"] \\
\H_\ast(\widehat{B};\widehat{\Z})                       &  & H_\ast(\widehat{\Z}\widehat{\otimes}_{\widehat{\Z}{\szkh{\widehat{B}}}}\widehat{F_\ast(B)}) \arrow[ll, "\cong"']                                     &  & \widehat{\Z}\widehat{\otimes}_{\widehat{\Z}{\szkh{\widehat{B}}}}\widehat{F_\ast(B)} \arrow[ll, "\cong"']                                              &  & {{\bigwedge}_{\widehat{\Z}}[\b_1,\cdots, \b_m]} \arrow[ll, "\cong"']                   
\end{tikzcd}
\end{equation*}

Note that  $\Psi_\ast$ is a homomorphism of exterior algebras since $\Phi_\ast$ is. Thus, $\Psi_\ast$ is determined by $\Psi_0:\widehat{\Z}\to \widehat{\Z}$ and $\Psi_1:\widehat{\Z}\alpha_1\oplus\cdots\oplus \widehat{\Z}\alpha_n \to \widehat{\Z}\beta_1\oplus \cdots \oplus \widehat{\Z}\beta_m$. From diagram (\ref{diagcommmm}), it is clear that $\Psi_0=id: \widehat{\Z}\to \widehat{\Z}$, and $\Psi_1$ fits into the following commutative diagram
\begin{equation*}
\begin{tikzcd}
\widehat{\Z} \szkh{\widehat{A}} \alpha_1 \oplus \cdots \oplus \widehat{\Z} \szkh{\widehat{A}} \alpha_n \arrow[r,"\Phi_1"] \arrow[d,"{(\epsilon,\cdots,\epsilon)}"'] & \widehat{\Z}\szkh{\widehat{B}} \beta_1 \oplus \cdots \oplus \widehat{\Z} \szkh{\widehat{B}}\beta_m  \arrow[d,"{(\epsilon,\cdots,\epsilon)}"] \\
\widehat{\Z} \alpha_1 \oplus \cdots \oplus \widehat{\Z}\alpha_n  \arrow[r,"\Psi_1"] & \widehat{\Z}\beta_1\oplus \cdots \oplus \widehat{\Z}\beta_m 
\end{tikzcd}
\end{equation*}
where the vertical maps are induced by the augmentation homomorphisms $\widehat{\Z}\szkh{\widehat{A}}\xrightarrow{\epsilon} \widehat{\Z}$ and  $\widehat{\Z}\szkh{\widehat{B}}\xrightarrow{\epsilon} \widehat{\Z}$. 
Consequently, $$\Psi_1(\alpha_i)=(\epsilon,\cdots,\epsilon)\Phi_1(\alpha_i)=\mathop{\textstyle\sum}\limits_{j=1}^{m}\epsilon((\mathop{\textstyle\prod}\limits_{k=1}^{j-1}v_k^{c_i^k})\omega_{v_j,c_{i}^j})\beta_j=\mathop{\textstyle\sum}\limits_{j=1}^{m}\epsilon(\omega_{v_j,c_{i}^j})\beta_j= \mathop{\textstyle\sum}\limits_{j=1}^{m}c_i^j \beta_j,$$
where  the last equation follows from \autoref{sublemma}. 
Therefore, under the isomorphism (\ref{equisoA}), $\Psi_1$ is identical with $f$, and $f_\ast:\H_{\ast}(\widehat{A},\widehat{\Z})\to \H_{\ast}(\widehat{B},\widehat{\Z})$ is exactly the homomorphism  $\Psi_\ast$ of the  exterior $\widehat{\Z}$-algebras induced by $\Psi_1=f$.  
\end{proof}

\begin{proof}[Proof of \autoref{sublemma}]

Let $\Gamma=\widehat{\Z}$. For each $n\in \mathbb{N}$, denote by $\Gamma_n=\{s^n\mid s\in \Gamma\}$ the unique index-$n$ open subgroup. Let $q_n: \Gamma\to \Gamma/\Gamma_n \cong \Z/n$  be the quotient homomorphism, and let  $(q_n)_{\sharp}: \widehat{\Z}\szkh{\Gamma}\to \Z/n\szkh{\Gamma/\Gamma_n}= \Z/n[\Gamma/\Gamma_n]\cong \Z/n[\Z/n]$ be the quotient homomorphism of complete group algebras. Then, $\widehat{\Z}\szkh{\Gamma}$ can be expressed as an inverse limit $\widehat{\Z}\szkh{\Gamma}=\limi \Z/n [\Gamma/\Gamma_n]$. 

Denote 
$$
S_n=\left\{\left.\mathop{\textstyle\sum}\limits_{i=0}^{a-1}q_n(t)^{i} \right |\, a\in \mathbb{N}\text{ and }a\cong \gamma \;(\mathrm{mod}\;{n})\right\}\subseteq \Z/n [\Gamma/\Gamma_n].
$$
The following three observations are direct. First, for any $w\in S_n$, $w(q_n(t)-1)=q_n(t)^a-1=(q_n)_{\sharp}(t^{\gamma}-1)$, since $a\cong \gamma \;(\mathrm{mod}\;{n})$. Second, let $\epsilon_n: \Z/n[\Gamma/\Gamma_n]\to \Z/n$ be the augmentation homomorphism, then $\epsilon_n(w)\cong a\cong \gamma\;(\mathrm{mod}\;{n})$ for any $w\in S_n$. Third, $\{S_n\}_{n\in \mathbb{N}}$ forms an inverse system (ordered by divisibility) of finite sets. Therefore, according to \cite[Proposition 1.1.4]{RZ10}, $\limi S_n$ is non-empty. Pick $\omega\in \limi S_n$. The first observation implies that $(q_n)_\sharp (\omega(t-1))=(q_n)_\sharp  (t^\gamma -1)$ for all $n\in \mathbb{N}$, and hence $\omega(t-1)= t^\gamma -1$. Moreover, with $\epsilon':\widehat{\Z}\szkh{\Gamma}\to \widehat{\Z}$ being the augmentation homomorphism, the second observation implies that $\epsilon'(\omega)\cong \gamma \;(\mathrm{mod}\;{n}) $ for all $n\in \mathbb{N}$, and hence $\epsilon'(\omega)= \gamma$.

Now, let $\varphi:\Gamma\to \Pi$ be the unique continuous homomorphism such that $\varphi(t)=x$, and let $\varphi_{\sharp}: \widehat{\Z}\szkh{\Gamma}\to \widehat{\Z}\szkh{\Pi}$ be the homomorphism of complete group algebras induced by $\varphi$. We show that $\omega_{x,\gamma}:=\varphi_{\sharp}(\omega)$ meets the requirements. Indeed, $\varphi_{\sharp}$ is a homomorphism of rings, so $\omega_{x,\gamma}(x-1)=\varphi_{\sharp}(\omega)\varphi_{\sharp}(t-1)=\varphi_{\sharp}(t^\gamma-1)= x^\gamma-1$. Moreover, $\varphi_{\sharp}$ commutes with the augmentation homomorphisms: $\epsilon\circ \varphi_{\sharp}=\epsilon'$, so $\epsilon(\omega_{x,\gamma})=\epsilon'(x)=\gamma$. 
\end{proof}

\begin{remark}
In fact, since $A$ and $\widehat{A}$ are abelian, Pontryagin products can be defined on $H_\ast(A;\widehat{\Z})$ and $\H_\ast(\widehat{A};\widehat{\Z})$, making them $\widehat{\Z}$-algebras. A careful argument based on \autoref{PROP: good tensor hatZ} shows that $\iota_\ast:H_\ast(A;\widehat{\Z})\to \H_\ast(\widehat{A};\widehat{\Z})$ is an isomorphism of $\widehat{\Z}$-algebras. In addition, a continuous homomorphism $\widehat{A}\to \widehat{B}$ induces a homomorphism of $\widehat{\Z}$-algebras. The result of \autoref{PROP: abelian} then follows by identifying $\H_1(\widehat{A};\widehat{\Z})$ with $\widehat{A}$ and $\H_1(\widehat{B};\widehat{\Z})$ with $\widehat{B}$, since they are abelian.
\end{remark}


\subsection{Topological models}
(Co)homology of groups are related to (co)homology of their Eilenberg-MacLane spaces, possibly with respect to local coefficents. For simplicity, we make the statements for $\Z$-coefficients in this section. 
\begin{definition}
Given a discrete group $G$, a cell complex $X$ is called an {\em Eilenberg-MacLane space  for $G$} or a {\em $K(G,1)$--cell complex} if $X$ is connected, $\pi_1X=G$, and $\pi_iX=0$ when $i\ge 2$. 
\end{definition}

The following theorem is well-known, see for example \cite[Section I.4 and III.1]{Bro82}. 
\begin{theorem}\label{THM: K(G,1)}
Suppose $G$ is a discrete group. For     a $K(G,1)$--cell complex $X$,  there is a natural isomorphism $H_\ast(X;\Z)\cong H_\ast(G;\Z)$.
\end{theorem}

\begin{corollary}\label{COR: 3-mfd homology}
Let $M$ be a compact aspherical 3-manifold. Then, there is an isomorphism 
\begin{equation*}
\begin{tikzcd}[column sep=small]
\tensor H_\ast(M;\Z) \arrow[r,"\cong"] & \tensor H_\ast(\pi_1M;\Z)\arrow[r,"\cong"]  & H_\ast(\pi_1M;\widehat{\Z}) \arrow[r,"\cong"]  & \H_\ast(\widehat{\pi_1M};\widehat{\Z}) .
\end{tikzcd}
\end{equation*}
\end{corollary}
\begin{proof}
The first isomorphism follows from \autoref{THM: K(G,1)}.  In addition, $M$ can be equipped with a finite cell complex structure, so $\pi_1M$ is of type $\text{FP}_\infty$ (in fact, $\text{FP}$) according to \cite[Chapter I, Proposition 4.2]{Bro82}. By \cite{Cav12}, $\pi_1M$ is cohomologically good. Thus, the second and third isomorphism follows from \autoref{PROP: good tensor hatZ}.
\end{proof}

Similar results are established in the relative version.

\begin{definition}
Let $(G,\mathcal{H}=\{H_i\}_{i=1}^{n})$ be a discrete group pair. 
A cell complex pair  $(X,Y)$ is called an {\em Eilenberg-MacLane pair for $(G,\mathcal{H})$} if 
\begin{enumerate}[leftmargin=*]
\item $X$ is a $K(G,1)$--cell complex, 
\item $Y$ is a subcomplex of $X$ consisting of $n$ connected components $Y_1,\cdots,Y_n$, 
\item each $Y_i$ is a $K(H_i,1)$--cell complex,
\item for each $1\le i \le n$, the inclusion map $Y_i\hookrightarrow X$ induces the injective homomorphism $H_i=\pi_1Y_i\to \pi_1X=G$, up to a suitable switch of basepoints.
\end{enumerate}
\end{definition}

\begin{theorem}[{\cite[Theorem 1.3]{BE78}}]\label{THM: space pair group pair}
Suppose $(G,\mathcal{H})$ is a discrete group pair and $(X,Y)$ is an Eilenberg-MacLane pair for $(G,\mathcal{H})$. Suppose $\mathcal{H}=\{H_i\}_{i=1}^{n}$ and $Y=\sqcup_{i=1}^{n}Y_i$. 
Then there is an isomorphism $H_\ast(X,Y;\Z)\cong H_\ast(G,\mathcal{H};\Z)$ that commutes with the long exact sequences up to the indicated sign change.
\newsavebox{\toples}
\begin{lrbox}{\toples}
$
\begin{tikzcd}[column sep=tiny]
\cdots \arrow[r] & {H_{k+1}(X,Y;\Z)} \arrow[dd,"\cong"] \arrow[rr] &  & \mathop{\oplus}\limits_{i=1}^{n}H_k(Y_i;\Z) \arrow[dd,"\cong"] \arrow[rr] &  & H_k(X;\Z) \arrow[dd,"\cong"] \arrow[rr] &            & {H_{k}(X,Y;\Z)} \arrow[dd,"\cong"] \arrow[r] & \cdots \\
                 &                                        &  &                                                                  &  &                                & \boxed{-1} &                                     &        \\
\cdots \arrow[r] & {H_{k+1}(G,\mathcal{H};\Z)} \arrow[rr]  &  & \mathop{\oplus}\limits_{i=1}^{n}H_k(H_i;\Z) \arrow[rr]            &  & H_k(G;\Z) \arrow[rr]            &            & {H_{k}(G,\mathcal{H};\Z)} \arrow[r]  & \cdots
\end{tikzcd}
$
\end{lrbox}
\begin{equation*}
{\usebox{\toples}}
\end{equation*}
\end{theorem}
The isomorphism   $H_\ast(X,Y;\Z)\cong H_\ast(G,\mathcal{H};\Z)$ in \autoref{THM: space pair group pair}  is natural in the sense of the following proposition.

\begin{proposition}\label{PROP: space pair group pair natural}
Suppose $(X,Y)$ is an Eilenberg-MacLane pair of $(G,\mathcal{H})$, and $(X',Y')$ is an Eilenberg-MacLane pair for $(G',\mathcal{H}')$. Suppose $\mathcal{H}=\{H_i\}_{i=1}^{n}$, $\mathcal{H}'=\{H_i'\}_{i=1}^{n}$, $Y=\sqcup_{i=1}^{n}Y_i$ and $Y'=\sqcup_{i=1}^{n}Y_i'$. Let $\Phi: (X,Y)\to (X',Y')$ be a cellular map such that $\Phi(Y_i)\subseteq Y_i'$. Up to a switch of basepoints, $\Phi$ induces homomorphisms $G=\pi_1X\to \pi_1X'=G'$ and $H_i=\pi_1Y_i\to \pi_1Y_i'=H_i'$, which we denote as a map of discrete group pairs $\phi:(G,\mathcal{H})\to (G',\mathcal{H}')$. Then the following diagram commutes
\begin{equation*}
\begin{tikzcd}
H_\ast(X,Y;\Z) \arrow[r,"\Phi_\ast"] \arrow[d,"\cong"] & H_\ast(X',Y';\Z) \arrow[d,"\cong"]\\
H_\ast(G,\mathcal{H};\Z) \arrow[r,"\phi_\ast"] & H_\ast (G',\mathcal{H}';\Z)
\end{tikzcd}
\end{equation*}
where the vertical isomorphisms are given by \autoref{THM: space pair group pair}.
\end{proposition}

Since the author could not immediately find this in literature, let us include a quick proof for \autoref{PROP: space pair group pair natural}.

\begin{proof}
Let $\widetilde{X}$ be the universal covering space of $X$, and let $\widetilde{Y}$ be the pre-image of $Y$ in $\widetilde{X}$. Then, each component of $\widetilde{Y}$ is a universal covering space of some $Y_i$.  
Then $H_k(\widetilde{X};\Z)=H_k(\widetilde{Y};\Z)=0$ for $k\ge 1$, and there is the following commutative diagram
\begin{equation*}
\begin{tikzcd}
H_0(\widetilde{Y};\Z)\arrow[r,"\mathrm{incl}_\ast"]\arrow[d,"\cong"']& H_0(\widetilde{X};\Z) \arrow[d,"\cong"]\\
\mathop{\oplus}\limits_{i=1}^{n}\Z[G/H_i] \arrow[r,"\epsilon",two heads] & \Z
\end{tikzcd}
\end{equation*}
where $\epsilon$ denotes the augmentation homomorphism.

By the long exact sequence, $H_k(\widetilde{X},\widetilde{Y};\Z)=0$  for $k\neq 1$, and $H_1(\widetilde{X},\widetilde{Y};\Z)=\ker(\epsilon)=\Delta_{G,\mathcal{H}}$. Let $C_{\bullet}(\widetilde{X},\widetilde{Y})=C_\bullet(\widetilde{X})/C_\bullet(\widetilde{Y})$ be the relative cellular chain complex consisting of free $\Z[G]$-modules. 
In particular, $C_1(\widetilde{X},\widetilde{Y})\to C_0(\widetilde{X},\widetilde{Y})$ is surjective, so we have a short exact sequence
$$
0\tto Z_1(\widetilde{X},\widetilde{Y})\tto C_1(\widetilde{X},\widetilde{Y})\tto C_0(\widetilde{X},\widetilde{Y})\tto 0.
$$
Note that both $C_1(\widetilde{X},\widetilde{Y})$ and $C_0(\widetilde{X},\widetilde{Y})$ are free $\Z[G]$-modules, so $Z_1(\widetilde{X},\widetilde{Y})$ is a projective $\Z[G]$-module. Consequently, we obtain a projective resolution
$$
\cdots \to C_3(\widetilde{X},\widetilde{Y})\to C_2(\widetilde{X},\widetilde{Y})\to Z_1(\widetilde{X},\widetilde{Y}) \to H_1(\widetilde{X},\widetilde{Y};\Z)\cong \Delta_{G,\mathcal{H}}\to 0.
$$

We denote $D_{\bullet}(\widetilde{X},\widetilde{Y})= [\cdots  \to C_2(\widetilde{X},\widetilde{Y})\to Z_1(\widetilde{X},\widetilde{Y}) \to 0]$ and $D_{\bullet}(X,Y)=[\cdots  \to C_2(X, Y)\to Z_1(X, Y) \to 0]$ as the adjusted chain complexes. Obviously, $D_{\bullet}(X,Y)\cong D_{\bullet}^\perp(\widetilde{X},\widetilde{Y})\otimes_{\Z[G]}\Z$, and $H_\ast(D_{\bullet}(X,Y))\cong H_{\ast}(X,Y;\Z)$. The isomorphism in \autoref{THM: space pair group pair} is given by 
$$
H_\ast(G,\mathcal{H};\Z)=\mathrm{Tor}_{\ast-1}^{G}(\Delta_{G,\mathcal{H}}^{\perp},\Z)\cong H_\ast(D_{\bullet}^\perp(\widetilde{X},\widetilde{Y})\otimes_{\Z[G]}\Z)\cong H_\ast(D_{\bullet}(X,Y))\cong H_{\ast}(X,Y;\Z).
$$

We use the similar notations for $(X',Y')$. 
Up to a choice of basepoints, the map $\Phi$ lifts to a cellular map $\widetilde{\Phi}:(\widetilde{X},\widetilde{Y})\to (\widetilde{X'},\widetilde{Y'})$, which induces homomorphisms between chain complexes of $\Z[G]$-modules $\widetilde{\Phi}_{\#}: D_{\bullet}(\widetilde{X},\widetilde{Y})\to D_{\bullet}(\widetilde{X},\widetilde{Y})$. 
The construction implies the following commutative diagram. 
\newsavebox{\grpspaceequ}
\begin{lrbox}{\grpspaceequ}
$
\begin{tikzcd}[column sep=tiny]
{D_{\bullet}(\widetilde{X},\widetilde{Y})} \arrow[rr] \arrow[dd, "\widetilde{\Phi}_{\#}"'] &  & {H_1(\widetilde{X},\widetilde{Y};\Z)} \arrow[rr, "\partial"] \arrow[dd, "\widetilde{\Phi}_{\ast}"'] \arrow[rd, "\cong"] &                                                             & H_0(\widetilde{Y};\Z) \arrow[rr,"\text{incl}_\ast"] \arrow[dd, "\widetilde{\Phi}_{\ast}"'{yshift=3ex}] \arrow[rd, "\cong"] &                                                                                & H_0(\widetilde{X};\Z) \arrow[dd, "\widetilde{\Phi}_{\ast}"'{yshift=3ex}] \arrow[rd, "\cong"] &                     \\
                                                                                             &  &                                                                                                                         & {\Delta_{G,\mathcal{H}}} \arrow[rr] \arrow[dd, "\phi_\ast"{yshift=-2.7ex}] &                                                                                             & {\mathop{\oplus}\limits_{i=1}^{n}\Z[G/H_i]} \arrow[rr] \arrow[dd, "\phi_\ast"{yshift=-2.7ex}] &                                                                                  & \Z \arrow[dd, "id"] \\
{D_{\bullet}(\widetilde{X'},\widetilde{Y'})} \arrow[rr]                                      &  & {H_1(\widetilde{X'},\widetilde{Y'};\Z)} \arrow[rr, "\partial"{xshift=2ex}] \arrow[rd, "\cong"]                                      &                                                             & H_0(\widetilde{Y'};\Z) \arrow[rr,"\text{incl}_\ast"{xshift=4ex}] \arrow[rd, "\cong"]                                       &                                                                                & H_0(\widetilde{X'};\Z) \arrow[rd, "\cong"]                                       &                     \\
                                                                                             &  &                                                                                                                         & {\Delta_{G',\mathcal{H}'}} \arrow[rr]                       &                                                                                             & {\mathop{\oplus}\limits_{i=1}^{n}\Z[G'/H_i']} \arrow[rr]                       &                                                                                  & \Z                 
\end{tikzcd}
$
\end{lrbox}
\begin{equation*}
{\usebox{\grpspaceequ}}
\end{equation*}

Taking homologies, this indicates the following  commutative diagram
\newsavebox{\cgsp}
\begin{lrbox}{\cgsp}
$
\begin{tikzcd}[column sep=0.1cm]
{H_\ast(G,\mathcal{H};\Z)} \arrow[d, "\phi_\ast"'] \arrow[r,symbol=\mathop{=}] & {\mathrm{Tor}_{\ast-1}^{G}(\Delta_{G,\mathcal{H}}^{\perp},\Z)} \arrow[d, "\phi_\ast"'] \arrow[r,symbol=\cong] & {H_\ast(D_{\bullet}^\perp(\widetilde{X},\widetilde{Y})\otimes_{\Z[G]}\Z)} \arrow[d, "\widetilde{\Phi}_\ast"] \arrow[r,symbol=\cong] & {H_\ast(D_{\bullet}(X,Y))} \arrow[d, "\Phi_\ast"] \arrow[r,symbol=\cong] & {H_{\ast}(X,Y;\Z)} \arrow[d, "\Phi_\ast"] \\
{H_\ast(G',\mathcal{H}';\Z)} \arrow[r,symbol=\mathop{=}]                       & {\mathrm{Tor}_{\ast-1}^{G'}(\Delta_{G',\mathcal{H}'}^{\perp},\Z)} \arrow[r,symbol=\cong]                      & {H_\ast(D_{\bullet}^\perp(\widetilde{X'},\widetilde{Y'})\otimes_{\Z[G']}\Z)} \arrow[r,symbol=\cong]                                 & {H_\ast(D_{\bullet}(X',Y'))} \arrow[r,symbol=\cong]                      & {H_{\ast}(X',Y';\Z)}                     
\end{tikzcd}
$
\end{lrbox}
\begin{equation*}
\scalebox{0.95}{\usebox{\cgsp}}
\end{equation*}
which proves the naturalness.
\end{proof}

\begin{remark}
The switch of basepoints in \autoref{THM: space pair group pair} and \autoref{PROP: space pair group pair natural} induces possible conjugations on the subgroups $H_i$ and $H_i'$. However, these do not affect the homology groups up to canonical isomorphisms similar as \autoref{PROP: conjugation invariant}.  
\end{remark}

In order to apply this construction to 3-manifolds, let us first make some preparations.

\begin{proposition}[{\cite[Proposition 6.15]{Wil19}}]\label{PROP: FPinfty pair}
Suppose $(X,Y)$ is an Eilenberg-MacLane pair for a discrete group pair $(G,\mathcal{H})$, and $X$ has finitely many cells in each dimension. Then $(G,\mathcal{H})$ is of type $\text{FP}_\infty$. 
\end{proposition}

\begin{proposition}\label{PROP: surface full top}
Let $M$ be a compact, orientable, aspherical, boundary-irreducible 3-manifold, and let $\Sigma$ be a properly embedded, incompressible and boundary-incompressible surface in $M$. Then the profinite topology on $\pi_1M$ induces the full profinite topology on $\pi_1\Sigma$. 
\end{proposition}
\begin{proof}
We divide the proof into three cases.

\textbf{Case 1:} $M$ is closed, and then $\Sigma$ is also closed. 

A sufficient condition for $\pi_1M$ to induce the full profinite topology on $\pi_1\Sigma$ is that any finite-index subgroup of $\pi_1\Sigma$ is separable in $\pi_1M$, since this will imply that any closed subset in $\pi_1\Sigma$ under its profinite topology is closed under the subspace topology of $\pi_1\Sigma$ inherited from the profinite topology of $\pi_1 M$. In other words,  for any finite cover $\widetilde{\Sigma}$ of $\Sigma$, there is a $\pi_1$-injective immersion $f:\widetilde{\Sigma}\to \Sigma \hookrightarrow M$, and it suffices to show that $f_\ast(\pi_1\widetilde{\Sigma})$ is separable in $\pi_1M$. 

The criteria for separability is established by Przytycki--Wise \cite{PW14}. \cite[Theorem 1.1]{PW14} showed that given a $\pi_1$-injective immersed closed surface $f: S\looparrowright M$, $f_\ast(\pi_1S)$ is separable in $\pi_1M$ if $f$  lifts to an embedding in a finite cover $M'$ of $M$. 
Liu \cite{Liu16} provides a  characterization for virtual embedding. In fact, for an essentially immersed closed surface $f: S\looparrowright M$, he 
defined a spirality character $s(f)\in H^1(\Phi(S);\mathbb{Q}^{\times})$, in which $\Phi(S)$ denotes the almost fiber part of $S$. \cite[Theorem 1.1]{Liu16} showed that $S$ virtually embeds if and only if $s(f):H_1(\Phi(S);\Z)\to \mathbb{Q}^{\times}$ takes value in $\{\pm 1\}$, for which Liu called that $S$ is aspiral in the almost fiber part. In our case, if we denote $\iota:\Sigma\to M$ as the embedding, then $s(\iota)$ takes value in $\{\pm 1\}$.

Moreover, \cite[Proposition 4.1]{Liu16} showed that the spirality character is natural. To be concrete, for a finite cover $\kappa:\widetilde{\Sigma}\to \Sigma$, the spirality character of the immersion $f:\widetilde{\Sigma}\xrightarrow{\kappa}\Sigma\xrightarrow{\iota} M$ satisfies $s(f)=\kappa^\ast s(\iota)$. In other words, we have $$s(f): H_1(\Phi(\widetilde{\Sigma});\Z)\xrightarrow{\;\kappa_\ast\;} H_1(\Phi(\Sigma);\Z)\xrightarrow{\;s(\iota)\;} \mathbb{Q}^\times.$$ Therefore, $s(f)$ also takes value in $\{\pm 1\}$. Consequently, $f$ lifts to an embedding in a finite cover of $M$ according to \cite[Theorem 1.1]{Liu16}, and $f_\ast(\pi_1\widetilde{\Sigma})$ is separable according to the criteria of Przytycki--Wise \cite[Theorem 1.1]{PW14}.

\textbf{Case 2:} $M$ has non-empty boundary, while $\Sigma$ is closed. 

Let $D(M)$ be the double of $M$ along its boundary, which is a closed 3-manifold. Then the retraction $D(M)\to M$ implies that $\pi_1M$ injects into $\pi_1D(M)$ through the inclusion map. Thus, $\Sigma$ is an essentially embedded surface in $D(M)$. Since $M$ is aspherical and boundary-incompressible, $D(M)$ is also aspherical. Thus, by \textbf{Case 1}, the profinite topology on $\pi_1D(M)$ induces the full profinite topology on $\pi_1\Sigma$. Note that $\pi_1\Sigma <\pi_1M< \pi_1D(M)$, so the profinite topology on $\pi_1M$ also induces the full profinite topology on $\pi_1\Sigma$, since the profinite topology on the subgroup $\pi_1M$ is always finer than the subspace topology inherited from the profinite topology of $\pi_1D(M)$. 

\textbf{Case 3:} $M$ has non-empty boundary, and $\Sigma$ also has non-empty boundary. 

Let $D(\Sigma)$ be the double of $\Sigma$ along its boundary, which embeds into $D(M)$. Since $\Sigma$ is incompressible and boundary-incompressible, the double $D(\Sigma)$ is incompressible in $D(M)$. According to \textbf{Case 1}, $\pi_1D(M)$ induces the full profinite topology on $\pi_1D(\Sigma)$. There is a retraction map $r: \pi_1D(\Sigma)\to \pi_1\Sigma$ induced by the folding, that is, $r|_{\pi_1\Sigma}=id$. For any finite-index subgroup $H<\pi_1\Sigma$, $r^{-1}(H)$ is a finite-index subgroup in $\pi_1D(\Sigma)$ satisfying $r^{-1}(H)\cap \pi_1\Sigma=H$. Thus, the profinite topology on $\pi_1D(\Sigma)$ induces the full profinite topology on $\pi_1\Sigma$. Combining these two steps, the profinite topology on $\pi_1D(M)$ induces the full profinite topology on $\pi_1\Sigma$. Then, by the same reasoning as \textbf{Case 2}, the profinite topology on $\pi_1M$ induces the full profinite topology on $\pi_1\Sigma$.
\end{proof}

\begin{corollary}\label{COR: 3mfd with surface}
Let $M$ be a compact, orientable, aspherical, boundary-irreducible 3-manifold, and let $\Sigma_1,\cdots,\Sigma_n$ be disjoint, properly embedded, incompressible and boundary-incompressible surfaces in $M$. Then there are isomorphisms
\begin{equation*}
\begin{aligned}
\tensor H_\ast(M,\mathop{\cup}\limits_{i=1}^{n} \Sigma_i;\Z) & \xrightarrow{\cong}  \tensor H_\ast(\pi_1M,\{\pi_1\Sigma_i\}_{i=1}^{n};\Z)\\  & \xrightarrow{\cong}  H_\ast(\pi_1M,\{\pi_1\Sigma_i\}_{i=1}^{n};\widehat{\Z})  \xrightarrow{\cong}  \H_\ast(\widehat{\pi_1M},\{\overline{\pi_1\Sigma_i}\}_{i=1}^{n};\widehat{\Z}).
\end{aligned}
\end{equation*}
\end{corollary}
\begin{proof}
One can equip the 3-manifold $M$ with a finite cell complex structure that realizes the surfaces $\Sigma_i$ as subcomplexes. Note that $M$ is aspherical and each $\Sigma_i$ is essential, so $\Sigma_i$ is also aspherical. Therefore, $ (M,\mathop{\cup}_{i=1}^{n} \Sigma_i)$ is an Eilenberg-MacLane pair for $(\pi_1M,\{\pi_1\Sigma_i\}_{i=1}^{n})$. The first isomorphism then follows from \autoref{THM: space pair group pair}. 

According to \cite{Cav12} and \cite{GJZ08} respectively, $\pi_1M$ and  $\pi_1\Sigma_i$ are cohomologically good (see \autoref{EX: good}). Moreover, by \autoref{PROP: surface full top}, the profinite topology on $\pi_1M$ induces the full profinite topology on each $\pi_1\Sigma_i$. Thus, according to  \autoref{PROP: pair goodness}, the discrete group pair $(\pi_1M,\{\pi_1\Sigma_i\}_{i=1}^{n})$ is cohomologically good. In addition, by \autoref{PROP: FPinfty pair}, the discrete group pair $(\pi_1M,\{\pi_1\Sigma_i\}_{i=1}^{n})$ has type $\text{FP}_\infty$. Therefore, the second and third isomorphisms follow from \autoref{PROP: good tensor hatZ}.
\end{proof}

Another example is to be used in \autoref{SEC: Dehn filling}, which we introduce as follows. 

\begin{proposition}\label{PROP: solid tori}
Let $M$ be a compact, orientable, aspherical 3-manifold. Suppose $V_1,\cdots, V_n$ are disjoint embedded solid tori in $\text{int}(M)$, so that their core curves are homotopically non-trivial in $M$. 
\begin{enumerate}[leftmargin=*]
\item The homomorphism $\pi_1V_i\to \pi_1M$ induced by inclusion is injective.
\item The profinite topology on $\pi_1M$ induces the full profinite topology on $\pi_1V_i$. In particular, $\overline{\pi_1V_i}\cong \widehat{\Z}$.
\item There are isomorphisms
\begin{align*}
\tensor H_\ast(M,\mathop{\sqcup}\limits_{i=1}^{n} V_i;{\Z}) & \xrightarrow{\cong} \tensor H_\ast(\pi_1M,\{\pi_1V_i\}_{i=1}^{n};\Z) \\ & \xrightarrow{\cong} H_\ast(\pi_1M,\{\pi_1V_i\}_{i=1}^{n};\widehat{\Z})   \xrightarrow{\cong}\H_\ast (\widehat{\pi_1M},\{\overline{\pi_1V_i}\}_{i=1}^n;\widehat{\Z}).
\end{align*}
\end{enumerate}
\end{proposition}
\begin{proof}
(1) $\pi_1V_i\cong \Z$ is generated by a representative $\alpha$ of its core curve. By assumption, the image of $\alpha$ is non-trivial in $\pi_1M$. Since $M$ is aspherical, $\pi_1M$ is in fact torsion-free. Hence the map $\pi_1V_i\to \pi_1M$ is injective. From now on, we identify $\pi_1V_i$ as its image in $\pi_1M$. 

(2) By \cite{Ham01}, any abelian subgroup in $\pi_1M$ is separable. Since $\pi_1V_i\cong \Z$, any finite-index subgroup of $\pi_1V_i$ is separable in $\pi_1M$. It then follows from the same reasoning as \autoref{PROP: surface full top} that the profinite topology on $\pi_1M$ induces the full profinite topology on $\pi_1V_i$. Thus, by \autoref{PROP: Left exact} $\overline{\pi_1V_i}\cong \widehat{\pi_1V_i}\cong  \widehat{\Z}$.

(3) One can equip $M$ with a finite cell complex structure that realizes each $V_i$ as a subcomplex. Since $M$ is aspherical, each $V_i$ is aspherical, and $\pi_1V_i\to \pi_1M$ is injective, the cell complex pair $(M, \sqcup_{i=1}^{n} V_i)$ is an Eilenberg-MacLane pair for $(\pi_1M, \{\pi_1V_i\}_{i=1}^{n})$. The first isomorphism then follows from \autoref{THM: space pair group pair}. Moreover, the discrete group pair $(\pi_1M, \{\pi_1V_i\}_{i=1}^{n})$ is of type $\text{FP}_\infty$ according to \autoref{PROP: FPinfty pair}. 
In addition, by \autoref{EX: good}, $\pi_1M$ and $\pi_1V_i\cong \Z$ are all cohomologically good. Together with (2), \autoref{PROP: pair goodness} implies that the discrete group pair $(\pi_1M, \{\pi_1V_i\}_{i=1}^{n})$ is cohomologically good. The second and third isomorphisms then follow from \autoref{PROP: good tensor hatZ}.
\end{proof}

\section{The profinite mapping degree}\label{sec:6}
\def\tensor{\widehat{\Z}\otimes}
\subsection{The closed case}
The mapping degree of a map between closed oriented manifolds is defined through its induced map on the top homology. We extend this definition into the setting of profinite groups.

Let $M$ and $N$ be closed, orientable, aspherical 3-manifolds, and suppose $f: \widehat{\pi_1M}\to \widehat{\pi_1N}$ is a continuous homomorphism. Then, $f$ induces a homormophism of $\widehat{\Z}$-modules 
$$
f_\ast:\; \widehat{\Z}\otimes H_3(M;\Z) \cong \H_3(\widehat{\pi_1M};\widehat{\Z}) \longrightarrow \H_3(\widehat{\pi_1N};\widehat{\Z}) \cong \widehat{\Z}\otimes H_3(N;\Z),
$$
where the canonical isomorphisms $\widehat{\Z}\otimes H_3(M;\Z) \cong \H_3(\widehat{\pi_1M};\widehat{\Z})$ and $  \widehat{\Z}\otimes H_3(N;\Z) \cong \H_3(\widehat{\pi_1N};\widehat{\Z})$ are given by \autoref{COR: 3-mfd homology}.
\begin{definition}\label{DEF: closed profinite mapping degree}
Suppose  $M$ and $N$ are closed, oriented, aspherical 3-manifolds. Let $[M]\in H_3(M;\Z)$ and $[N]\in H_3(N;\Z)$ be the fundamental classes representing the orientations. For a continuous homomorphism $f:\widehat{\pi_1M}\to \widehat{\pi_1N}$, 
 the {\em profinite mapping degree} of $f$ is the unique element $\deg(f)\in\widehat{\Z}$ such that
$f_\ast(1\otimes[M])=\deg(f)\otimes [N]$.
\end{definition}

In particular, when $f$ is an isomorphism, $\deg(f)\in \Zx$. 

\begin{lemma}\label{LEM: deg ord}
Let $M$ and $N$ be closed, oriented, aspherical 3-manifolds, and let $p:\pi_1M\to \pi_1N$ be a homomorphism. Then, $\widehat{p}:\widehat{\pi_1M}\to \widehat{\pi_1N}$ is a continuous homomorphism, and $\deg(\widehat{p})=\mathrm{deg}(p)$, the usual mapping degree of $p$. 
\end{lemma}  
\begin{proof}
We claim that the following diagram commutes.
\begin{equation}\label{equaldegequ1}
\begin{tikzcd}[column sep=small, row sep=0.2cm]
\tensor H_3(\pi_1M;\Z) \arrow[dd, "\cong"'] \arrow[rr, "1\otimes p_\ast"] &     & \tensor H_3(\pi_1N;\Z) \arrow[dd, "\cong"]   \\
                                                                          & \text{(a)} &                                              \\
H_3(\pi_1M;\widehat{\Z}) \arrow[dd, "\cong"',"\iota_\ast"] \arrow[rr, "p_\ast"]        &     & H_3(\pi_1N;\widehat{\Z}) \arrow[dd, "\cong","\iota_\ast"'] \\
                                                                          & \text{(b)} &                                              \\
\H_3(\widehat{\pi_1M};\widehat{\Z}) \arrow[rr, "\widehat{p}_\ast"]        &     & \H_3(\widehat{\pi_1N};\widehat{\Z})         
\end{tikzcd}
\end{equation}
In fact, the commutativity of block (a) follows from the universal coefficient theorem, and the commutativity of block (b) follows from the corresponding commutative diagram on the groups.
\begin{equation*}
\begin{tikzcd}
\pi_1M \arrow[d, "\iota"'] \arrow[r, "p"] & \pi_1N \arrow[d, "\iota"] \\
\widehat{\pi_1M} \arrow[r, "\widehat{p}"] & \widehat{\pi_1N}         
\end{tikzcd}
\end{equation*}

The top row of (\ref{equaldegequ1}) sends $1\otimes [M]$ to $1\otimes \mathrm{deg}(p)[N]=\mathrm{deg}(p)\otimes [N]$, while passing to the bottom row through the vertial isomorphisms sends $1\otimes [M]$ to $\deg(\widehat{p})\otimes [N]$. Thus, $\deg(\widehat{p})=\mathrm{deg}(p)$. 
\end{proof}

The profinite mapping degree of a profinite isomorphism is invariant when passing to finite covers.
\begin{lemma}\label{LEM: finite cover mapping degree}
Suppose $M$ and $N$ are closed, oriented, aspherical 3-manifolds, and suppose $f:\widehat{\pi_1M}\to \widehat{\pi_1N}$ is an isomorphism. Let $M'$ and $N'$ be an $f$--corresponding pair of finite covers of $M$ and $N$, and equip $M'$ and $N'$ with the orientations lifted from $M$ and $N$. Let $f'=f|_{\widehat{\pi_1M'}}:\widehat{\pi_1M'}\to \widehat{\pi_1N'}$. Then $\deg(f')=\deg(f)$.
\end{lemma}
\begin{proof}

Let $p:M'\to M$ and $q:N'\to N$ be the covering maps. Then we have the following commutative diagram.
\begin{equation*}
\begin{tikzcd}
\widehat{\pi_1M'}\arrow[r,"f'"]\arrow[d,"\widehat{p_\ast}"'] & \widehat{\pi_1N'} \arrow[d,"\widehat{q_\ast}"] \\
\widehat{\pi_1M}\arrow[r,"f"] & \widehat{\pi_1N}
\end{tikzcd}
\end{equation*}

By definition, the profinite mapping degree of a composition of two continuous homomorphisms is the product of their respective profinite mapping degrees. Thus, $\deg(f)\cdot \deg(\widehat{p})=\deg(f')\cdot \deg(\widehat{q})$. 
Suppose $[M':M]=[N':N]=n$, where $n$ is a positive integer. Then, by \autoref{LEM: deg ord}, $\deg(\widehat{p})=\mathrm{deg}(p)=n$ and $\deg(\widehat{q})=\mathrm{deg}(q)=n$. Hence, $\deg(f)=\deg(f')$ since $\widehat{\Z}$ is torsion-free. 
\end{proof}

For a profinite isomorphism between closed hyperbolic 3-manifolds, \cite[Lemma 5.3]{Liu25} established the following relation between the profinite mapping degree and the homology coefficient. Since this formula is key to our proof, we display the proof here for the convenience of the readers. 

\begin{proposition}\label{PROP: closed relation}
Suppose $M$ and $N$ are closed oriented hyperbolic 3-manifolds, and $f:\widehat{\pi_1M}\to \widehat{\pi_1N}$ is an isomorphism.  Then, $\coef(f)^3=\pm\deg(f)$. 
\end{proposition}
\begin{proof}
According to \cite[Theorem 1.1]{LS18}, there exists a finite cover $N'$ of $N$ that admits a map $g: N'\to T^3$ to the (oriented) 3-torus $T^3$ with mapping degree $\mathrm{deg}(g)=1$. In this case, we obviously have $b_1(N')>0$.  Let $M'$ be the $f$--corresponding finite cover of $M$, that is, $\widehat{\pi_1M'}=f^{-1}(\widehat{\pi_1N'})$, and denote $f': \widehat{\pi_1M'}\to \widehat{\pi_1N'}$ as the restriction of $f$. By 
\autoref{THM: Zx regular} and \autoref{LEM: free part}, on the abelianization level, $f_\ast': \tensor H_1(M';\Z)_{\mathrm{free}} \to \tensor H_1(N';\Z)_{\mathrm{free}}$ can be decomposed as $\lambda\otimes \phi$, where $\pm\lambda=\coef(f')$ and $\phi: H_1(M';\Z)_{\mathrm{free}} \to H_1(N';\Z)_{\mathrm{free}}$ is an isomorphism.

On the fundamental groups, $g_\ast: \pi_1N'\to \pi_1T^3\cong \Z^3$ factors through $H_1(N';\Z)_{\mathrm{free}}$, which we denote as $$g_\ast: \pi_1N'\twoheadrightarrow H_1(N';\Z)_{\mathrm{free}}\xrightarrow{\psi} \pi_1T^3\cong \Z^3.$$
Then we define a homomorphism
$$
h:\, \pi_1M'\twoheadrightarrow H_1(M';\Z)_{\mathrm{free}} \xrightarrow{ \;\phi\;} H_1(N';\Z)_{\mathrm{free}}\xrightarrow{\;\psi\;} \pi_1T^3\cong \Z^3.
$$  

Let $\Lambda: \widehat{\Z}^3\cong \widehat{\pi_1T^3} \to \widehat{\pi_1T^3} \cong \widehat{\Z}^3$ be a scalar multiplication by $\lambda$. Then, 
 by construction, the following diagram commutes.
\begin{equation*}
\begin{tikzcd}[column sep=tiny]
& \widehat{\pi_1M'} \arrow[rrrrr,"f'"] \arrow[d,"\widehat{h}"'] & & & & &  \widehat{\pi_1N'} \arrow[d, "\widehat{g_\ast}"] &  \\
\widehat{\Z}^3 \arrow[r,symbol=\cong] & \widehat{\pi_1T^3} \arrow[rrrrr, "\Lambda"]  & &  &  &  & \widehat{\pi_1T^3} \arrow[r,symbol=\cong] &\widehat{\Z}^3
\end{tikzcd}  
\end{equation*}
We thus have 
\begin{equation}\label{degt3equ}
\deg(f')\cdot \deg(\widehat{g_\ast})=\deg(\widehat{h})\cdot \deg(\Lambda). 
\end{equation}

Let us list the properties for each of these profinite mapping degrees. First, since $f'$ is an isomorphism, $\deg(f')\in \Zx$. Second, by \autoref{LEM: deg ord}, $\deg(\widehat{g_\ast})=\mathrm{deg}(g_\ast)=\mathrm{deg}(g)=1$. Again by \autoref{LEM: deg ord}, $\deg(\widehat{h})=\deg(h)\in \Z$. Finally, let $(u_1,u_2,u_3)$ be an ordered basis of $\pi_1T^3\cong \Z^3$ that represents the positive orientation. Then, the fundamental class $[T^3]\in H_3(T^3;\Z)$ is represented by $u_1\wedge u_2\wedge u_3$, and by \autoref{PROP: abelian}, 
$$
\Lambda(u_1\wedge u_2\wedge u_3)=\Lambda(u_1)\wedge \Lambda(u_2)\wedge \Lambda(u_3)= \lambda u_1\wedge \lambda u_2\wedge \lambda u_3= \lambda^3 (u_1\wedge u_2\wedge u_3).
$$
Thus, $\deg(\Lambda)=\lambda^3$, and in particular, $\deg(\Lambda)\in \Zx$. 

Therefore, we can apply \autoref{LEM: Zx+-1} to (\ref{degt3equ}), and derive that $\lambda^3=\pm \deg(f')$. By \autoref{COR: homology coef finite cover}, $\pm\lambda= \coef(f')=\coef(f)$; and by \autoref{LEM: finite cover mapping degree}, $\deg(f')= \deg(f)$. Consequently, $\coef(f)^3=\pm \deg(f')$. 
\end{proof}
\subsection{The cusped hyperbolic case}

Mapping degrees can also be defined in the case of manifolds with boundary. 
Let $M$ and $N$ be compact, oriented, aspherical 3-manifolds with non-empty boundary. The mapping degree of a proper map $F:M\to N$ (i.e. $F(\partial M)\subseteq \partial N$) can be defined through its induced map on the relative homology $F_\ast: H_3(M,\partial M;\Z)\to H_3(N,\partial N;\Z)$. 

However, when we replace $F$ by only a homomorphism of fundamental groups $f:\pi_1M\to \pi_1N$ or a continuous homomorphism of profinite fundamental groups $f:\widehat{\pi_1M}\to \widehat{\pi_1N}$, the homomorphism $f$ may not respect the peripheral structure, so there is not a well-defined map from $H_3(\pi_1M,\pi_1\partial M;\Z)$ or $\H_3(\widehat{\pi_1M},\overline{\pi_1\partial M};\widehat{\Z})$ to $H_3(\pi_1N,\pi_1\partial N;\Z)$ or $\H_3(\widehat{\pi_1N},\overline{\pi_1\partial N};\Z)$. Fortunately, there is no such trouble when we consider a  profinite isomorphism between cusped hyperbolic 3-manifolds, which always respects the peripheral structure according to \autoref{PROP: respect peripheral structure}.

Let $(M,N,f,\lambda,\Psi)$ be a cusped hyperbolic setting of profinite isomorphism (\autoref{conv+}). 
Then, $f$ induces an isomorphism of $\widehat{\Z}$-modules
\begin{align*}
f_\star :\; \tensor H_3(M,\partial M;\Z) & \cong \H_3(\widehat{\pi_1M},\{\overline{P_i}\}_{i=1}^{n};\widehat{\Z}) \xrightarrow[\cong]{f_\ast} \H_3(\widehat{\pi_1N},\{\overline{Q_i}^{g_i}\}_{i=1}^{n};\widehat{\Z})\\& \xrightarrow[\cong]{conj} \H_3(\widehat{\pi_1N},\{\overline{Q_i}\}_{i=1}^{n};\widehat{\Z})\cong \tensor H_3(N,\partial N;\Z).
\end{align*}

The canonical isomorphisms $\tensor H_3(M,\partial M;\Z)  \cong \H_3(\widehat{\pi_1M},\{\overline{P_i}\}_{i=1}^{n};\widehat{\Z})$ and $\tensor H_3(N,\partial N;\Z)  \cong \H_3(\widehat{\pi_1N},\{\overline{Q_i}\}_{i=1}^{n};\widehat{\Z})$ are given by \autoref{COR: 3mfd with surface}, and the isomorphism $ \H_3(\widehat{\pi_1N},\{\overline{Q_i}^{g_i}\}_{i=1}^{n};\widehat{\Z}) \cong \H_3(\widehat{\pi_1N},\{\overline{Q_i}\}_{i=1}^{n};\widehat{\Z})$ is given by \autoref{PROP: conjugation invariant}. 

\begin{definition}\label{DEF: cusped mapping degree}
Suppose $(M,N,f,\lambda,\Psi)$ is a cusped hyperbolic setting of profinite isomorphism (\autoref{conv+}), with $M$ and $N$ oriented. 
Let $[M]\in H_3(M,\partial M;\Z)\cong \Z$ and $[N]\in H_3(N,\partial N;\Z)\cong\Z$ be the fundamental classes representing the orientations. 
The {\em profinite mapping degree} of $f$ is the unique element $\deg(f)\in \Zx$ such that $f_\star(1\otimes [M])= \deg(f)\otimes [N]$. 
\end{definition}


\section{Calculation through boundary}\label{sec:7}

In this section, we prove the first relation between the profinite mapping degree and the homology coefficient for a profinite isomorphism between cusped hyperbolic 3-manifolds. The method is to pass the calculation onto the boundary through the long exact sequence established in \autoref{PROP: LES}. 

\begin{proposition}\label{PROP: square}
Let $(M,N,f,\lambda,\Psi)$ be a cusped hyperbolic setting of profinite isomorphism (\autoref{conv+}) with $M$ and $N$ oriented. Then,
\begin{enumerate}[leftmargin=*]
\item\label{7.1-1} Equip $\partial M$ and $\partial N$ with the boundary orientations inherited from $M$ and $N$. Then, the homeomorphism $\Psi:\partial M\xrightarrow{{\tiny \cong}} \partial N$ is either orientation preserving on all components, or orientation reversing on all components.
\item\label{square-2} $\deg(f)=\coef(f)^2\cdot \mathrm{sgn}(\Psi)$. Here $\mathrm{sgn}(\Psi)=+1$ (resp. $-1$) if $\Psi$ is orientation preserving (resp. orientation reversing), and $\coef(f)^2=\lambda^2$ is a well-defined element in $\Zx$.
\end{enumerate}
\end{proposition}

\newsavebox{\tempDiagram}
\begin{lrbox}{\tempDiagram}

    \(
\begin{tikzcd}[column sep=0.06cm, row sep=tiny]
{\tensor H_3(M,\partial M;\Z)} \arrow[dddd,"\cong"] \arrow[rrrr, "1\otimes \partial"]                                                      &  &            &  & \tensor H_2(\partial M;\Z) \arrow[dddd,"\cong"] \arrow[r,symbol=\cong]                                        & \mathop{\oplus}\limits_{i=1}^{n} \tensor H_2(\partial_iM;\Z) \arrow[dddddddddddd,"\cong"']                                                      & 1\otimes [\partial_iM]\arrow[l,symbol=\in]\arrow[dddddddddddd,maps to] \\
                                                                                                                       &  &            &  &                                                                                       &                                                                                                                             & \\
                                                                                                                       &  & \text{(a)} &  &                                                                                       &                                                                                                                             & \\
                                                                                                                       &  &            &  &                                                                                       &                                                                                                                             & \\
{\tensor H_3(\pi_1M,\{P_i\}_{i=1}^n;\Z)} \arrow[dddd, "\cong"] \arrow[rrrr, "1\otimes\partial"]                     &  &            &  & \mathop{\oplus}\limits_{i=1}^{n} \tensor H_2(P_i;\Z)  \arrow[dddd, "\cong"]          &                             & \\
                                                                                                                       &  &            &  &                                                                                       &                                                                                                                             & \\
                                                                                                                       &  & \text{(b)} &  &                                                                                       &                                                                                                                             & \\
                                                                                                                       &  &            &  &                                                                                       &                                                                                                                             & \\
{ H_3(\pi_1M,\{P_i\}_{i=1}^n;\widehat{\Z})} \arrow[dddd, "\cong","\iota_\ast"'] \arrow[rrrr, "\partial"]                     &  &            &  &  \mathop{\oplus}\limits_{i=1}^n H_2(P_i;\widehat{\Z}) \arrow[dddd,"\cong","\iota_\ast"']          &                                      & \\
                                                                                                                       &  &            &  &                                                                                       &                                                                                                                             & \\
                                                                                                                       &  & \text{(c)} &  &                                                                                       &                                                                                                                             & \\
                                                                                                                       &  &            &  &                                                                                       &                                                                                                                             & \\
{\H_3(\widehat{\pi_1M},\{\overline{P_i}\}_{i=1}^{n};\widehat{\Z})} \arrow[rrrr, "\partial"] \arrow[dddd, "f_\ast"',"\cong"]     &  &            &  & \mathop{\oplus}\limits_{i=1}^{n}\H_2(\overline{P_i};\widehat{\Z}) \arrow[r,symbol=\cong] \arrow[dddd, "f_\ast"',"\cong"]        & \mathop{\oplus}\limits_{i=1}^{n}\H_2(\widehat{P_i};\widehat{\Z}) \arrow[dddddddd, "((\partial f_i)_\ast)_{i=1}^{n}"] &  p_i\wedge p_i' \arrow[l, symbol=\in] \arrow[dddddddd,"(\lambda\otimes \psi_i)_\ast", maps to]\\
                                                                                                                       &  &            &  &                                                                                       &                                                                                                                             & \\
                                                                                                                       &  & \text{(d)} &  &                                                                                       &                                                                                                                             & \\
                                                                                                                       &  &            &  &                                                                                       &                                                                                                                             & \\
{\H_3(\widehat{\pi_1N},\{\overline{Q_i}^{g_i}\}_{i=1}^{n};\widehat{\Z})} \arrow[rrrr, "\partial"] \arrow[dddd, "\cong"] &  &            &  & \mathop{\oplus}\limits_{i=1}^n \H_2(\overline{Q_i}^{g_i};\widehat{\Z}) \arrow[dddd, "\cong","(C_{g_i})_\ast"']  & & \\
                                                                                                                       &  &            &  &                                                                                       &                                                                                                                             & \\
                                                                                                                       &  & \text{(e)} &  &                                                                                       &                                                                                                                             & \\
                                                                                                                       &  &            &  &                                                                                       &                                                                                                                             & \\
{\H_3(\widehat{\pi_1N},\{\overline{Q_i}\}_{i=1}^{n};\widehat{\Z})} \arrow[rrrr, "\partial"]                             &  &            &  &\mathop{\oplus}\limits_{i=1}^n \H_2( \overline{Q_i} ;\widehat{\Z}) \arrow[r,symbol=\cong]                              & \mathop{\oplus}\limits_{i=1}^{n}\H_2(\widehat{Q_i};\widehat{\Z})                                                             & \lambda^2 \sgn(\psi_i)\cdot q_i\wedge q_i' \arrow[l, symbol=\in] \arrow[dddddddddddd, maps to]\\
                                                                                                                       &  &            &  &                                                                                       &                                                                                                                             & \\
                                                                                                                       &  & \text{(f)} &  &                                                                                       &                                                                                                                             & \\
                                                                                                                       &  &            &  &                                                                                       &                                                                                                                             & \\
{ H_3(\pi_1N,\{Q_i\}_{i=1}^{n};\widehat{\Z})} \arrow[uuuu, "\iota_\ast","\cong"'] \arrow[rrrr, "\partial"]           &  &            &  & \mathop{\oplus}\limits_{i=1}^n H_2( Q_i ;\widehat{\Z}) \arrow[uuuu, "\iota_\ast","\cong"']    &    &   \\
                                                                                                                       &  &            &  &                                                                                       &                                                                                                                             & \\
                                                                                                                       &  & \text{(g)} &  &                                                                                       &                                                                                                                             & \\
                                                                                                                       &  &            &  &                                                                                       &                                                                                                                             & \\
{\tensor H_3(\pi_1N,\{Q_i\}_{i=1}^{n};\Z)} \arrow[uuuu, "\cong"'] \arrow[rrrr, "1\otimes\partial"]           &  &            &  & \mathop{\oplus}\limits_{i=1}^n\tensor H_2(Q_i;\Z) \arrow[uuuu, "\cong"']     &   &   \\
                                                                                                                       &  &            &  &                                                                                       &                                                                                                                             & \\
                                                                                                                       &  & \text{(h)} &  &                                                                                       &                                                                                                                             & \\
                                                                                                                       &  &            &  &                                                                                       &                                                                                                                             & \\
{\tensor H_3(N,\partial N;\Z)} \arrow[rrrr, "1\otimes\partial"]                                                               \arrow[uuuu,"\cong"'] &  &            &  & \tensor H_2(\partial N;\Z) \arrow[r,symbol=\cong]           \arrow[uuuu,"\cong"']                                       & \mathop{\oplus}\limits_{i=1}^{n} \tensor H_2(\partial_iN;\Z)                                                               \arrow[uuuuuuuuuuuu,"\cong"] & \lambda^2\mathrm{sgn}(\psi_i)\otimes [\partial_iN] \arrow[l,symbol=\in]
\end{tikzcd}
\)
\end{lrbox}

\begin{proof}
Let us consider \autoref{Diagram A}, which is a  commutative diagram of $\widehat{\Z}$-modules. 
\begin{diagram}[ht]
\resizebox{13.8cm}{!}{\usebox{\tempDiagram}}
\caption{}
\label{Diagram A}
\end{diagram}

We first explain its validity. The vertical isomorphisms in the left column of the blocks (a), (b), (c), (f), (g), (h) are given by \autoref{COR: 3mfd with surface}. The vertical isomorphisms in the right column of blocks (a) and (h) follow from \autoref{THM: K(G,1)}, and the isomorphisms in the right column of blocks (b), (c), (f), (g) follow from \autoref{PROP: good tensor hatZ} since $\Z^2$ is cohomologically good (\autoref{EX: good}). The commutativity of blocks (a) and (h) follow from \autoref{THM: space pair group pair}, the commutativity of blocks (b) and (g) follow from the universal coefficient theorem, and the commutativity of blocks (c) and (f) follow from \autoref{PROP: completion induce LES}. The block (d) is given by \autoref{PROP: functorial LES}, and the block (e) is given by \autoref{PROP: conjugation invariant}.

Then, by \autoref{DEF: cusped mapping degree}, the left column of \autoref{Diagram A} sends $1\otimes [M]\in \tensor H_3(M,\partial M;\Z)$ to $\deg(f) \otimes [N]\in \tensor H_3(N,\partial N;\Z)$. Let $[\partial_iM]$ and $[\partial_iN]$ be the fundamental classes of $H_2(\partial_iM;\Z)$ and $H_2(\partial_iN;\Z)$ representing the boundary orientations. Then, the topological boundary maps in the top and bottom rows send $1\otimes [M]$ to $\sum_{i=1}^{n} 1\otimes[\partial_iM]$ and $1\otimes [N]$ to $\sum_{i=1}^{n} 1\otimes[\partial_iN]$ respectively. Now, we calculate the map on the right column. Let $(p_i,p_i')$ be a free basis of $P_i\cong \Z^2$ representing the positive orientation of $\partial _iM$, and let $(q_i,q_i')$ be a free basis of $Q_i\cong \Z^2$ representing the positive orientation of $\partial _iN$. The isomorphism $\partial f_i: \widehat{P_i}\to \widehat{Q_i}$ splits as $\lambda\otimes \psi_i$, so it can be represented by a matrix $\lambda\cdot A\in \mathrm{GL}(2,\widehat{\Z})$ 
$$
\partial f_i\begin{pmatrix} p_i & p_i' \end{pmatrix} = \begin{pmatrix} q_i & q_i'\end{pmatrix} \lambda\cdot A_i
$$
where $A_i\in \mathrm{GL}(2,\Z)$ represents the isomorphism $\psi_i:P_i\to Q_i$.
In particular,
$$
\det (A_i)= \mathrm{sgn}(\psi_i)=\left\{ \begin{aligned} +1, &\quad \psi_i\text{ is orientation preserving;}\\-1, &\quad \psi_i\text{ is orientation reversing.}\end{aligned}\right. 
$$

According to \autoref{PROP: abelian}, $\H_2(\widehat{P_i};\widehat{\Z})$ is the free $\widehat{\Z}$-module generated by $p_i\wedge p_i'$ that corresponds to $1\otimes [\partial_iM]\in \tensor H_2(\partial_iM ;\Z)$; and similarly, $\H_2(\widehat{Q_i};\widehat{\Z})$ is the free $\widehat{\Z}$-module generated by $q_i\wedge q_i'$ that corresponds to $1\otimes [\partial_iN]\in \tensor H_2(\partial_iN ;\Z)$. In addition, 
$$
(\partial f_i)_\ast(p_i\wedge p_i')= \partial f_i(p_i)\wedge \partial f_i(p_i')= \det(\lambda\cdot A_i)\cdot q_i\wedge q_i'=\lambda^2\mathrm{sgn}(\psi_i)\cdot q_i\wedge q_i'.
$$
Thus, the right column sends each $1\otimes[\partial_iM]$ to $\lambda^2  \mathrm{sgn}(\psi_i)\otimes[\partial_iN]$. 

Therefore, if we start from $1\otimes [M]\in \tensor H_3(M,\partial M;\Z)$ in the upper-left corner, first mapping through the top row  and then through the right column yields
$$
1\otimes [M] \mapsto \sum_{i=1}^{n} \lambda^2 \mathrm{sgn}(\psi_i)\otimes[\partial_iN],
$$ 
while first mapping through the left column and then through the bottom row yields
$$
1\otimes [M] \mapsto \sum_{i=1}^{n} \deg(f)\otimes[\partial_iN].
$$
Thus, $\sum_{i=1}^{n} \lambda^2 \mathrm{sgn}(\psi_i)\otimes[\partial_iN]= \sum_{i=1}^{n} \deg(f)\otimes[\partial_iN]$.

Since ${\oplus}_{i=1}^{n} \tensor H_2(\partial_iN;\Z)$ is a free $\widehat{\Z}$-module with basis $([\partial_1N],\cdots,[\partial_nN])$, we derive that $\mathrm{sgn}(\psi_i)$ are equal for all $1\le i \le n$, which is denoted as $\sgn(\Psi)$, and that $\deg(f) =\lambda^2\sgn(\Psi)=\coef(f)^2\cdot \sgn(\Psi)$. This finishes the proof of this proposition.
\end{proof}

\section{Calculation by Dehn filling}\label{SEC: Dehn filling}

In this section we prove the second relation between the profinite mapping degree and the homology coefficient for a profinite isomorphism between cusped hyperbolic 3-manifolds. The method is to apply the Dehn filling construction established in \cite{Xu24b} so as to pass the calculation to a profinite isomorphism between closed  hyperbolic manifolds, and then deduce  the relation according to \autoref{PROP: closed relation}.

\subsection{Dehn filling}
\def\R{\mathbb{R}}
Let us first specify our notions for Dehn fillings. Let $M$ be a compact   3-manifold with non-empty boundary consisting of tori $\partial_1M, \cdots, \partial_n M$. The {\em slopes} on $\partial_i M$ is denoted as 
$$
\slope(\partial_iM)= \{\text{free homotopy classes of  unoriented essential simple closed curves on }\partial_iM\cong T^2\}
$$
which is equipped with discrete topology. Indeed, an element in $\slope(\partial_iM)$ can be viewed as (a $\pm$-pair of) primitive integral points in $H_1(\partial_iM;\R)\cong \R^2$. The one-point compactification of $\slope(\partial_iM)$ is denoted as $\slp(\partial_iM)=\slope(\partial_iM)\cup\{\infty\}$. In fact, $\slp(\partial_iM)$ can be viewed as a subset of $\overline{\R^2}/(x\sim -x)\cong S^2/\sim\cong S^2$. For simplicity of notation, we denote $\slope(\partial M)=\slope(\partial_1M)\times \cdots \times \slope(\partial _nM)$ and $\slp(\partial M)= \slp(\partial_1M)\times \cdots \times \slp(\partial_nM)$. 

\def\c{\mathbf{c}}

\begin{definition}
Let $M$ be a compact   3-manifold with non-empty boundary consisting of $n$ tori. For $\c=(c_1,\cdots, c_n)\in \slp(\partial M)$, {\em the Dehn filling of $M$ along $\c$}, denoted by $M_{\c}$, is a compact 3-manifold constructed from $M$ as follows. For each $1\le i\le n$, if $c_i\neq \infty$, we glue a solid torus to the boundary component $\partial_i M$ so that the meridian of the solid torus is attached to $c_i$; if $c_i=\infty$, we do nothing and skip this boundary component. 
\end{definition}

The key result in \cite{Xu24b} is that a profinite isomorphic pair of orientable cusped hyperbolic 3-manifolds carries profinite isomorphisms between their  Dehn fillings. For the application in the following subsections, let us succinctly explain this result. We first introduce a notation. 

\def\ccc{\langle \!\langle \c\rangle\!\rangle}

\begin{definition}
Let $M$ be a compact 3-manifold with non-empty boundary consisting of $n$ tori. Let $\c=(c_1, \cdots, c_n)\in \slp(\partial M)$. For each $1\le i\le n$ such that $c_i\neq \infty$, we choose a conjugacy representative $\gamma_i\in \pi_1M$ that represents the free homotopy class $c_i$. We define a subgroup $\langle \!\langle \c\rangle\!\rangle \lhd \pi_1M$ to be the normal subgroup generated by the elements $\gamma_i$ for $1\le i\le n$ such that $c_i\neq \infty$. 
\end{definition}

\begin{lemma}[{\cite[Lemma 3.5]{Xu24b}}]\label{LEM: van Kampen}
Let $M$ be a compact 3-manifold with non-empty boundary consisting of tori, and let $\c\in \slp(\partial M)$. Let $p:\pi_1M \to \pi_1M_{\c}$ be the homomorphisms induced by inclusion.
\begin{enumerate}[leftmargin=*]
\item $p$ is surjective, and $\ker(p)=\ccc$.
\item $\widehat{p} : \pi_1M \to \pi_1M_{\c} $ is also surjective, and $\ker(\widehat{p})= \overline{\ccc}$. 
\end{enumerate}
\end{lemma}
\begin{proof}
(1) is a direct application of van-Kampen's theorem. Thus, we have a short exact sequence 
$$
1\to \ccc\hookrightarrow  \pi_1M\xrightarrow{p} \pi_1M_{\c}\to 1.
$$
Then, \autoref{PROP: Completion Exact} implies the following  exact sequence
$$
 \widehat{\ccc}\tto \widehat{\pi_1M}\xrightarrow{\;\,\widehat{p}\;\,} \widehat{\pi_1M_{\c}} \tto 1.
$$
In particular, $\widehat{p}$ is surjective, and the kernel of $\widehat{p}$ is the image of $ \widehat{\ccc}$ in $\widehat{\pi_1M}$, which is exactly $\overline{\ccc}$. This finishes the proof of (2). 
\end{proof}

\begin{theorem}[{\cite[Theorem A]{Xu24b}}]\label{THM: Dehn filling}
Let $(M,N,f, \lambda,\Psi)$ be a cusped hyperbolic setting of profinite isomorphisms (\autoref{conv+}). For any $\c\in \slp(\partial M)$, we have $\Psi(\c)\in \slp(\partial N)$. Let $p:\pi_1M \to \pi_1M_{\c}$ and $q:\pi_1N\to \pi_1N_{\c}$ be the homomorphisms induced by inclusions. Then, there exists an isomorphism $\fdf{c}:\widehat{\pi_1M_{\c}}\to \widehat{\pi_1N_{\Psi(\c)}}$ such that the following diagram commutes.
\begin{equation}\label{dehnfillingequ}
\begin{tikzcd}[column sep=large]
\widehat{\pi_1M}\arrow[d,"\widehat{p}"',two heads] \arrow[r,"f","\cong"'] & \widehat{\pi_1N} \arrow[d,"\widehat{q}", two heads]\\
\widehat{\pi_1M_{\c}} \arrow[r,"\fdf{c}","\cong"'] & \widehat{\pi_1N_{\Psi({\c})}}
\end{tikzcd}
\end{equation}
\end{theorem}
\begin{proof}
By \autoref{LEM: van Kampen}, it suffices to show that $f(\ker(\widehat{p}))= \ker(\widehat{q})$, or in other words, $f(\overline{\ccc})=\overline{\langle \! \langle \Psi(c) \rangle \! \rangle }$. For $1\le i\le n$ such that $c_i\neq  \infty$, we pick a representative $\gamma_i \in P_i$ that represents $c_i$. Then, $\psi_i(\gamma_i)\in Q_i\le \pi_1N $ is a conjugacy representative that represents  $\Psi(c_i)$. Within \autoref{conv+}, we  have $f(\gamma_i)= g_i^{-1} \psi_i(\gamma_i)^\lambda g_i$. Therefore,
\begin{align*}
f(\overline{\ccc})&= f(\overline{\langle \! \langle \gamma_i : c_i\neq \infty \rangle \! \rangle}) = \overline{\langle \! \langle f(\gamma_i) : c_i\neq \infty \rangle \! \rangle} = \overline{\langle \! \langle g_i^{-1} \psi_i(\gamma_i)^\lambda g_i : c_i\neq \infty \rangle \! \rangle} \\ &=\overline{\langle \! \langle   \psi_i(\gamma_i)^\lambda   : c_i\neq \infty \rangle \! \rangle} =\overline{\langle \! \langle   \psi_i(\gamma_i)   : c_i\neq \infty \rangle \! \rangle} =\overline{\langle \! \langle \Psi(c) \rangle \! \rangle }
\end{align*}
since $\lambda\in \Zx$, which finishes the proof. 
\end{proof}

\subsection{Relating homology coefficients}
\begin{lemma}\label{LEM: pass quotient regular}
Let $(M,N,f,\lambda,\Psi)$ be a cusped hyperbolic setting of profinite isomorphisms (\autoref{conv+}). Then, for any $\c\in \slp(\partial M)$, $\fdf{c}$ is $\lambda$-regular.
\end{lemma}
\begin{proof}
Since the homomorphisms $p:\pi_1M\to \pi_1M_{\c}$ and $q:\pi_1N\to \pi_1N_{\Psi(\c)}$ are surjective by \autoref{LEM: van Kampen}, we can apply \autoref{LEM: Descend Zx regular} to the commutative diagram (\ref{dehnfillingequ}) in \autoref{THM: Dehn filling}. Since $f$ is $\lambda$-regular (in fact, $M$ and $N$ are cusped, so $b_1(M)=b_1(N)>0$), \autoref{LEM: Descend Zx regular} implies that $\fdf{c}$ is also $\lambda$-regular. 
\end{proof}

The $\lambda$-regularity of $\fdf{c}$ may not convey any information when $b_1(M_{\c})= b_1(N_{\Psi(\c)})=0$, see \autoref{RMK: betti number 0}. However, we can show that $\fdf{c}$ is in fact ``strongly $\lambda$-regular'' in the sense of the following lemma.

\begin{lemma}\label{LEM: strongly regular}
Let $(M,N,f,\lambda,\Psi)$ be a cusped hyperbolic setting of profinite isomorphisms (\autoref{conv+}), and let $\c\in \slp(\partial M)$. Then, for any $\fdf{c}$--corresponding finite covers $M_{\c}'$ and $N_{\Psi(\c)}'$ of $M_{\c}$ and $N_{\Psi(\c)}$, the isomorphism $(\fdf{c})'= \fdf{c}|_{\widehat{\pi_1M_{\c}'}}: \widehat{\pi_1M_{\c}'} \to \widehat{\pi_1 N_{\Psi(\c)}'}$  is $\lambda$-regular.  
\end{lemma}
\begin{proof}
As before, we denote $p: \pi_1M\to \pi_1M_{\c}$ and $q:\pi_1N\to \pi_1N_{\Psi(\c)}$ as the surjective homomorphisms induced by inclusions. Let $M^\star$ be the finite cover of $M$ corresponding to the finite-index subgroup $p^{-1}(\pi_1M_{\c}')\le \pi_1M$, and let $N^\star$ be the finite cover of $N$ corresponding to the finite-index subgroup $q^{-1}(\pi_1N_{\Psi(\c)}')\le \pi_1N$. Then the commutative diagram (\ref{dehnfillingequ}) implies that $f(\widehat{\pi_1M^\star})=\widehat{\pi_1N^\star}$, that is, $M^\star$ and $N^\star$ is an $f$--corresponding pair of finite covers of $M$ and $N$.  We denote $f^\star= f|_{\widehat{\pi_1M^\star}} : \widehat{\pi_1M^\star}\to \widehat{\pi_1N^\star}$. 

The boundary slope $\c\in \slp(\partial M)$ lifts to its pre-image $\c^{\star}\in \slp(\partial M^\star)$, and similarly, $\Psi(\c)$ lifts to its pre-image $\Psi(\c)^\star \in \slp(\partial N^\star)$. By  construction, $M_{\c}'=(M^\star)_{\c^\star}$ and $N_{\Psi(\c)}'= (N^\star)_{\Psi(\c)^{\star}}$, where $\c^\star$ and $\Psi(\c)^\star$ are $f^\star$--corresponding boundary slopes.  In particular, we have the following commutative diagram.
\begin{equation*}
\begin{tikzcd}[column sep=large]
                                                                                                    & \widehat{\pi_1M}  \arrow[rr, "f"{xshift=-2ex}] &                                                                               & \widehat{\pi_1N} \arrow[dd, "\widehat{q}"{yshift=-2ex}, two heads] \\
\widehat{\pi_1M^\star} \arrow[ru, hook] \arrow[dd, "\widehat{p|}"{yshift=-2ex}, two heads] \arrow[rr, "f^\star"{xshift=2ex}] &                                                                       & \widehat{\pi_1N^\star} \arrow[ru, hook] \arrow[dd, "\widehat{q|}"{yshift=-2.2ex}, two heads] &                                                       \\
                                                                                                    & \widehat{\pi_1M_{\c}}  &                                                                               & \widehat{\pi_1N_{\Psi(\c)}}                           \\
\widehat{\pi_1M_{\c}'} \arrow[ru, hook] \arrow[rr, "(f^\star)^{\mathrm{Df}(\c^\star)}=(\fdf{c})'"]  &                                                                       & \widehat{\pi_1N_{\Psi(\c)}'} \arrow[ru, hook]                                 &                           
\arrow[from=1-2, to=3-2,      "\widehat{p}"{yshift=-2ex}, two heads, crossing over]                  
\arrow[from=3-2, to=3-4,     "\fdf{c}"{xshift=-3ex},crossing over] 
\end{tikzcd}
\end{equation*}


According to \autoref{LEM: homology coefficient well defined}, $f^\star$ is also $\lambda$-regular. 
We apply            \autoref{LEM: pass quotient regular} to the backward block of the commutative diagram, and derive that   $(f^\star)^{\mathrm{Df}(\c^\star)}=(\fdf{c})'  $ is also $\lambda$-regular. 
\end{proof}

To be consistent with the definition of homology coefficients in the case of hyperbolic manifolds, we reformulate the conclusion of \autoref{LEM: strongly regular} as the following corollary. 

\begin{corollary}\label{LEM: coefficient equal}
Let $(M,N,f,\lambda,\Psi)$ be a cusped hyperbolic setting of profinite isomorphisms (\autoref{conv+}). For any $\mathbf{c}\in \slp(\partial M)$ such that $M_{\mathbf{c}}$ is hyperbolic, $\coef(\fdf{c})=\coef(f)$.
\end{corollary}
\begin{proof}
By \autoref{DEF: homology coefficient}, one can take a $\fdf{c}$--corresponding pair of finite covers of $M_{\c}$ and $N_{\Psi(\c)}$ with positive first betti number, and apply \autoref{LEM: strongly regular}. 
\end{proof}

\subsection{Relating profinite mapping degrees}

For preciseness  in defining the mapping degrees, we specify the orientation of Dehn fillings as follows. 
Suppose $M$ is a compact oriented 3-manifold  with toral boundary. For $\mathbf{c}\in \slp(\partial M)$, $M$ is viewed as a submanifold of $M_{\mathbf{c}}$, and $M_{\mathbf{c}}$ is equipped with the orientation coherent to the orientation on the submanifold $M$.

\begin{lemma}\label{LEM: degree equal}
Let $(M,N,f,\lambda,\Psi)$ be a cusped hyperbolic setting of profinite isomorphisms (\autoref{conv+}) with $M$ and $N$ oriented.  Then there exists an open neighbourhood $U$ of $(\infty,\cdots,\infty)$ in $\slp(\partial M)$ such that for any $\mathbf{c}\in U\cap \slope(\partial M)$, $M_{\c}$ and $N_{\Psi(\c)}$ are closed hyperbolic manifolds and $\deg(\fdf{c})=\deg(f)$. 
\end{lemma}

\newsavebox{\excisionlargediagram}
\begin{lrbox}{\excisionlargediagram}
\(
\begin{tikzcd}[column sep=0.15cm, row sep=small]
{\tensor H_3(M,\partial M;\Z)} \arrow[dddd, "\cong"] \arrow[rrrr, "1\otimes \text{incl}_{\ast}","\cong\text{ (excision)}"']                                        &  &            &  & {\tensor H_3(M_{\c}, \mathop{\sqcup}\limits_{i=1}^{n} V_i;\Z)} \arrow[dddd, "\cong"]                                                        &  &            &  & \tensor H_3(M_{\c};\Z) \arrow[llll, "-1\otimes \text{LES}"',"\cong"] \arrow[dddd, "\cong"]            \\
                                                                                                                                        &  &            &  &                                                                                                                                                     &  &            &  &                                                                                               \\
                                                                                                                                        &  & \text{(a)} &  &                                                                                                                                                     &  & \text{(b)} &  &                                                                                               \\
                                                                                                                                        &  &            &  &                                                                                                                                                     &  &            &  &                                                                                               \\
{\tensor H_3(\pi_1M, \{ P_i\}_{i=1}^{n};\Z)} \arrow[rrrr, "1\otimes p_\ast"] \arrow[dddd, "\cong"]                                      &  &            &  & {\tensor H_3(\pi_1M_{\c},\{\langle \alpha_i\rangle \}_{i=1}^{n};\Z)} \arrow[dddd, "\cong"]                                                          &  &            &  & \tensor H_3(\pi_1M_{\c};\Z) \arrow[llll, "1\otimes \text{LES}"'] \arrow[dddd,"\cong"]                 \\
                                                                                                                                        &  &            &  &                                                                                                                                                     &  &            &  &                                                                                               \\
                                                                                                                                        &  & \text{(c)} &  &                                                                                                                                                     &  & \text{(d)} &  &                                                                                               \\
                                                                                                                                        &  &            &  &                                                                                                                                                     &  &            &  &                                                                                               \\
{ H_3(\pi_1M, \{ P_i\}_{i=1}^{n};\widehat{\Z})} \arrow[rrrr, "p_\ast"] \arrow[dddd, "\cong"]                                      &  &            &  & { H_3(\pi_1M_{\c},\{\langle \alpha_i\rangle \}_{i=1}^{n};\widehat{\Z})} \arrow[dddd, "\cong"]                                                          &  &            &  &  H_3(\pi_1M_{\c};\widehat{\Z}) \arrow[llll, "\text{LES}"'] \arrow[dddd,"\cong"]                 \\
                                                                                                                                        &  &            &  &                                                                                                                                                     &  &            &  &                                                                                               \\
                                                                                                                                        &  & \text{(e)} &  &                                                                                                                                                     &  & \text{(f)} &  &                                                                                               \\
                                                                                                                                        &  &            &  &                                                                                                                                                     &  &            &  &                                                                                               \\
{\H_3(\widehat{\pi_1M},\{\overline{P_i}\}_{i=1}^{n};\widehat{\Z})} \arrow[rrrr, "\widehat{p}_{\ast}"] \arrow[dddd, "f_\ast"',"\cong"]            &  &            &  & {\H_3(\widehat{\pi_1 M_{\c}},\{\overline{\langle \alpha_i\rangle }\}_{i=1}^{n} ;\widehat{\Z})} \arrow[dddd, "\fdf{c}_{\ast}"',"\cong"]                       &  &            &  & \H_3(\widehat{\pi_1M_{\c}};\widehat{\Z}) \arrow[llll, "\text{LES}"'] \arrow[dddd, "\fdf{c}_\ast"',"\cong"]       \\
                                                                                                                                        &  &            &  &                                                                                                                                                     &  &            &  &                                                                                               \\
                                                                                                                                        &  & \text{(g)} &  &                                                                                                                                                     &  & \text{(h)} &  &                                                                                               \\
                                                                                                                                        &  &            &  &                                                                                                                                                     &  &            &  &                                                                                               \\
{\H_3(\widehat{\pi_1N},\{\overline{Q_i}^{g_i}\}_{i=1}^{n};\widehat{\Z})} \arrow[rrrr, "\widehat{q}_\ast"] \arrow[dddd, "\text{conj}."',"\cong"] &  &            &  & {\H_3(\widehat{\pi_1N_{\Psi(\c)}}, \{\overline{\langle \beta_i \rangle}^{\widehat{q}(g_i)}\}_{i=1}^{n};\widehat{\Z})} \arrow[dddd, "\text{conj.}"',"\cong"] &  &            &  & \H_3(\widehat{\pi_1N_{\Psi(\c)}};\widehat{\Z}) \arrow[llll, "\text{LES}"'] \arrow[dddd, "id"] \\
                                                                                                                                        &  &            &  &                                                                                                                                                     &  &            &  &                                                                                               \\
                                                                                                                                        &  & \text{(i)} &  &                                                                                                                                                     &  & \text{(j)} &  &                                                                                               \\
                                                                                                                                        &  &            &  &                                                                                                                                                     &  &            &  &                                                                                               \\
{\H_3(\widehat{\pi_1N},\{\overline{Q_i}\}_{i=1}^{n};\widehat{\Z})} \arrow[rrrr, "\widehat{q}_\ast"]                                     &  &            &  & {\H_3(\widehat{\pi_1N_{\Psi(\c)}}, \{\overline{\langle \beta_i \rangle}\}_{i=1}^{n};\widehat{\Z})}                                                  &  &            &  & \H_3(\widehat{\pi_1N_{\Psi(\c)}};\widehat{\Z}) \arrow[llll, "\text{LES}"']                    \\
                                                                                                                                        &  &            &  &                                                                                                                                                     &  &            &  &                                                                                               \\
                                                                                                                                        &  & \text{(k)} &  &                                                                                                                                                     &  & \text{(l)} &  &                                                                                               \\
                                                                                                                                        &  &            &  &                                                                                                                                                     &  &            &  &                                                                                               \\
{ H_3(\pi_1N, \{ Q_i\}_{i=1}^{n};\widehat{\Z})} \arrow[uuuu, "\cong"'] \arrow[rrrr, "q_\ast"]                                     &  &            &  & { H_3(\pi_1N_{\Psi(\c)},\{\langle \beta_i\rangle \}_{i=1}^{n};\widehat{\Z})} \arrow[uuuu, "\cong"']                                                    &  &            &  &  H_3(\pi_1N_{\Psi(\c)};\widehat{\Z}) \arrow[uuuu, "\cong"'] \arrow[llll, " \text{LES}"'] \\
                                                                                                                                        &  &            &  &                                                                                                                                                     &  &            &  &                                                                                               \\
                                                                                                                                        &  & \text{(m)} &  &                                                                                                                                                     &  & \text{(n)} &  &                                                                                               \\
                                                                                                                                        &  &            &  &                                                                                                                                                     &  &            &  &                                                                                               \\
{\tensor H_3(\pi_1N, \{ Q_i\}_{i=1}^{n};\Z)} \arrow[uuuu, "\cong"'] \arrow[rrrr, "1\otimes q_\ast"]                                     &  &            &  & {\tensor H_3(\pi_1N_{\Psi(\c)},\{\langle \beta_i\rangle \}_{i=1}^{n};\Z)} \arrow[uuuu, "\cong"']                                                    &  &            &  & \tensor H_3(\pi_1N_{\Psi(\c)};\Z) \arrow[uuuu, "\cong"'] \arrow[llll, "1\otimes \text{LES}"'] \\
                                                                                                                                        &  &            &  &                                                                                                                                                     &  &            &  &                                                                                               \\
                                                                                                                                        &  & \text{(o)} &  &                                                                                                                                                     &  & \text{(p)} &  &                                                                                               \\
                                                                                                                                        &  &            &  &                                                                                                                                                     &  &            &  &                                                                                               \\
{\tensor H_3(N,\partial N;\Z)} \arrow[uuuu, "\cong"'] \arrow[rrrr, "1\otimes \text{incl}_\ast","\cong\text{ (excision)}"']                                         &  &            &  & {\tensor H_3(N_{\Psi(\c)}, \mathop{\sqcup}\limits_{i=1}^{n} W_i;\Z)} \arrow[uuuu, "\cong"']                                                 &  &            &  & \tensor H_3(N_{\Psi(\c)};\Z) \arrow[llll, "-1\otimes \text{LES}"',"\cong"] \arrow[uuuu, "\cong"']    
\end{tikzcd}
\)
\end{lrbox}

\begin{proof}
By Thurston's hyperbolic Dehn surgery theorem \cite[Theorem 5.8.2]{Thurston}, we can choose an open  neighbourhood $U_1$ of $(\infty,\cdots,\infty)$ in $\slp(\partial M)$ such that for any $\c\in U_1$, $M_{\c}$ is hyperbolic, and the core curves of the Dehn filled solid tori are short geodesics in $M_{\c}$. In particular, the core curves of the Dehn filled solid tori are homotopically non-trivial in $M_{\c}$. Similarly, choose an  open  neighbourhood $U_2$ of $(\infty,\cdots,\infty)$ in $\slp(\partial N)$ such that $N_{\c'}$ is hyperbolic for $\c'\in U_2$ and the  core curves of the Dehn filled solid tori are homotopically non-trivial in $N_{\c'}$. Let $U=U_1\cap \Psi^{-1}(U_2)$. Then $U$ is an open neighbourhood of $(\infty,\cdots,\infty)$ in $\slp(\partial M)$.

Now we assume $\c=(c_1,\cdots,c_n)\in U\cap \slope(\partial M)$. As before, denote $p: \pi_1M\to \pi_1M_{\c}$ and $q: \pi_1N\to \pi_1N_{\Psi(\c)}$ as the surjective  homomorphisms induced by inclusion. Denote by $V_1,\cdots, V_n$ the Dehn filled solid tori in $M_{\c}$ attached to the boundary components $\partial_1M,\cdots,\partial_nM$ respectively, and denote by $W_1,\cdots, W_n$ the Dehn filled solid tori in $N_{\Psi(\c)}$ attached to the boundary components $\partial _1N,\cdots, \partial _n N$ respectively. By \autoref{PROP: solid tori}, $\pi_1V_i$ injects into $\pi_1M_{\c}$ and $\pi_1W_i$ injects into $\pi_1N_{\Psi(\c)}$. We specify the conjugacy representatives of $\pi_1V_i$ and $\pi_1W_i$ as $p(P_i)$ and $q(Q_i)$ respectively. Then, $\Z\cong \pi_1V_i=p(P_i)$ is generated by an element $\alpha_i=p(x_i)$, where $x_i\in P_i$ represents a simple closed curve $c_i^{\perp}$ on $\partial_iM$ having geometric intersection number $1$ with $c_i$. Let $y_i=\psi_i(x_i)\in Q_i$, which represents the simple closed curve $\Psi(c_i^{\perp})$ on $\partial_iN$ having  geometric intersection number $1$ with $\Psi(c_i)$. Let $\beta_i=q(y_i)$. Then $\Z\cong \pi_1W_i=q(Q_i)$ is generated by $\beta_i$. By assumption in \autoref{conv+}, $f(x_i)=g_i^{-1}\cdot y_i^\lambda \cdot g_i$, and by the commutative diagram (\ref{dehnfillingequ}), $\fdf{c}(\alpha_i)=\widehat q(g_i)^{-1} \cdot \beta_i^{\lambda}\cdot \widehat q(g_i)$. In particular, $\fdf{c}(\overline{\langle \alpha_i\rangle})=\overline{\langle \beta_i\rangle}^{\widehat q(g_i)}$.

Now we consider  \autoref{Diagram B}, which is a commutative diagram of $\widehat{\Z}$-modules. 
\begin{diagram}[ht]
\resizebox{14cm}{!}{\usebox{\excisionlargediagram}}
\caption{}
\label{Diagram B}
\end{diagram}

Let us first explain its validity. The isomorphisms in the left column of blocks (a), (c), (e), (k), (m), (o) are given by \autoref{COR: 3mfd with surface}. Since   $M_{\c}$ and $N_{\Psi(\c)}$ are hyperbolic and hence aspherical, the isomorphisms in the right column of blocks (a), (c), (e), (k), (m), (o) are given by \autoref{PROP: solid tori}, and the isomorphisms in the right column of blocks (b), (d), (f), (l), (n), (p) are given by \autoref{COR: 3-mfd homology}. The commutativity of blocks (a) and (o) follows from \autoref{PROP: space pair group pair natural} since these cell complex pairs are Eilenberg-MacLane pairs. The commutativity of blocks (b) and (p) follows from \autoref{THM: space pair group pair}, for which we remind the readers that there is a sign change. The commutativity of blocks (c), (d), (m), (n) is given by the universal coefficient theorem. The commutativity of blocks (e) and (k) follows from \autoref{PROP: completion induce natural}. The commutativity of blocks (f) and (l) follows from \autoref{PROP: completion induce LES}. 
The commutativity of block (g) follows from the commutative diagram (\ref{dehnfillingequ}) of maps between the profinite group pairs. The commutativity of block (h) follows from \autoref{PROP: functorial LES}. The commutativity of block (i) follows from \autoref{PROP: conj natural}, and the commutativity of block (j) follows from \autoref{PROP: conjugation invariant}. The maps in the top row of block (a) and the bottom row of block (o) are isomorphisms according to the topological excision theorem \cite[Theorem 31.7]{Mun}. And the maps in the top row of block (b) and the bottom row of block (p) are isomorphisms since in the long exact sequence, $H_3(\sqcup_{i=1}^{n}D^2\times S^1;\Z)=H_2(\sqcup_{i=1}^{n}D^2\times S^1;\Z)=0$. 



By \autoref{DEF: cusped mapping degree}, the left column of \autoref{Diagram B} sends $1\otimes [M]\in \tensor H_3(M,\partial M;\Z)$ to $\deg(f) \otimes [N] \in \tensor H_3(N,\partial N;\Z)$. Recall that $M_{\c}$ and $N_{\Psi(\c)}$ are closed hyperbolic manifolds, so by \autoref{DEF: closed profinite mapping degree}, the right column of \autoref{Diagram B} sends  $1\otimes[M_{\c}]\in \tensor H_3(M_{\c};\Z) $ to $\deg(\fdf{c}) \otimes [N_{\Psi(\c)}]\in \tensor H_3(N_{\Psi(\c)};\Z)$. It is clear from the topological construction  that the topological excision and the long exact sequence (composed with a sign change) in the top and bottom rows of \autoref{Diagram B} send $1\otimes [M]$ to $-1\otimes [M_{\c}]$ and $1\otimes [N]$ to $-1\otimes [N_{\Psi(\c)}]$ respectively. To conclude, if we start from $1\otimes [M]\in \tensor H_3( M,\partial M;\Z)$ in the upper-left corner, first mapping through the top row and then through the right column yields ${-}\deg(\fdf{c})\otimes [N_{\Psi(\c)}]$, while first mapping through the left column and then through the bottom row yields ${-}\deg(f)\otimes [N_{\Psi(\c)}]$. Therefore, $\deg(f)=\deg(\fdf{c})$. 
\end{proof}

\subsection{Conclusions} Now we can summarize the various relations between the homology coefficient and the profinite mapping degree. 
\begin{proposition}\label{PROP: cube}
Suppose $(M,N,f,\lambda,\Psi)$ is a cusped hyperbolic setting of profinite isomorphism, with $M$ and $N$ oriented. Then $\coef(f)^3=\pm \deg(f)$.
\end{proposition}
\begin{proof}
Let $U$ be the open neighbourhood of $(\infty,\cdots,\infty)$ in $\slp(\partial M)$ given by \autoref{LEM: degree equal}, and let $\c\in U\cap \slope(\partial M)$. 
Then according to \autoref{LEM: coefficient equal} and \autoref{LEM: degree equal} respectively, $\coef(\fdf{c})=\coef(f)$ and $\deg(\fdf{c})=\deg(f)$. Note that $M_{\mathbf{c}}$ and $N_{\Psi(\mathbf{c})}$ are closed hyperbolic 3-manifolds, so by \autoref{PROP: closed relation}, $\coef(\fdf{c})^3=\pm \deg(\fdf{c})$. Combining these equations, we have $\coef(f)^3=\coef(\fdf{c})^3= \pm \deg(\fdf{c})=\pm \deg(f)$.  
\end{proof}

We restate \autoref{mainthm: regular} as the following corollary.

\begin{corollary}\label{COR: restate main}
Suppose $(M,N,f,\lambda,\Psi)$ is a cusped hyperbolic setting of profinite isomorphism, with $M$ and $N$ oriented. Then $\coef(f)=\pm1$ and $\deg(f)=\pm1$. In particular, $f$ is regular and peripheral regular.
\end{corollary}
\begin{proof}
By \autoref{PROP: square}~(\ref{square-2}) we have $\coef(f)^2=\pm \deg(f)$, and by \autoref{PROP: cube} we have $\coef(f)^3=\pm \deg(f)$. These two equations together imply  that $\coef(f)=\pm 1$ and $\deg(f)=\pm1$. Since $M$ and $N$ are cusped, $b_1(M)=b_1(N)>0$. Thus, by \autoref{DEF: homology coefficient} and its well-definedness, $f$ is $1$-regular, which is regular in the sense of \autoref{introdef: Zx regular}. And $f$ is peripheral regular by \autoref{THM: peripheral Zx regular}.
\end{proof}
This finishes the proof of \autoref{mainthm: regular}.

\section{Application to mixed manifolds}\label{sec:9}
\begin{definition}
A {\em mixed 3-manifold} is a compact, orientable, irreducible 3-manifold with empty or incompressible toral boundary whose JSJ-decomposition is nontrivial and contains a hyperbolic piece.
\end{definition}

This section is devoted to the proof of the following theorem.
\begin{theorem}\label{THM: Mixed peripheral}
Suppose $M$ and $N$ are mixed 3-manifolds, and each Seifert fibered piece of $M$ contains at most one component of $\partial M$. If $f:\widehat{\pi_1M}\to \widehat{\pi_1N}$ is an isomorphism that respects the peripheral structure, then $f$ is peripheral regular.
\end{theorem}

\begin{remark}
The conclusion of \autoref{THM: Mixed peripheral} is invalid if some Seifert fibered piece of $M$ contains at least two components $\partial_1^\star M$ and $\partial_2^\star M$ of $\partial M$. A counter-example can be constructed from `Dehn twists' along a fiber-parallel annulus joining  $\partial_1^\star M$ and $\partial_2^\star M$ with coefficients in $\widehat{\Z}$. Indeed, there is a submanifold $X$ of $M$ homeomorphic to $S_{0,3}\times S^1$ that contains $\partial_1^\star M$ and $\partial_2^\star M$ as two boundary components, where $S_{0,3}$ denotes $S^2$ removing 3 open disks. $\pi_1X\cong F_2\times \Z$ has the presentation $\left\langle a,b,c,h\mid abc=1,\,h\text{ central}\right\rangle$, where $\langle a,h\rangle$ and $\langle b,h\rangle$ are peripheral subgroups corresponding to $\partial_1^\star M$ and $\partial_2^\star M$, and $\langle c,h\rangle$ is the peripheral subgroup corresponding to the remaining boundary component of $X$ that belongs to the interior of $M$. For $\mu\in \widehat{\Z}$, one can define an automorphism $\phi$ on $\widehat{\pi_1X}$ such that   $\phi(a)=ah^\mu$, $\phi(b)=bh^{-\mu}$, $\phi(c)=c$ and $\phi(h)=h$. In particular, $\phi$ is identity on the third boundary component, so one can combine $\phi$ with $id$ on $\widehat{\pi_1(M\setminus X)}$ to obtain an automorphism $f$ of $\widehat{\pi_1M}$ that respects the peripheral structure. However, when $\mu\notin \Z$, $f$ is not peripheral regular at $\partial_1^\star M$ and $\partial_2^\star M$. 
\end{remark}


\subsection{Profinite detection of JSJ-decomposition}

Let $M$ be a compact, irreducible 3-manifold with empty or incompressible toral boundary. The JSJ-decomposition endows $M$ with a graph-of-space structure $(M_{\bullet},\Gamma^M)$. $\Gamma^M$ is a finite combinatorial graph called the JSJ-graph of $M$, each vertex  $v\in V(\Gamma^M)$ corresponds to a JSJ-piece $M_v$ which is either Seifert fibered or hyperbolic, and each edge  $e\in E(\Gamma^M)$ corresponds to a JSJ-torus $T_e$. The graph-of-space structure of $M$ defines a graph-of-group structure on $\pi_1M$. Moreover, it is shown by \cite[Theorem A]{WZ10} that this structure is preserved under profinite completion. Indeed, $\widehat{\pi_1M}$ is the profinite fundamental group of a finite graph of profinite groups $(\widehat{\mathcal{G}^M},\Gamma^M)$, where $\widehat{G^M_v}=\widehat{\pi_1M_v}$ for $v\in V(\Gamma_M)$, and $\widehat{G^M_e}=\widehat{\pi_1T^M_e}\cong \widehat{\Z}^2$ for $e\in E(\Gamma_M)$. We refer the readers to \cite{WZ19} for details.

The profinite detection of geometrization and JSJ-decomposition was proven  by Wilton--Zalesskii in a series of works \cite{WZ17,WZ17b,WZ19}. In \cite{Xu24a}, the author reformulated this result in the language of profinite Bass-Serre theory.

\newsavebox{\butterfly}
\begin{lrbox}{\butterfly}
\(
\begin{tikzcd}[column sep=tiny, row sep=huge]
{\widehat{\pi_1 M_{\partial_-(e)}}} \arrow[r,symbol=\supseteq] \arrow[d, "\text{conj}\circ f_{\partial_-(e)}"'] & {\overline{\pi_1\partial_iM_{\parital_-(e)}}} \arrow[d, "\partial f_{\partial_-(e)}"] &  & &  & {\widehat{\pi_1T^M_e}} \arrow[llll, "\widehat{\varphi_-}"'] \arrow[d,"f_{e}"] \arrow[rrrr, "\widehat{\varphi_+}"] &  & &  & {\overline{\pi_1\partial_kM_{\parital_+(e)}}} \arrow[d, "\partial f_{\partial_+(e)}"'] \arrow[r,symbol=\subseteq] & {\widehat{\pi_1 M_{\partial_+(e)}}} \arrow[d, "\text{conj}\circ f_{\partial_+(e)}"] \\
{\widehat{\pi_1 N_{\partial_-(\F(e))}}} \arrow[r,symbol=\supseteq]                                              &{ \overline{\pi_1\partial_jN_{\parital_-(\F(e))}}  }                                   &  & & & {\widehat{\pi_1T^N_{\F(e)}} }\arrow[llll, "\widehat{\psi_-}"] \arrow[rrrr, "\widehat{\psi_+}"']           &  &  & & {\overline{\pi_1\partial_lN_{\parital_+(\F(e))}}} \arrow[r,symbol=\subseteq]                                      & {\widehat{\pi_1 N_{\partial_+(\F(e))}}                                            }
\end{tikzcd}
\)
\end{lrbox}

\begin{theorem}[{\cite[Theorem 5.5]{Xu24a}}]\label{THM: JSJ}
Suppose $M$ and $N$ are compact, irreducible 3-manifolds with empty or incompressible toral boundary which are not closed $Sol$-manifolds. Let $f:\widehat{\pi_1M}\to \widehat{\pi_1N}$ be an isomomrphism. Then $f$ induces a {congruent isomorphism} between graphs of profinite groups
$$
f_{\bullet}\;: (\widehat{\mathcal{G}^M},\Gamma^M)\xrightarrow{\;\;\cong\;\;} (\widehat{\mathcal{G}^N},\Gamma^N)
$$
consisting of the following ingredients. 
\begin{enumerate}[leftmargin=*]
\item With $\Gamma^M$ and $\Gamma^N$ equipped with appropriate orientations, $\F:\Gamma^M\to \Gamma^N$ is an isomorphism of oriented graphs. 
\item For each $x\in \Gamma^M=V(\Gamma^M)\cup E(\Gamma^M)$, $f_x: \widehat{G^M_{x}}\to \widehat{G^N_{\F(x)}}$ is an isomorphism that fits into the following commutative diagram.
\begin{equation}\label{congruentdiag1}
\begin{tikzcd}
\widehat{G^M_x} \arrow[rr, hook] \arrow[d, "f_x"'] &                                       & \widehat{\pi_1M} \arrow[d, "f"] \\
\widehat{G^N_{\F(x)}} \arrow[r, hook]              & \widehat{\pi_1N} \arrow[r, "C_{g_x}","\cong"'] & \widehat{\pi_1N}               
\end{tikzcd}
\end{equation}
\item For each $e\in E(\Gamma^M)$ there is a coherence relation shown by the following commutative diagram
\begin{equation}\label{butterfly}
\scalebox{0.94}{\usebox{\butterfly}}
\end{equation}
where $\varphi_\pm$ and $\psi_{\pm}$ denote  the gluing maps.
\end{enumerate}
\end{theorem}  

Some useful properties within this congruent isomorphism are listed as follows. 

\begin{lemma}\label{LEM: JSJ}
Follow the setting of \autoref{THM: JSJ}.
\begin{enumerate}[leftmargin=*]
\item\label{lem1} For $v\in V(\Gamma^M)$, $M_v\cap \partial M=\varnothing$ if and only if $N_{\F(v)}\cap \partial N=\varnothing$.
\item\label{lem2} For $v\in V(\Gamma^M)$, $M_v$ is hyperbolic (resp. Seifert fibered) if and only if $N_{\F(v)}$ is hyperbolic (resp. Seifert fibered).
\item\label{lem3} Following the notation in the commutative diagram (\ref{butterfly}), $f_{\partial_-(e)}$ is peripheral $\lambda$-regular at $\partial_iM_{\partial_-(e)}$ if and only if $f_{\partial_+(e)}$ is peripheral $\lambda$-regular at $\partial_kM_{\partial_+(e)}$. Here, by $f_{\partial_-(e)}$ being peripheral $\lambda$-regular at $\partial_iM_{\partial_-(e)}$, we mean that $\partial f_{\partial_-(e)}$ is $\lambda$-regular, and likewise for $f_{\partial_+(e)}$. 
\item\label{lem4} $f$ respects the peripheral structure if and only if all the $f_v$'s ($v\in V(\Gamma^M)$) respect  the peripheral structure.
\end{enumerate}
\end{lemma}
\begin{proof}
(\ref{lem2}) is proven by \cite{WZ17,WZ17b}, see also \cite[Theorem 4.18 and 4.20]{Reid:2018}. (\ref{lem3}) is proven in \cite[Lemma 6.13]{Xu24a}, and (\ref{lem4}) is proven in \cite[Corollary 6.11]{Xu24a}.

We now prove (\ref{lem1}). For each $v\in V(\Gamma^M)$, let $\partial^{\mathcal{T}}M_v=\partial^{\mathcal{T}}_1M_v\cups \partial^{\mathcal{T}}_kM_v$ be the boundary components given by the JSJ-tori, where $k$ is the degree of $v$ in $\Gamma^M$.  Similarly, let $\partial^{\mathcal{T}}N_{\F(v)}=\partial^{\mathcal{T}}_1N_{\F(v)}\cups \partial^{\mathcal{T}}_kN_{\F(v)}$ be the boundary components given by the JSJ-tori. Then, up to a reordering of indices, the coherence relation (\ref{butterfly}) implies that 
$$f_v: (\widehat{\pi_1M_v},\{\overline{\pi_1\partial^{\mathcal{T}}_iM_v}\}_{i=1}^{k}) \tto (\widehat{\pi_1N_{\F(v)}},\{\overline{\pi_1\partial^{\mathcal{T}}_iN_{\F(v)}}^{h_i}\}_{i=1}^{k})$$
is an isomorphism of profinite group pairs, where $h_1,\cdots, h_k\in \widehat{\pi_1N_{\F(v)}}$ are some elements that induces the conjugation. 
Thus, by \autoref{PROP: conjugation invariant},
$$
\H_3(\widehat{\pi_1M_v},\{\overline{\pi_1\partial^{\mathcal{T}}_iM_v}\}_{i=1}^{k};\widehat{\Z})\cong \H_3(\widehat{\pi_1N_{\F(v)}},\{\overline{\pi_1\partial^{\mathcal{T}}_iN_{\F(v)}}\}_{i=1}^{k};\widehat{\Z}).
$$

Moreover, according to \autoref{COR: 3mfd with surface},
$$
\H_3(\widehat{\pi_1M_v},\{\overline{\pi_1\partial^{\mathcal{T}}_iM_v}\}_{i=1}^{k};\widehat{\Z}) \cong \tensor H_3(M_v,\partial^{\mathcal{T}}M_v;\Z)
$$
and
$$
\H_3(\widehat{\pi_1N_{\F(v)}},\{\overline{\pi_1\partial^{\mathcal{T}}_iN_{\F(v)}}\}_{i=1}^{k};\widehat{\Z})\cong  \tensor H_3(N_{\F(v)},\partial^{\mathcal{T}}N_{\F(v)};\Z).
$$
Since $H_3(M_v,\partial^{\mathcal{T}}M_v;\Z)$ and $H_3(N_{\F(v)},\partial^{\mathcal{T}}N_{\F(v)};\Z)$ are finitely generated, we derive that
$$
H_3(M_v,\partial^{\mathcal{T}}M_v;\Z)\cong H_3(N_{\F(v)},\partial^{\mathcal{T}}N_{\F(v)};\Z).
$$
Note that $M_v\cap \partial M=\varnothing$ if and only if $H_3(M_v,\partial^{\mathcal{T}}M_v;\Z)\cong \Z$ and $N_{\F(v)}\cap \partial N=\varnothing$ if and only if $ H_3(N_{\F(v)},\partial^{\mathcal{T}}N_{\F(v)};\Z)\cong \Z$. So this finishes the proof of (\ref{lem1}).
\end{proof}

\subsection{Profinite isomorphism between Seifert fibered spaces}

In this subsection, we shall only consider Seifert fibered spaces with non-empty boundary which are not $S^1\times D^2$, $S^1\times S^1\times I$, or the orientable $I$-bundle over Klein bottle. To simplify notation, the Seifert fibered spaces being considered are called {\em major Seifert fibered spaces}, following the notation of Wilkes \cite{Wil17,Wil18}. 

Any profinite isomorphism $f:\widehat{\pi_1M}\to \widehat{\pi_1N}$  between major Seifert fibered spaces that respects the peripheral structure  has a well-defined scale type $(\pm\lambda,\pm \mu)\in (\Zx/\{\pm1\})^2$. In fact, fixing a Seifert presentation of the fundamental group $\pi_1M$, we can find  a positively oriented $\Z$-basis $(h^M,e_i^M)$ for each peripheral subgroup $\pi_1\partial_iM$ such that $h^M$ represents the regular fiber, and $e_i^M$ represents the peripheral cross-section. We use the similar notation for $N$. We remind the readers that the regular fibers $h^M$ and $h^N$ generate infinite cyclic normal subgroups in $\pi_1M$ and $\pi_1N$, and  profinite isomorphisms between major Seifert fibered spaces always preserve the closure of the fiber subgroup according to \cite[Theorem 5.5]{Wil17}. The scale type $(\pm\lambda,\pm \mu)$ is defined such that
\begin{equation}\label{scaletypeequ}
C_{g_i}\circ  f\begin{pmatrix}  h^M & e_i^M \end{pmatrix}= \begin{pmatrix}  h^N & e_i^N \end{pmatrix} \begin{pmatrix} o(g_i)\lambda & o(g_i)\rho_i \\ 0 &  o(g_i)\mu\end{pmatrix}
\end{equation}
where $o:\widehat{\pi_1N}\to \widehat{\pi_1^{\text{orb}}\O_N}\to \{\pm1\}$ is the orientation homomorphism when the base orbifold $\O_N$ is non-orientable, and  $o(g)=1$ when the base orbifold $\O_N$ is orientable. 
Note that $\lambda$ and $\mu$ are unified when varying between boundary components. We refer the readers to \cite{Wil17,Wil18} and \cite[Section 8]{Xu24a} for details.

\begin{lemma}[{\cite[Lemma 9.3]{Xu24a}}]\label{LEM: gluing Seifert}
Follow the setting of \autoref{THM: JSJ} and the notation in diagram (\ref{butterfly}). Suppose $e\in E(\Gamma^M)$ such that both $M_{\partial_-(e)}$ and $M_{\partial_+(e)}$ are Seifert fibered, and both $f_{\partial_-(e)}$ and $f_{\partial_+(e)}$ respects the peripheral structure.
\begin{enumerate}[leftmargin=*]
\item\label{sl1} If $f_{\partial_-(e)}$ has scale type $(\pm\lambda, \pm\mu)$, then $f_{\partial_+(e)}$ has scale type $(\pm\mu, \pm\lambda)$.
\item\label{sl2} If $f_{\partial_-(e)}$ has scale type $(\pm\lambda, \pm\lambda)$, then $f_{\partial_-(e)}$ is peripheral $\lambda$-regular at $\partial_iM_{\partial_-(e)}$ and $f_{\partial_+(e)}$ is peripheral $\lambda$-regular at $\partial_kM_{\partial_+(e)}$.
\end{enumerate}
\end{lemma}

\begin{lemma}\label{LEM: Seifert last one}
Let $M$ and $N$ be Seifert fibered spaces each containing at least two boundary components. Label the boundary components of $M$ as $\partial_1M,\cdots, \partial _n M$ ($n\ge 2$). Suppose $f:\widehat{\pi_1M}\to \widehat{\pi_1N}$ is an isomorphism that respects the peripheral structure, and $f$ is peripheral $\lambda$-regular at $\partial_1M,\cdots,\partial_{n-1}M$. 
Then, $M$ is homeomorphic to $N$, and $f$ is also peripheral $\lambda$-regular at $\partial_nM$.
\end{lemma}
\begin{proof}
The fact that $M$ is homeomorphic to $N$ follows from \cite[Proposition 8.19]{Xu24a}. As a result, one can choose the same Seifert presentation for $\pi_1M$ and $\pi_1N$, which defines the $\Z$-bases $(h^M,e_i^M)$ and $(h^N,e_i^N)$ for the peripheral subgroups in $\pi_1M$ and $\pi_1N$. 
In this case, the scale type of $f$ is $(\pm\lambda,\pm\lambda)$. By a careful choice of the orientations, we can assume that
$$
C_{g_i}\circ  f\begin{pmatrix}  h^M & e_i^M \end{pmatrix}= \begin{pmatrix}  h^N & e_i^N \end{pmatrix} \begin{pmatrix} o(g_i)\lambda & o(g_i)\rho_i \\ 0 &  o(g_i)\lambda\end{pmatrix}
$$

Since $f$ is peripheral $\lambda$-regular at $\partial_1M,\cdots,\partial_{n-1}M$, we have $\rho_1,\cdots, \rho_{n-1}\in \lambda\Z$. In addition, since we are starting from the same Seifert presentation, \cite[Lemma 8.15]{Xu24a} implies that 
$$
\sum_{i=1}^{n} \rho_i=0.
$$
Therefore, $\rho_n\in \lambda\Z$, and $f$ is peripheral $\lambda$-regular at $\partial_nM$. 
\end{proof}

\subsection{The mixed case}
In this subsection, we finish the proof of \autoref{THM: Mixed peripheral}. 
\begin{proof}[Proof of \autoref{THM: Mixed peripheral}]
Let $f_{\bullet}:(\widehat{\mathcal{G}^M},\Gamma^M)\to (\widehat{\mathcal{G}^N},\Gamma^N)$ be the congruent isomorphism induced by $f$ given in \autoref{THM: JSJ}. Then, by \autoref{LEM: JSJ}~(\ref{lem4}), for any $v\in V(\Gamma^K)$, $f_v$ respects the peripheral structure. 

We claim that for any $v\in V(\Gamma^K)$ such that $M_v$ is Seifert fibered, the scale type of $f_v$ is $(\pm1,\pm1)$. First, we consider any Seifert fiberd piece $M_v$ adjacent to a hyperbolic piece $M_u$. Then, according to \autoref{mainthm: regular}, $f_u$ is peripheral regular at the boundary component corresponding to their common JSJ-torus, and  \autoref{LEM: JSJ}~(\ref{lem3}) implies that $f_v$ is also peripheral regular at this corresponding boundary component. Thus, by definition in (\ref{scaletypeequ}), $f_u$ has scale type $(\pm1,\pm1)$. Next,  we consider a Seifert fiberd piece $M_v$ not adjoint to a hyperbolic piece. Then there exists a finite sequence of edges in $\Gamma^M$ with Seifert fibered endpoints that connects $v$ to a Seifert fibered vertex $w$ which is adjacent to a hyperbolic piece. We have shown that $f_w$ has scale type $(\pm1,\pm1)$, so by repeated applying \autoref{LEM: gluing Seifert}~(\ref{sl1}) along this path, we derive that $f_v$ also has scale type $(\pm1,\pm1)$. 

Now we prove the peripheral regularity. Suppose $\partial_iM$ is a boundary component of $M$ belonging to the JSJ-piece $M_v$. According to the commutative diagram (\ref{congruentdiag1}), in order to show that $f$ is peripheral regular at $\partial_iM$, it suffices to show that $f_v$ is peripheral regular at $\partial_iM$. When $M_v$ is a hyperbolic piece, this follows directly from \autoref{mainthm: regular}. When $M_v$ is a Seifert fibered piece, by assumption, all other boundary components of $M_v$ are given by the JSJ-tori. Since $f_v$ has scale type $(\pm1,\pm1)$, \autoref{LEM: gluing Seifert}~(\ref{sl2}) implies that $f_v$ is peripheral regular at all boundary components given by JSJ-tori. Thus, it follows from  \autoref{LEM: Seifert last one} that $f_v$ is also peripheral regular at the only remaining boundary component $\partial_iM$, which finishes the proof. 
\end{proof}


\section{The $A$-polynomial of knots}\label{sec:10}

\subsection{Definition of $A$-polynomial}\label{subsec: A-poly}
The $A$-polynomial of a knot in $S^3$   was introduced by Cooper--Culler--Gillet--Long--Shalen in \cite{CCGLS94}, and further studied in \cite{CL96,CL98}.

Suppose $K$ is a (smooth) knot in $S^3$. Let $N(K)$ denote the tubular neighbourhood of $K$, and let $X(K)=S^3-\text{int}(N(K))$ be the exterior of $K$. 
Let $m_K$ and $l_K$ be oriented simple closed curves on $\partial N(K)$ that represent the meridian and the longitude respectively, such that the ordered basis $(m_K,l_K)$  represents the boundary orientation of $\partial N(K)$ induced from the orientation of $N(K)$ as a submanifold of the oriented ambient space $S^3$. When $K$ is oriented, we further require that $l_K$ follows the orientation of $K$. We choose the basepoint $\ast$ of $X(K)$ at the intersection of $m_K$ and $l_K$ in $\partial N(K)$. By a slight abuse of notation, the {\em preferred meridian-longitude pair} of $K$ is also the pair of elements $(m_K, l_K)$ in $\pi_1(X(K),\ast)$. In particular, $m_K$ and $l_K$ commutes with each other. Whenever $K$ is a non-trivial knot, $m_K$ and $l_K$ generates a peripheral subgroup in $\pi_1(X(K))$ isomorphic to $\Z^2$. 

The $A$-polynomial of $K$ is defined through restricting the $\mathrm{SL}(2,\C)$-representation variety of $X(K)$ to its peripheral subgroup and then taking its eigenvalues. We provide a quick definition following \cite{CL96}.

\begin{definition}\label{DEF: A-polynomial}
Let $K$ be a knot in $S^3$. Define
$$
\xi(K)=\left \{ (M,L)\in \C^2\left |\;  \begin{gathered}\text{there exists a representation }\rho: \pi_1(X(K))\to \mathrm{SL}(2,\C)\\ \text{such that }\rho(m_K)=\begin{pmatrix} M & 0 \\ 0 & M^{-1}\end{pmatrix} \text{ and }\rho(l_K)=\begin{pmatrix} L & 0 \\ 0 & L^{-1}\end{pmatrix}\end{gathered} \right. \right\}.
$$
and 
$$
\xi^\ast(K)=\left \{ (M,L)\in \C^2 \left |\;  \begin{gathered}\text{there is a non-abelian representation }\rho: \pi_1(X(K))\to \mathrm{SL}(2,\C)\\ \text{such that }\rho(m_K)=\begin{pmatrix} M & 0 \\ 0 & M^{-1}\end{pmatrix} \text{ and }\rho(l_K)=\begin{pmatrix} L & 0 \\ 0 & L^{-1}\end{pmatrix}\end{gathered} \right. \right\}.
$$
The closures $\overline{\xi(K)}$ and $\overline{\xi^\ast(K)}$ in $\C^2$ (either in the complex analytic topology or in the Zariski topology) are  affine algebraic sets. The {\em $A$-polynomial of $K$}, denoted as $A_K(M,L)\in \C[M^{\pm},L^{\pm}]$, is the product of the defining equations for the 1-dimensional components in $\overline{\xi(K)}$; and the {\em enhanced $A$-polynomial of $K$}, denoted as $\widetilde{A}_K(M,L)\in  \C[M^{\pm},L^{\pm}]$, is the product of the defining equations for the 1-dimensional components in $\overline{\xi^\ast(K)}$
\end{definition}


Elements in $\xi(K)$ given by abelian representations are obviously $\{(M,1)\mid M\in \C^\times\}$, since $H_1(X(K);\Z)\cong \Z$ is generated by $[m_K]$. Thus, by definition, 
\begin{equation}\label{Arelationequ}
A_K(M,L)=\mathrm{Red}((L-1)\widetilde{A}_K(M,L)),
\end{equation}
where `$\mathrm{Red}$' means reducing repeated factors. 
This is to say that $A_K$ and $\widetilde{A}_K$ do not have repeated factors, and $A_K$ always contains $L-1$ as a factor while $\widetilde{A}_K$ may or may not contain $L-1$ as a factor. 

The $A$-polynomial is defined up to multiplicative units in $\C[M^{\pm},L^{\pm}]$. It is shown in \cite{CCGLS94} that a multiplicative  constant in $\C^\times$ can be chosen so that $A_K(M,L)$ has integral coefficients. We say that $A_K(M,L)$ or $\widetilde{A}_K(M,L)$ is {\em normalized} if it belongs to $\Z[M,L]$, it is not divisible by $M$ or $L$, and its coefficients have no common divisors. Then, for normalized $A$-polynomials, $A_J(M,L)\doteq A_K(M,L)$ if and only if $A_J(M,L)=\pm A_K(M,L)$.  

In \autoref{DEF: A-polynomial}, one might worry  that the preferred meridian-longitude basis is not unique for an unoriented knot, since the orientation of $m_K$ and $l_K$ can be simultaneously reversed. This is resolved by the following proposition.

\begin{proposition}[{\cite[Proposition 4.2]{CL96}}]\label{PROP: A-polynomial reverse}
Reversing the orientation of $K$ does not change $A_K$ and $\widetilde{A}_K$. However, reversing the orientation of the ambient space $S^3$ changes $A_K$; to be precise, $\widetilde{A}_{\overline{K}}(M,L)\doteq \widetilde{A}_{K}(M^{-1},L)$ where $\overline{K}$ is the mirror image of $K$,
\end{proposition}

\subsection{Profinite isomorphism between knot complements}

This subsection serves as  a preparation for the proof of  \autoref{mainthm: A-polynomial general}. We shall prove two lemmas. 

\begin{lemma}\label{THM: knot boundary}
Let $J$ and $K$ be non-trivial knots in $S^3$, and let $(m_J,l_J)$, $(m_K,l_K)$ be the preferred meridian-longitude pair of $J$ and $K$ respectively. Suppose $f: \widehat{\pi_1(X(J))}\to \widehat{\pi_1(X(K))}$ is an isomorphism which is peripheral regular. Then, there exists $g\in \widehat{\pi_1(X(K))}$ such that $$f(m_J)=g^{-1} \cdot (m_K)^{\pm 1}\cdot g\quad \text{and}\quad f(l_J)=g^{-1} \cdot (l_K)^{\pm 1}\cdot g. $$ 
\end{lemma}
\begin{proof}
The proof mainly follows from \cite[Theorem 4.5]{Xu24b}.

Let $P\le \pi_1(X(J))$ be the peripheral subgroup generated by $m_J$ and $l_J$, and let $Q\le \pi_1(X(K))$ be the peripheral subgroup generated by $m_K$ and $l_K$. Then, by assumption of peripheral regularity, there exists $g$ such that $f(\overline{P})=\overline{Q}^g$, and there exists an isomorphism $\psi:P\to Q$ such that $C_g\circ f: \widehat{P}\cong \overline{P}\to \overline{Q}\cong \widehat{Q}$ is the profinite completion of $\psi$. It suffices to show that $\psi(m_J)=\pm m_K$ and $\psi(l_J)=\pm l_K$ (where the peripheral subgroups are written in additive convention). 

Let $\Psi: \partial X(J)\to \partial X(K)$ be the unique homeomorphism up to isotopy that induces $\psi$ on the fundamental group. 
Readers may check that the proof of \autoref{THM: Dehn filling} only relies on the peripheral $\Zx$-regularity of the profinite isomorphism. Thus, by the same construction of \autoref{THM: Dehn filling}, for any $\c\in \slope(\partial X(J))$, there is an isomorphism $\widehat{\pi_1X(J)_{\c}}\cong \widehat{\pi_1 X(K)_{\Psi(\c)}}$. 

On one hand, the longitude $l_J$ is the only slope on $\partial X(J)$ such that $b_1(X(J)_{\c})>0$, and the longitude $l_K$ is the only slope on $\partial X(K)$ such that $b_1(X(K)_{\c'})>0$. Thus, by \autoref{COR: b1},  $\Psi(l_J)$ is the unoriented longitude $l_K$. 

On the other hand, $X(J)_{m_J}\cong S^3$, so $\widehat{\pi_1X(K)_{\Psi(m_J)}}\cong \widehat{\pi_1X(J)_{m_J}}$ is the trivial group. Since $\pi_1X(K)_{\Psi(m_J)}$ is residually finite by \autoref{PROP: RF},  $\pi_1X(K)_{\Psi(m_J)}$ is also the trivial group. The validity of the Poincar\'e conjecture implies that $X(K)_{\Psi(m_J)}\cong S^3$. Thus, \cite[Theorem 2]{GL89} implies that $\Psi(m_J)$ is also the unoriented meridian $m_K$.

Therefore, when equipped with orientations, $\psi(m_J)=\pm m_K$ and $\psi(l_J)=\pm l_K$, finishing the proof. 
\end{proof}

By a {\em graph knot}, we mean a satellite knot whose JSJ-pieces are all Seifert fibered. 

\begin{lemma}\label{LEM: graph knot}
Suppose $J$ and $K$ are graph knots in $S^3$. If there is an isomorphism $f:\widehat{\pi_1(X(J))}\to \widehat{\pi_1(X(K))}$ which respects the peripheral structure, then $J$ and $K$ are equivalent (including mirror image).
\end{lemma}
\begin{proof}
For simplicity of notation, let $M=X(J)$ and $N=X(K)$. Let $f_{\bullet}:(\widehat{\mathcal{G}^M},\Gamma^M)\to (\widehat{\mathcal{G}^N},\Gamma^N)$ be the congruent isomomrphism induced by $f$ given in \autoref{THM: JSJ}. Then by \autoref{LEM: JSJ}~(\ref{lem4}), each $f_v$ respects the peripheral structure. In fact, the JSJ-graphs $\Gamma^M$ and $\Gamma^N$ are rooted trees whose roots correspond to the outermost JSJ-piece, i.e. the pieces containing the boundary component of $M$ and $N$. In the proof of \cite[Theorem 4.1]{WilkesKnot}, Wilkes showed that if $v\in \Gamma^M$ is a leaf which is not the root, then $f_v$ has scale type $(\pm\lambda,\pm\lambda)$. In fact, this follows from the classification of profinite isomorphisms between graph manifolds \cite{Wil18} and the fact that the piece $M_v$ has non-zero `total slope' \cite[Lemma 4.2]{WilkesKnot}. Since $\Gamma^M$ is connected, by repeatedly applying \autoref{LEM: gluing Seifert}~(\ref{sl1}), we derive that $f_u$ has scale type $(\pm\lambda,\pm\lambda)$ for all $u\in V(\Gamma^M)$. In particular, let $u$ be the root. \autoref{LEM: gluing Seifert}~(\ref{sl2}) implies that $f_u$ is peripheral $\lambda$-regular at all its boundary components obtained from JSJ-tori, so \autoref{LEM: Seifert last one} implies that $f_u$ is also peripheral $\lambda$-regular at the remaining boundary component $\partial M$. The commutative diagram (\ref{congruentdiag1}) then implies that 
$f$ is peripheral $\lambda$-regular. Then, by \cite[Theorem 9.11]{Xu24a}, $M=X(J)$ and $N=X(K)$ are homeomorphic. It then follows from \cite{GL89} that $J$ and $K$ are equivalent.
\end{proof}


\subsection{Matching up $\SL(n,\C)$--representations}

\def\Q{\mathbb{Q}}

\newsavebox{\rep}
\begin{lrbox}{\rep}
\(
\begin{tikzcd}[row sep=large,column sep=large]
             G \arrow[d, "\iota_G"']          \arrow[rrrrd, "\rho"]         &   &                                   &  &             \\
\widehat{G} \arrow[rr, "\Phi","\text{(continuous)}"'] &                                                                   & \SL(n,\overline{\Q_p}) \arrow[rr, "\varsigma_\ast", "\cong"'] &  & {\SL(n,\C)}
\end{tikzcd}
\)
\end{lrbox}


In a series of works by Spitler and Bridson--McReynolds--Reid--Spitler \cite{BMRS20,BMRS21,Spi19}, a standard method has been established to study the interplay between linear representations and the  profinite completion. 
We shall explain how peripheral regularity plays its role in this process. Firstly, let us reintroduce the construction in  \cite{BMRS20,BMRS21,Spi19}.

\begin{lemma}\label{LEM: extend}
Let $G$ be a finitely generated group, and let $\rho: G\to \SL(n,\C)$ be a representation, where $n\in \mathbb{N}$. 
Then, there exists a prime number $p$, a field isomorphism $\varsigma:\overline{\Q_p}\to \mathbb{C}$ (which is not continuous in the classical topology), and a continuous homomorphism $\Phi:\widehat{G}\to \SL(n,\overline{\Q_p})$ such that the following diagram commutes.
\begin{equation*}
\usebox{\rep}
\end{equation*}
\end{lemma}

\begin{proof}
It is shown in \cite[Lemma 4.2]{BMRS20} (though only stated for $\SL(2,\C)$ representations, but also works for $\SL(n,\C)$-representations) that there exist some prime number $p$ and a field isomorphism $\theta: \C\to \overline{\Q_p}$ such that the representation
\begin{equation*}
\begin{tikzcd}
\theta_\ast \rho:\; G\arrow[r,"\rho"] & \SL(n,\C) \arrow[r,"\theta","\cong"'] & \SL(n,\overline{\Q_p})
\end{tikzcd}
\end{equation*}
is bounded, that is, the matrix entries in $\mathrm{im}(\theta_{\ast}\rho)$ is bounded with respect to the unique $p$-adic norm. 

Since $G$ is finitely generated, there exists a finite extension $K$ of $\Q_p$ such that $\mathrm{im}(\theta_{\ast}\rho)\subseteq \SL(n,K)$.
Since $K$ is finite over $\Q_p$, $K$ is a complete and locally compact metric space, and so is $\SL(n,K)$. Thus, the boundness of $\theta_\ast\rho$ implies that the closure $\overline{\mathrm{im}(\theta_{\ast}\rho)}$ in $\SL(n,K)$ is compact. Moreover, $K$ is totally disconnected, and so is $\SL(n,K)$. Therefore, the topological group $\overline{\mathrm{im}(\theta_{\ast}\rho)}$  is compact, Hausdorff, and totally disconnected, which is thus a profinite group according to \autoref{PROP: Top}. 

Consider the homomorphism
\begin{equation*}
\begin{tikzcd}
\psi:\;G\arrow[r,"{\theta_{\ast}\rho}"] & \mathrm{im}(\theta_\ast\rho)\arrow[r,hook] & \overline{\mathrm{im}(\theta_\ast\rho)}\subseteq \SL(2,K).
\end{tikzcd}
\end{equation*}
The universal property of profinite completion \cite[Lemma 3.2.1]{RZ10} implies that the homomorphism $\psi$ extends to a continuous homomorphism $\Psi:\widehat{G}\to \overline{\mathrm{im}(\theta_\ast\rho)}\subseteq \SL(2,K)$ such that $\Psi\circ \iota_{G}=\psi$. 
Then, 
we define a continuous homomorphism
\begin{equation*}
\begin{tikzcd}
\Phi: \;\widehat{G}\arrow[rr,"\Psi","\text{(continuous)}"'] & &   \overline{\mathrm{im}(\theta_\ast\rho)}\subseteq \SL(2,K) \arrow[r,hook] & \SL(2,\overline{\Q_p}).
\end{tikzcd}
\end{equation*}
Denote $\varsigma=\theta^{-1}:\overline{\Q_p}\to \C$, then $\varsigma_\ast\circ \Phi\circ \iota_G=\varsigma_\ast \circ \theta_\ast \rho=\rho$, which yields the commutative diagram in the statement of this lemma. 
%
\end{proof}

\autoref{LEM: extend} enables us to `match up' $\SL(n,\C)$-representations  of a profinitely isomorphic pair of finitely generated groups in a weak sense. Indeed, the construction of $\phi$ depends on the choice of $p$ and $\varsigma$; however, some properties of $\phi$ are independent with the choice,  for instance, regarding the ingredients of the discrete groups that are preserved under the profinite isomorphism. We apply this observation to a peripheral regular profinite isomorphism between bounded 3-manifolds. 


As a reminder for the definition of peripheral regularity, we reformulate \autoref{introdef: peripheral Zx regular} into the following setting. 

\begin{convention}\label{convsprs}
A {\em standard peripheral regular setting} $(\!(M,N,f)\!) $ consists of the following ingredients. 
\begin{enumerate}[leftmargin=*]
\item $M$ and $N$ are compact 3-manifolds with incompressible toral boundaries, and  $f:\widehat{\pi_1M}\to \widehat{\pi_1N}$ is a  peripheral regular isomorphism. 
\item Denote the boundary components of $M$ and $N$ as $\partial_1M,\cdots, \partial  _ k M$ and $\partial_1N,\cdots, \partial_kN$ respectively, and choose conjugacy representatives of peripheral subgroups $\pi_1\partial_iM\le \pi_1M$ and $\pi_1\partial_iN\le \pi_1N$.
\item Up to a reordering, assume $f(\overline{\pi_1\partial_iM})= \overline{\pi_1\partial_iN}^{g_i}$ for $g_i\in \widehat{\pi_1N}$.
\item\label{convsprs4} There are isomorphisms $\psi_i: \pi_1\partial_iM\to \pi_1\partial_iN$ such that $C_{g_i}\circ f|_{\overline{\pi_1\partial_iM}}: \widehat{\pi_1\partial_iM}\to \widehat{\pi_1\partial_iN}$ is the profinite completion of $\psi_i$. 
\end{enumerate} 
\end{convention}


\begin{proposition}\label{THM: rep matchup}
Let $(\!(M,N,f)\!)$ be a standard peripheral regular setting (\autoref{convsprs}), and let $n\in \mathbb{N}$. Given any representation $\theta: \pi_1N\to \SL(n,\C)$, there exists a representation $\phi:\pi_1M\to \SL(n,\C)$ such that for each $1\le i \le k$, the two representations
\begin{equation*}
\begin{gathered}
\phi_i: \pi_1\partial_iM\hookrightarrow \pi_1 M \xrightarrow{\phi} \SL(n,\C) \quad \text{and}\quad 
\psi_i^{\ast}\theta_i: \pi_1\partial_iM \xrightarrow{\psi_i}\pi_1\partial_iN \hookrightarrow \pi_1N \xrightarrow{\theta} \SL(n,\C)
\end{gathered}
\end{equation*}
are equivalent by conjugation in $\SL(n,\C)$. 
Moreover, when $\theta$ is non-abelian (resp. irreducible),  $\phi$ is also non-abelian (resp. irreducible).
\end{proposition}
\begin{proof}
According to  \autoref{LEM: extend}, there exists a field isomorphism $\varsigma: \overline{\Q_p}\to \C$ such that $\theta$ extends to a representation 
\begin{equation*}
\begin{tikzcd}
 \widehat{\pi_1N}\arrow[rr,"\Theta","\text{(continuous)}"'] & & \SL(n,\overline{\Q_p})   \arrow[r,"\varsigma_\ast"] &  \SL(n,\C). 
\end{tikzcd}
\end{equation*}
Let $\Phi=\Theta\circ f: \widehat{\pi_1M}\to \SL(n,\Q_p)$, which is continuous, and let 
\begin{equation*}
\begin{tikzcd}
\phi:\;\pi_1M \arrow[r,hook] & \widehat{\pi_1M} \arrow[rr,"\Phi=\Theta\circ f","\text{(continuous)}"'] &  & \SL(n,\overline{\Q_p}) \arrow[r,"\varsigma_\ast","\cong"']& \SL(n,\C).
\end{tikzcd}
\end{equation*}

Then, for any $1\le i\le k$ and any $\alpha \in \pi_1\partial_iM$, 
\begin{align*}
\phi_{i}(\alpha)&=\varsigma_\ast\Theta(f(\alpha))=\varsigma_\ast\Theta(g_i ^{-1}\psi_i(\alpha)g_i) &(\text{\autoref{convsprs}~(\ref{convsprs4})}) \\
&=\varsigma_\ast\Theta(g_i)  ^{-1} \cdot \theta(\psi_i(\alpha))  \cdot  \varsigma_\ast\Theta(g_i)&\\
&= \varsigma_\ast\Theta(g_i)^{-1}\cdot (\psi_i^{\ast}\theta_i) (\alpha)  \cdot \varsigma_\ast\Theta(g_i).& 
\end{align*}
Note that the element $\varsigma_\ast\Theta(g_i)\in \SL(n,\C)$ that induces the conjugation is irrelavent with $\alpha$. Therefore, ${\phi_i}$ and ${\psi_i^\ast \theta_i}$ are equivalent.

Next, we prove the two extra properties. First, we show that if the presentation $\phi$ in this construction is abelian, then $\theta$ is also abelian. 
In this case, $[\pi_1M,\pi_1M]\subseteq \ker(\phi)$; since $\Phi$ is continuous, we have $\overline{[\pi_1M,\pi_1M]}\subseteq \ker(\Phi)$. 
According to \cite[Lemma 2.8]{Xu24a}, $\overline{[\pi_1M,\pi_1M]}=\overline{[\widehat{\pi_1M},\widehat{\pi_1M}]}$. 
Since $f$ is an isomorphism of profinite groups, $f(\overline{[\widehat{\pi_1M},\widehat{\pi_1M}]})= \overline{[\widehat{\pi_1N},\widehat{\pi_1N}]}$. 
Therefore, $\overline{[\widehat{\pi_1N},\widehat{\pi_1N}]}\subseteq \ker(\Phi\circ f^{-1})=\ker(\Theta)$. Since $[\pi_1N,\pi_1N]\subseteq \overline{[\widehat{\pi_1N},\widehat{\pi_1N}]}$, we finally derive that $[\pi_1N,\pi_1N]\subseteq \ker(\theta)$, which is to say  that $\theta$ is also abelian. 

Second, we show that if the presentation $\phi$ in this construction is reducible, then $\theta$ is also reducible. In this case, $\Phi(\pi_1M)$ keeps invariant a non-trivial linear subspace  $W\subsetneq \overline{\Q_p}^n$. Since $\pi_1M$ is dense in $\widehat{\pi_1M}$ and $\Phi$ is continuous, $\Phi(\widehat{\pi_1M})$ actually keeps $W$ invariant. Note that $\Phi(\widehat{\pi_1M})=\Theta(\widehat{\pi_1N})$, so $\Theta(\widehat{\pi_1N})$ also keeps $W$ invariant. Consequently, $\theta(\pi_1N)$ keeps the non-trivial linear subspace $\varsigma(W)\subsetneq \C^n$ invariant, which is to say $\theta$ is also reducible. 
\end{proof}

For a finitely generated group $G$, let $\mathscr{Y}_n(G)$ be the set of  equivalence classes of $\SL(n,\C)$ representations of $G$. There is a standard quotient map $\mathscr{Y}_n(G)\to \mathscr{X}_n(G)$ onto the $\SL(n,\C)$--character variety. When $G=\pi_1Y$ is the fundamental group of a compact, connected manifold $Y$, we abbreviate $\mathscr{Y}_n(G)$ into $\mathscr{Y}_n(Y)$. When $Z$ is  a disconnected manifold consisting of components $Z_1,\cdots, Z_k$, we denote $\mathscr{Y}_n(Z)=\mathscr{Y}_n(Z_1)\times \cdots \times \mathscr{Y}_n(Z_k)$. Any continuous map $Z\to Y$ between compact manifolds induces a map $\mathscr Y_n(Y)\to \mathscr{Y}_n(Z)$ by pulling back. 

\begin{corollary}\label{COR: rep variety}
Let $(\!(M,N,f)\!)$ be a standard peripheral regular setting (\autoref{convsprs}). Then, the images of the restriction maps $\mathscr{Y}_n(M)\to \mathscr{Y}_n(\partial M)$ and $\mathscr{Y}_n(N)\to \mathscr{Y}_n(\partial N)$  are identical through $\Psi=(\psi_1,\cdots,\psi_k)$. 
\end{corollary}
\begin{proof}
According to \autoref{THM: rep matchup},  $ \Psi^\ast(\mathrm{im}(\mathscr{Y}_n(N)\to \mathscr{Y}_n(\partial N))) \subseteq \mathrm{im}(\mathscr{Y}_n(M)\to \mathscr{Y}_n(\partial M)) $. 
By symmetry, we consider $f^{-1}:\widehat{\pi_1N}\to \widehat{\pi_1M}$. Let $h_i=f^{-1}(g_i^{-1})\in\widehat{\pi_1M}$. Then $f^{-1}(\overline{\pi_1\partial_i N})= {\overline{\pi_1\partial_iM}}^{h_i}$, and $C_{h_i}\circ f^{-1}|_{\overline{\pi_1\partial_iN}}$ is the profinite completion of $\psi_i^{-1}$. Again by \autoref{THM: rep matchup}, $(\Psi^{-1})^\ast (\mathrm{im}(\mathscr{Y}_n(M)\to \mathscr{Y}_n(\partial M))\subseteq \mathrm{im}(\mathscr{Y}_n(N)\to \mathscr{Y}_n(\partial N))$. Combining the two sides yields the conclusion. 
\end{proof}

\autoref{COR: rep variety}, in particular, holds for cusped hyperbolic 3-manifolds according to \autoref{mainthm: regular}, in which the boundary homeomorphism $\Psi:\partial M\to \partial N$ is either orientation-preserving on all components or orientation-reversing on all components according to \autoref{PROP: square}~(\ref{7.1-1}). This proves \autoref{mainthm: character variety}. 
\newsavebox{\repnew}
\begin{lrbox}{\repnew}
\(
\begin{tikzcd}
                               & \pi_1(X(K)) \arrow[ld, hook] \arrow[rd, "\theta"'{yshift=0.1ex}] \arrow[rrrd, "\tau"] &                                   &  &             \\
\widehat{\pi_1(X(K))} \arrow[rr, "\Theta"{xshift=-1ex},"\text{(continuous)}"'{xshift=-0.5ex}] &                                                                   & \Pi \arrow[rr, "\varsigma", hook] &  & {\SL(2,\C)}
\end{tikzcd}
\)
\end{lrbox}

We can now prove the first part of \autoref{mainthm: A-polynomial general}. Note that $\Z$, the group of the unknot, is profinitely rigid among all finitely generated  residually finite groups, so we are justified to stay with non-trivial knots in the proof of \autoref{mainthm: A-polynomial general}.

\begin{theorem}\label{THM: A-polynomial peripheral}
Let $J$ and $K$ be non-trivial knots in $S^3$. Suppose   $f: \widehat{\pi_1(X(J))}\to \widehat{\pi_1(X(K))}$ is an isomorphism that respects the peripheral structure. Then, up to possibly replacing $K$ with its mirror image, the enhanced $A$-polynomials of $J$ and $K$ coincide: $\widetilde{A}_J(M,L)\doteq \widetilde{A}_K(M,L)$.
\end{theorem}
\begin{proof}
By the profinite detection of geometrization and JSJ-decomposition (\autoref{LEM: JSJ}), $J$ and $K$ falls in the same class of torus knots, graph knots, hyperbolic knots, and satellite knots containing hyperbolic JSJ-pieces. When $J$ and $K$ are torus knots or graph knots, \cite[Theorem 1.5]{BF19} for the former case and \autoref{LEM: graph knot} for the latter case imply that $J$ and $K$ are equivalent. Hence, in these two cases, the (enhanced) $A$-polynomials of $J$ and $K$ certainly coincide up to a possible mirror image. 

Now we assume that $J$ and $K$ are hyperbolic knots or satellite knots containing hyperbolic JSJ-pieces.  Then, by \autoref{mainthm: regular} for the former case, and by \autoref{THM: Mixed peripheral} for the latter case, $f$ is peripheral regular. 
We follow the standard peripheral regular setting in \autoref{convsprs} for $(\!(X(J),X(K),f)\!)$. Then, according to \autoref{THM: knot boundary}, the isomorphism $\psi: \pi_1(\partial X(J))\to \pi_1(\partial X(K))$ is defined by $\psi(m_J)=\pm m_K$ and $\psi(l_J)=\pm l_K$. Up to possibly replacing $K$ with its mirror image and possibly changing the orientation of $l_K$, we can assume that $\psi(m_J)=m_K$ and $\psi(l_J)=l_K$.

By \autoref{DEF: A-polynomial}, in order to show $\widetilde{A}_J\doteq \widetilde{A}_K$, it suffices to show that $\xi^\ast(J)=\xi^\ast(K)$. We first show that $\xi^\ast(J)\supseteq \xi^\ast(K)$. For any $(M,L)\in \xi^\ast(K)$, let $\theta: \pi_1 (X(K))\to \SL(2,\C)$ be a non-abelian representation such that
$$
\tau(m_K)=\begin{pmatrix} M & 0 \\ 0 & M^{-1}\end{pmatrix} \quad \text{and}\quad \tau(l_K)=\begin{pmatrix} L & 0 \\ 0 & L^{-1}\end{pmatrix}.
$$

According to \autoref{THM: rep matchup}, there exists a non-abelian representation $\phi:\pi_1(X(J))\to \SL(2,\C)$ such that the two representations $\phi|: \pi_1(\partial X(J)) \to \SL(2,\C)$ and $\tau|\circ \psi:\pi_1(\partial X(J))\to \pi_1(\partial X(K))\to \SL(2,\C)$ are equivalent. Thus, we can compose $\phi$ with a conjugation in $\SL(2,\C)$ to obtain a representaion $\phi':\pi_1(X(J))\to \SL(2,\C)$, such that $\phi'|: \pi_1(\partial X(J)) \to \SL(2,\C)$ is equal to $\tau|\circ \psi:\pi_1(\partial(X(J))\to \SL(2,\C)$. In particular,
$$
\phi'(m_J)=\tau(\psi(m_J))=\tau(m_K)=\begin{pmatrix} M & 0 \\ 0 & M^{-1}\end{pmatrix} \; \text{ and }\;\phi'(l_J)=\tau(\psi(l_J))=\tau(l_K)=\begin{pmatrix} L & 0 \\ 0 & L^{-1}\end{pmatrix}.
$$
Therefore, $(M,L)\in \xi^\ast(J)$.

By symmetry, replacing $f$ with $f^{-1}$ also implies $\xi^\ast (J)\subseteq\xi^\ast(K)$. Thus, $\xi^\ast(K)=\xi^\ast (J)$, which finishes the proof. 
\end{proof}
\subsection{Getting rid of the peripheral structure}

Now we prove the second part of \autoref{mainthm: A-polynomial general}. Most cases here can be restored to \autoref{THM: A-polynomial peripheral}, such as the case of hyperbolic knots and hyperbolic-type satellite knots (i.e. satellite knots whose outermost JSJ-piece is hyperbolic), as profinite isomorphisms between them always respect the peripheral structure. The most tricky case here is cabled knots, for which we utilize the cabling formula by Ni--Zhang \cite{NZ17} to determine the $A$-polynomials. 

\begin{definition}
Let $K$ be a knot in $S^3$, and let $p,q$ be a coprime pair of integers, where $\abs{q}\ge 2$. The $(p,q)$-cable of $K$ is a knot $C_{p,q}(K)$ lying on $\partial N(K)$ that represents the homotopy class $p[m_K]+q[l_K]$ on $\partial N(K)$, where $(m_K,l_K)$ is the  preferred meridian-longitude pair.
\end{definition}
Usually, the sign of $q$  determines whether the orientation of $C_{p,q}(K)$ follows or reverses that of $K$. However, the $A$-polynomial is not sensible to orientations on the knot (see \autoref{PROP: A-polynomial reverse}), so in the discussion of this article, we may always assume $q>0$.  

Now we state the cabling formula for the (enhanced) $A$-polynomial introduced by Ni--Zhang.
\begin{theorem}[{\cite[Theorem 2.8]{NZ17}}]\label{THM: cable formula}
Let $K$ be a non-trivial knot, and let $p,q$ be a pair of coprime integers such that $q\ge 2$. Let $\widetilde{A}_K(M,L)$ be the normalized enhanced $A$-polynomial. Then
$$
\widetilde{A}_{C_{p,q}(K)}(M,L)=\left\{ \begin{aligned} &\mathrm{Red}\left(F_{p,q}(M,L)\cdot \mathrm{Res}_{y}(\widetilde{A}_{K}(M^q,y),y^q-L)\right ), & \text{if }\mathrm{deg}_{L}(\widetilde{A}_K)\neq 0\\&
F_{p,q}(M,L)\cdot \widetilde{A}_K(M^q), & \text{if }\mathrm{deg}_{L}(\widetilde{A}_K)=0 \end{aligned}\right.
$$
where $\mathrm{Res}_y$ is the resultant of two polynomials in $y$, and $F_{p,q}$ is defined by
$$
F_{p,q}(x,y)=\left\{\begin{aligned} & 1+x^{2p}y, &\text{if }q=2\text{ and }p>0,\\
&x^{-2p}+y, & \text{if }q=2\text{ and }p<0,\\
&-1+x^{2pq}y^2, &\text{if }q>2\text{ and }p>0,\\
&-x^{-2pq}+y^2, & \text{if }q>2\text{ and }p<0.\end{aligned}\right. 
$$
\end{theorem}

Though \autoref{THM: cable formula} is stated for the enhanced $A$-polynomial, one can deduce from the relation (\ref{Arelationequ}) and the formula $\mathrm{Res}_y(f(y)g(y),h(y))=\mathrm{Res}_y(f(y),h(y))\mathrm{Res}_y(g(y),h(y))$ that for any non-trivial knot $K$,
$$
A_{C_{p,q}(K)}(M,L)=\mathrm{Red}\left(F_{p,q}(M,L)\cdot \mathrm{Res}_{y}({A}_{K}(M^q,y),y^q-L)\right ),
$$
where $A_{K}$ is normalized. 

In particular, $\widetilde{A}_{C_{p,q}(K)}$ only relies on $\widetilde{A}_K$ and $(p,q)$, and ${A}_{C_{p,q}(K)}$ only relies on ${A}_K$ and $(p,q)$. 

\newsavebox{\gluecableequ}
\begin{lrbox}{\gluecableequ}
\(
\begin{tikzcd}[column sep=tiny, row sep=large]
\widehat{\pi_1M_r} \arrow[d, "\text{conj}\circ f_r"'] \arrow[r,symbol=\supseteq] & \overline{\pi_1\partial_0M_r} \arrow[d, "\partial f^0_r"] \arrow[rrrrr, "\widehat{\varphi}","\cong"'] &  & & &  & \overline{\pi_1\partial M^0} \arrow[d, "\partial f^0"'] \arrow[r,symbol=\mathop{=}] & {\mathrm{Span}_{\widehat{\Z}}\{m_{c(J)},l_{c(J)}\}} \\
\widehat{\pi_1N_s} \arrow[r, symbol=\supseteq ]                                    & \overline{\pi_1\partial_0N_s} \arrow[rrrrr, "\widehat{\psi}","\cong"']                                &  &  & &  & \overline{\pi_1\partial N^0} \arrow[r,symbol=\mathop{=}]                            & {\mathrm{Span}_{\widehat{\Z}}\{m_{c(K)},l_{c(K)}\}}
\end{tikzcd}
\)
\end{lrbox}

\begin{theorem}\label{THM: A prime}
Suppose $J$ and $K$ are knots in $S^3$ and $\widehat{\pi_1(X(J))}\cong \widehat{\pi_1(X(K))}$. Assume one of $J$ and $K$ is a prime knot. Then, up to possibly replacing $K$ with its mirror image, $\widetilde{A}_{J}(M,L)\doteq \widetilde A_{K}(M,L)$.
\end{theorem}
\begin{proof}
As in \autoref{THM: A-polynomial peripheral}, we may also assume that $J$ and $K$ are non-trivial knots. 
According to \autoref{THM: JSJ} and \autoref{LEM: JSJ}~(\ref{lem2}), which does not rely on the peripheral structure, $J$ and $K$ belong  to the same class of torus knots, graph knots, hyperbolic knots, or satellite knots with hyperbolic pieces. When they are torus knots, again by \cite{BF19}, $J$ and $K$ are equivalent; and when they are graph knots, Wilkes \cite[Theorem B]{WilkesKnot} proved that $J$ and $K$ are equivalent whenever one of them is prime. Thus, in these two cases, obviously $\widetilde A_J\doteq \widetilde A_K$ up to a possible mirror image. When $J$ and $K$ are hyperbolic knots, it follows from \autoref{PROP: respect peripheral structure} that any isomorphism  $\widehat{\pi_1(X(J))}\cong \widehat{\pi_1(X(K))}$ respects the peripheral structure. Hence, the conclusion for this case follows from \autoref{THM: A-polynomial peripheral}.

In what follows, we assume that $J$ and $K$ are satellite knots with hyperbolic pieces. 
For simplicity of notation, let $M=X(J)$ and $N=X(K)$. The JSJ-graphs $\Gamma^M$ and $\Gamma^N$ are rooted trees, where the roots $r\in \Gamma^M$ and $s\in \Gamma^N$ correspond to the outermost JSJ-pieces. Let $f:\widehat{\pi_1M}\to \widehat{\pi_1N}$ be an isomorphism, and let $f_\bullet: (\widehat{\mathcal{G}^M},\Gamma^M)\xrightarrow{\cong}(\widehat{\mathcal{G}^N},\Gamma^N)$ be the congruent isomorphism given by \autoref{THM: JSJ}. Then, by \autoref{LEM: JSJ}~(\ref{lem1}), $\F(r)=s$. In particular, $M_r$ and $N_s$ are either both hyperbolic or both Seifert fibered by (\ref{lem2}) of the same lemma.  

When $M_r$ and $N_s$ are hyperbolic, it follows from \cite[Proposition 3.3]{Xu24b} that $f$ respects the peripheral structure. This can also be deduced from \autoref{LEM: JSJ}~(\ref{lem4}). Thus, in this case, the conclusion follows from \autoref{THM: A-polynomial peripheral}.

Now we assume that $M_r$ and $N_s$ are Seifert fibered. Then, according to \cite[Theorem 4.18]{Bud06}, the outermost Seifert fibered piece of a satellite knot can only be a cable  space $C_{p,q}=(S_{0,2}; p/q)$ ($\abs{q}\ge 2$) or a composite space $S_{0,n}\times S^1$ ($n\ge 3$), where $S_{0,n}$ denotes $S^2$ removing $n$ open disks. Note that when the outermost piece is a cable space, the degree of the root in the JSJ-graph is $1$, and when  the outermost piece is a composite space, the degree of the root in the JSJ-graph is strictly greater than $1$. Thus, the graph isomorphism $\F:\Gamma^M\to \Gamma^N$ implies that $M_r$ and $N_s$ are either both cable spaces or both composite spaces. When they are both composite spaces, both $J$ and $K$ are non-prime. Thus, by assumption, $M_r=C_{p,q}$ and $N_s=C_{p',q'}$ are cable spaces.

Consequently, $J$ is the $(p,q)$-cable of the companion knot $c(J)$, and $K$ is the $(p',q')$-cable of the companion knot $c(K)$. As our results are not affected by the orientation of the knots, we may assume  that $q>0$ and $q'>0$.  By assumption, $c(J)$ and $c(K)$ are hyperbolic knots or satellite knots containing hyperbolic pieces. Let $M^0=M-M_r$ which is the exterior of $c(J)$, and let $N^0=N-N_s$ which is the exterior of $c(K)$. The congruent isomorphism $f_\bullet$ restricts to a congruent isomorphism $f_\bullet^0:(\widehat{{\mathcal{G}}^{M^0}},\Gamma^M-r)\to (\widehat{{\mathcal{G}}^{N^0}},\Gamma^N-s)$. Then $f_\bullet^0$ also induces an isomorphism $f^0:\widehat{\pi_1M^0}\to \widehat{\pi_1N^0}$ according to \cite[Proposition 3.11]{Xu24a}.  For any $v\in V(\Gamma^M-r)$,  $f_v^0=f_v$ respects the peripheral structure, as is completely witnessed in the JSJ-graph. Thus, by \autoref{LEM: JSJ}~(\ref{lem4}), $f^0$ respects the peripheral structure. Then, by \autoref{THM: A-polynomial peripheral}, up to a possible mirror image, $\widetilde A_{c(J)}\doteq \widetilde A_{c(K)}$.

Let us be precise about the possible mirror image. 
Since $M^0$ and $N^0$ are either hyperbolic or mixed, by \autoref{mainthm: regular} and \autoref{THM: Mixed peripheral} respectively, $f^0$ is peripheral regular. $\partial M^0$ is endowed with  a preferred meridian-longitude basis $(m_{c(J)}, l_{c(J)})$, and $\partial N^0$ is endowed with  a preferred meridian-longitude basis $(m_{c(K)}, l_{c(K)})$. According to \autoref{THM: knot boundary}, there exists $g\in \widehat{\pi_1(N^0)}$ such that $f^0(m_{c(J)})=g^{-1}  (m_{c(K)})^{\pm 1}  g$ and $f^0(l_{c(J)})=g^{-1}  (l_{c(K)})^{\pm 1}  g $. Up to possibly  reversing the orientation of the  orientation of the ambient space $S^3$, and possibly reversing the orientation of the knot $K$ along with its companion knot $c(K)$, we may assume $f^0(m_{c(J)})=g^{-1}   m_{c(K)}  g$ and $f^0(l_{c(J)})=g^{-1}   l_{c(K)}  g $. In this case, as shown in the proof of  \autoref{THM: A-polynomial peripheral}, $\widetilde A_{c(J)}\doteq \widetilde A_{c(K)}$. 

We follow this adjusted orientation. Then the cabling parameters $(p,q)$ and $(p',q')$ can be determined as follows. For the knot $J$, the cabling parameter is $(p,q)$ if and only if the regular fiber of the cable space $C_{p,q}=M_r$ is glued to the boundary slope $\pm (pm_{c(J)}+ql_{c(J)})$ at the JSJ-torus $\partial M^0$. Similarly, the regular fiber of the cable space $C_{p',q'}=N_s$ is glued to the boundary slope $\pm (p'm_{c(K)}+q'l_{c(K)})$ at the JSJ-torus $\partial N^0$. Denote $\partial f^0 = C_g\circ f|_{\overline{\pi_1\partial M^0}}$. In fact, $\partial f^0$ does not depend on the choice of the element $g$ that induces the conjugation, according to \cite[Proposition 6.7]{Xu24a}. Then, the coherence relation (\ref{butterfly}) implies a commutative diagram
\begin{equation*}
\usebox{\gluecableequ}
\end{equation*}
where $\varphi$ and $\psi$ are gluing maps at the JSJ-tori. It is shown in \cite[Theorem 5.5]{Wil17} that $f_r$ preserves the closure of the fiber subgroup. Indeed, let $h_M\in \pi_1\partial_0 M_r\subseteq \pi_1M_r$ and $h_N\in \pi_1\partial_0 N_s\subseteq \pi_1N_s$ represent the regular fibers. If we write these peripheral subgroups in the addditive convention of $\widehat{\Z}$-modules, then $\partial f_r^0(h_M)= \lambda\cdot h_N$ for some $\lambda\in \Zx$. Note that $$\partial f^0(\widehat{\varphi}(h_M))=\partial f^0(\pm (pm_{c(J)}+ql_{c(J)}))= \pm (pm_{c(K)}+ql_{c(K)}),$$ while $$\widehat{\psi}(\partial f^0_r( h_M))= \widehat{\psi}(\lambda h_N )=\pm \lambda(p'm_{c(K)}+q'l_{c(K)}).$$ Thus, it follows from \autoref{LEM: Zx+-1} that $\lambda=\pm1$ and $(p,q)=\pm (p',q')$. As we have assumed $q,q'>0$, we finally obtain that $(p,q)=(p',q')$.

Thus, by the cabling formula \autoref{THM: cable formula}, $\widetilde A_{c(J)}\doteq \widetilde  A_{c(K)}$ and $(p,q)=(p',q')$ implies that $\widetilde A_{J}\doteq \widetilde A_{K}$. 
\end{proof}

\section*{Acknowledgements}
The author would like to thank his advisor   Yi Liu for insightful discussions and suggestions.

\bibliographystyle{amsplain}
\bibliography{main.bib}

\end{sloppypar}
\end{document}